\RequirePackage{fix-cm}
\documentclass[envcountsame,smallextended]{svjour3}       
\smartqed  
\usepackage{amsthm}
\usepackage{graphicx}
\usepackage[top=35mm, bottom=25mm, left=20mm, right=20mm]{geometry}
\usepackage[numbers,sort&compress]{natbib}

 \journalname{Computational Optimization and Applications}

\usepackage{amsmath}
\usepackage{amsfonts}
\usepackage{amssymb}
\usepackage{enumitem}
\usepackage{graphicx}
\usepackage{subcaption}
\usepackage{booktabs}
\usepackage{tikz}
\usepackage{pgfplots}
\usepackage{multicol}
\usepackage{hyperref}
\usepackage{xcolor}
\usepackage{cancel}
\usepackage{ulem}
\usepackage[active]{srcltx}

\numberwithin{equation}{section}

\usepackage{algorithm}%
\usepackage{algorithmicx}%
\usepackage{algpseudocode}%

\newcommand\R{\mathbb{R}}
\newcommand\Rinf{\overline{\mathbb{R}}}
\newcommand{\inter}[1]{\mathop{\mathbf{int}(#1)}\displaylimits}
\newcommand\dom[1]{ \bs{{\rm dom}}(#1)} 
\newcommand\Dom[1]{ \bs{{\rm Dom}}(#1)} 
\newcommand\dist{ \bs{{\rm dist}}} 
\newcommand\gf{\varphi} 
\newcommand\gh{\psi} 
\newcommand\fgam[3]{#1_{#3}^{#2}}

\newcommand\prox[3]{ \bs{{\rm prox}}_{#2#1}^{#3}}

\newcommand\ov[1]{\overline{#1}}
\newcommand\mb{\mathbf{B}}
\newcommand\bs[1]{\boldsymbol{#1}}
\newcommand\argmin[1]{\bs{\arg\min}_{#1}}
\newcommand\argmint[1]{\mathop{\bs{\arg\min}}\limits_{#1}}
\newcommand\Nz{\mathbb{N}_0}
\newcommand\fv{\widehat{\gf}}

 \numberwithin{thm}{section}

\begin{document}
\title{ItsDEAL: Inexact two-level smoothing descent algorithms for weakly convex optimization}

\author{Alireza Kabgani         \and
        Masoud Ahookhosh 
}

\institute{A. Kabgani, M. Ahookhosh \at
              Department of Mathematics, University of Antwerp, Antwerp, Belgium. \\
              \email{alireza.kabgani@uantwerp.be, masoud.ahookhosh@uantwerp.be}      
}

\date{Received: date / Accepted: date}

\maketitle

\begin{abstract}
This paper addresses nonconvex optimization problems via a two-level smoothing framework, where the high-order Moreau envelope (HOME) provides a smooth approximation of weakly convex cost functions. We study the differentiability and weak smoothness of HOME. Building on the inexact two-level smoothing optimization framework (ItsOPT), we propose an inexact two-level smoothing descent algorithm (ItsDEAL) by: (i) approximately solving the proximal subproblem to obtain an inexact first-order oracle for HOME at the lower level; and (ii) applying an inexact first-order method at the upper level. We derive computable sufficient conditions for lower-level inexactness, enabling approximate high-order proximal-points to be certified without access to the exact proximal-point or HOME value. In particular, we study a parameter-free H\"olderian inexact gradient method and an inexact descent method with Armijo line search, both exploiting the weak smoothness of HOME. We establish subsequential convergence under mild inexactness assumptions, and global and linear convergence under the Kurdyka--{\L}ojasiewicz (KL) property. 
Preliminary experiments on ridge-regularized hinge-loss classification, a high-dimensional MCP-regularized classification model, and robust matrix completion illustrate the numerical behavior of the proposed methods and the role of the HOME order.

\end{abstract}

\section{Introduction}\label{sec:intro}
Let us consider the optimization problem
\begin{equation}\label{eq:mainproblemconv2}
     \mathop{\bs\min}\limits_{x\in \mathbb {R}^n}\  \gf(x),
\end{equation}
where $\gf: \R^n \to \Rinf:=\R\cup\{+\infty\}$ is a proper, lower semicontinuous (lsc), and  \textit{weakly convex}  function (Definition~\ref{def:weakconvex}) that is possibly nonsmooth and admits at least one minimizer $x^*\in \R^n$ with the corresponding minimal value $\gf^*$.
The class of {\it weakly convex functions} appears in a variety of problems in many application domains such as machine learning, signal and image processing, statistics, optimal control, 
and beyond; see, e.g., \cite{Bohm21,Goujon2024Learning,Kungurtsev2021zeroth,Liao2024error,Montefusco2013fast,Pougkakiotis2023Zeroth,Yang2019Weakly}. 
Compared with convexity, weak convexity allows a controlled amount of negative curvature,
accommodating a broader class of functions including many interesting nonconvex functions; see, e.g., \cite{Atenas2023,Bohm21,Davis2018,Rahimi2024}. As an interesting example, a composition of a convex Lipschitz continuous function with a smooth mapping whose Jacobian is Lipschitz is a weakly convex function~\cite{Davis2018,Rahimi2024}. Consequently, the development of theoretically strong and numerically efficient algorithms for this broad class of functions is of significant practical relevance.

A major challenge in addressing the weakly convex optimization problem of the form \eqref{eq:mainproblemconv2} lies in handling the nonconvexity and the nonsmoothness of the cost function. The classical methods for dealing with such functions are {\it subgradient methods}; see, e.g., \cite{Atenas2023,Davis2018,Davis2019Stochastic,Li20Low,Rahimi2024,Yang2019Weakly}. An alternative approach 
employs {\it smoothing techniques}, which approximate the original nonsmooth cost function by a smooth surrogate and thereby enable the use of {\it first-order methods}. Such techniques have received considerable attention over the last few decades; see, e.g., \cite{Ahookhosh21,Beck12,ben2006smoothing,bertsekas2009nondifferentiable,Bot15,Bot2020,Moreau65,nesterov2005smooth,patrinos2013proximal,Stella17,Themelis18, themelis2019acceleration,themelis2020douglas}.

Among all smoothing techniques, the Moreau envelope is particularly interesting due to its continuity and differentiability properties under suitable conditions; see, e.g., \cite{Moreau65}. 
Indeed, it can generate a smooth approximation of the original cost function by employing a proximal operator. Specifically, the \textit{proximal-point operator} and the \textit{Moreau envelope} with parameter $\gamma>0$ are defined as
\begin{equation}\label{eq:prox2}
\prox{\gf}{\gamma}{} (x):=\argmint{y\in \R^n} \left(\gf(y)+\frac{1}{2\gamma}\Vert x- y\Vert^2\right),
    \quad
    \fgam{\gf}{}{\gamma}(x):=\mathop{\bs{\inf}}\limits_{y\in \R^n} \left(\gf(y)+\frac{1}{2\gamma}\Vert x- y\Vert^2\right).
\end{equation}
 Evaluating the Moreau envelope requires computing the proximal operator. Although this operator is available in closed form for various convex (e.g., \cite[Chapter~6]{beck2017first}) and weakly convex (e.g., \cite{Bohm21}) functions, its exact evaluation can be challenging for more general models. In such cases, the proximal subproblem \eqref{eq:prox2} is typically solved only approximately, and consequently the Moreau envelope and its gradient are also available only through
inexact evaluations. This naturally leads to inexact variants of proximal-point and proximal-gradient methods; see, e.g., \cite{Khanh24weak,Salzo12} and the references therein.

Recently, an {\it inexact two-level smoothing optimization framework} (ItsOPT) was introduced in \cite{Kabgani24itsopt}, which leverages the {\it high-order proximal operator} (HOPE) and its corresponding {\it high-order Moreau envelope} (HOME); 
see Definition~\ref{def:Hiorder-Moreau env} and Algorithm~\ref{alg:itsopt}. 
ItsOPT consists of two levels:
(i) at the lower-level, the high-order proximal auxiliary problems are solved inexactly to produce an inexact oracle for HOME; (ii) at the upper-level, an inexact zero-, first-, or second-order method is developed to minimize HOME. This offers a flexible framework for developing inexact two-level smoothing methods, where the numerical experiments indicate that using regularization powers different from two can lead to more effective algorithms; cf. \cite{Ahookhosh2025,Kabgani24itsopt, Kabganidiff}. 
While \cite{Kabganidiff} studies fundamental properties of HOPE and HOME in a general nonconvex setting and \cite{Kabgani24itsopt} develops the general ItsOPT framework, the present paper focuses specifically on weakly convex optimization. In particular, we exploit weak convexity to study the properties of HOME needed for the analysis and to develop concrete inexact descent methods within the ItsOPT framework.

A practical issue in this two-level setting is that lower-level inexactness conditions are naturally stated in terms of the exact HOPE point and the exact HOME value, which are generally unavailable in an implementation. We show that weak convexity also provides computable sufficient conditions for these requirements. In particular, a stationarity residual, or alternatively a valid lower bound for the proximal model, can be converted into explicit distance and value-error bounds. These certificates allow the lower-level solver to be stopped using quantities that can be evaluated numerically and are used directly in the experiments of Section~\ref{sec:numerical}.

Since this paper focuses on minimization problems, our numerical experiments are conducted on lower-bounded weakly convex functions. However, to study fundamental properties of HOME such as differentiability and weak smoothness and properties of HOPE, including boundedness, we work with the broader class of high-order prox-bounded weakly convex functions. This class includes, in particular, all convex functions and all lower-bounded weakly convex functions; see Definition~\ref{def:s-prox-bounded} and Remark~\ref{rem:boundedproxbound}.

\vspace{-5mm}
\subsection{ Contribution}\label{sec:contribution}
Our contributions can be summarized as follows:
\begin{description}
 \item[(i)] (\textbf{Further properties of high-order Moreau envelope}) We further study the fundamental and differential properties of the high-order Moreau envelope (HOME) as a tool for constructing smooth approximations of nonsmooth weakly convex functions. Unlike the traditional Moreau envelope via regularization with power two (see \eqref{eq:prox2}), HOME incorporates arbitrary-order regularization terms (see Definition~\ref{def:Hiorder-Moreau env} with $p>1$), leading to a much more flexible smoothing approach that can be adapted to the geometry of the original cost function. In particular, we establish the boundedness (see Proposition~\ref{prop:findtau:lip} and Corollary~\ref{cor:findtau:lip}), the differentiability, and the weak smoothness (see Theorems~\ref{th:dif:dif} and~\ref{th:dif:weak}, and Corollary~\ref{cor:dif:weak}). As such, leveraging the weak convexity, this ensures a weakly smooth approximation of the cost on a prescribed ball, i.e., if $x^0$ is located in this ball, one can develop first-order methods generating a sequence of iterations
 that remain inside this ball.
     
 \item[(ii)] \textbf{(Inexact two-level smoothing descent algorithm)} 
 We introduce an Inexact two-level smoothing DEscent ALgorithm (ItsDEAL; see Algorithm~\ref{alg:fram:nonmono}), which serves as a generic algorithmic strategy for developing first-order descent methods via HOME. At the lower-level, solving high-order proximity problems approximately, we provide an inexact oracle for HOME, and in the upper-level, we employ this inexact information to construct inexact descent methods. 
 We also derive computable sufficient conditions for the lower-level inexactness requirements using either a stationarity residual or a valid lower bound for the proximal model; see Proposition~\ref{prop:proxmodel:strong} and
Corollary~\ref{cor:computable:inhope}.
 In particular, we generate descent directions for HOME in Lemma~\ref{lem:disdir} via relationships between these directions and the inexact gradient. Building on the H\"{o}lderian smoothness of HOME, we propose a H\"olderian inexact gradient descent algorithm (HiGDA) in Algorithm~\ref{alg:ingrad} and its parameter-free version in Algorithm~\ref{alg:ingrad2}, which does not require knowledge of the H\"{o}lderian parameters, thereby enhancing its practicality for real-world scenarios. Additionally, we introduce an Inexact DEscent Armijo Line Search (IDEALS) in Algorithm~\ref{alg:first}, a robust strategy that dynamically adapts the step-size to ensure convergence. Besides well-definedness, the subsequential convergence of the sequence generated by ItsDEAL is investigated under some reasonable assumptions on the inexactness of the proximity operator. Moreover, the global and linear convergence are verified under the Kurdyka-{\L}ojasiewicz (KL) property. 

\item[(iii)] \textbf{(Numerical behavior of the HOME order)}
We investigate the role of the regularization order $p$ in three numerical settings with certified lower-level solves. For ridge-regularized hinge-loss classification, the relative performance of the subquadratic and quadratic orders depends on the upper-level method and the requested accuracy. We also consider a high-dimensional MCP-regularized hinge-loss model, which is weakly convex and nonconvex. Under a  short CPU budget, the subquadratic choices preserve held-out classification performance better than the quadratic order, while $p=2$ decreases the training objective more aggressively. For nonsmooth robust matrix completion, IDEALS with $p=1.75$ gives the best early-update objective and recovery values among the tested IDEALS realizations. These experiments show that the HOME order is an algorithmic parameter whose effect depends on the problem and computational regime.

\end{description}

\subsection{Organization}
The remainder of this paper is structured as follows. We provide essential preliminaries and notations in Section~\ref{sec:prelim}. In Section~\ref{sec:home}, we discuss high-order Moreau envelopes and highlight some key new properties of the HOME. Additionally, we analyze the differentiability and smoothness characteristics of HOME under weak convexity assumptions. In Section \ref{sec:itsopt}, we introduce a two-level smoothing framework and present several algorithms. 
Section~\ref{sec:numerical} presents numerical experiments on ridge- and MCP-regularized hinge-loss classification and robust matrix completion, with particular attention to the effect of the HOME order.
 Finally, Section \ref{sec:conc} delivers our concluding remarks.

\section{Preliminaries and notation}
\label{sec:prelim}

Throughout the paper, $\R^n$ denotes the $n$-dimensional \textit{Euclidean space}, and
$\Vert \cdot \Vert$ and $\langle \cdot, \cdot \rangle$ represent the \textit{Euclidean norm} and \textit{inner product}, respectively.
We set $\Nz := \mathbb{N} \cup \{0\}$. 
The \textit{open ball} centered at $\ov{x} \in \R^n$ with radius $r>0$ is denoted by $\mb(\ov{x}; r)$.
Let $\inter{C}$ denote the \textit{interior} of $C \subseteq \R^n$. 
We write $\dist(x, C) := \bs\inf_{y \in C} \Vert y - x \Vert$.

The \textit{effective domain} of $\gh: \R^n \to \Rinf := \R \cup \{+\infty\}$ is $\dom{\gh} := \{x \in \R^n \mid \gh(x) < +\infty\}$, and it is said to be proper if $\dom{\gh} \neq \emptyset$. 
We define the \textit{sublevel set} of $\gh$ at $\lambda \in \R$ as $\mathcal{L}(\gh, \lambda) := \{x \in \R^n \mid \gh(x) \leq \lambda\}$. 
The function $\gh$ is \textit{lower semicontinuous} (lsc) at $\ov{x} \in \R^n$ if 
$\bs\liminf_{x \to \ov{x}} \gh(x) \geq \gh(\ov{x})$. It is \textit{coercive} if $\bs\lim_{\Vert x \Vert \to +\infty} \gh(x) = +\infty$. For a set-valued map $T: \R^n \rightrightarrows \R^n$, its domain is $\Dom{T} := \{x \in \R^n \mid T(x) \neq \emptyset\}$. We use $\partial\gh$ for the Mordukhovich/limiting subdifferential of $\gh$;
see, e.g., \cite{Mordukhovich2018,Rockafellar09}.
For $p > 1$, $\nabla \left(\frac{1}{p} \Vert x \Vert^p \right) = \Vert x \Vert^{p-2} x$, 
where the right-hand side is understood as $0$ when $x=0$.
For a sequence $\{x^k\}_{k\in \Nz}$, 
$\widehat{x}$ is a \textit{cluster point} if there exists an infinite subset $J\subseteq \Nz$ such that the subsequence $\{x^{j}\}_{j\in J}$ satisfies $x^j\to \widehat{x}$.

\begin{lemma}[Basic inequality]\label{lem:findlowbounknu:lemma}
Let $r>0$ and $p\in[1,2]$. For $p=1$, we use the convention $\|x\|^{-1}x=0$ when $x=0$. Then, for any $a,b\in\R^n$ satisfying
$\Vert a\Vert\leq r$ and $\Vert b\Vert\leq r$,
\[
\left\langle \Vert a\Vert^{p-2}a-\Vert b\Vert^{p-2}b, a-b\right\rangle \geq (p-1)r^{p-2}\Vert a-b\Vert^2.
\]
\end{lemma}
\begin{proof}
See Appendix~\ref{proof:lem:findlowbounknu:lemma}
\end{proof}
For $a, b\geq 0$  and $q \in (0,1]$, the following inequality holds
\begin{equation}\label{eq:intrp:p01}
(a+b)^q\leq a^q+b^q.
\end{equation}

A differentiable function $\gh: \R^n \to\Rinf$  has a \textit{$\nu$-H\"{o}lder continuous gradient} on $C\subseteq \inter{\dom{\gh}}$ with $\nu\in (0, 1]$, if 
there exists a constant $L_\nu\geq 0$ such that
\[
\Vert \nabla \gh(y)- \nabla \gh(x)\Vert \leq L_\nu \Vert y-x\Vert^\nu, \qquad \forall x, y\in C.
\]
Such a function is called \textit{weakly smooth}, and we write $\gh\in\mathcal{C}^{1,\nu}_{L_\nu}(C)$. For $\nu=1$, this reduces to Lipschitz continuity of the gradient. We use $\mathcal{C}^k(C)$
for $k$-times continuously differentiable functions on $C$.

\begin{fact}[H\"olderian descent lemma]\label{fact:holder:declem}\cite{Nesterov15univ}
Let $C\subseteq \inter{\dom{\gh}}$ and $\gh\in \mathcal{C}^{1, \nu}_{L_{\nu}}(C)$. For each $x, y\in C$ satisfying $[x,y]:=\{x+t(y-x)\mid t\in [0,1]\}\subseteq C$, the following holds
\[
\gh(y) \leq \gh(x) + \langle \nabla \gh(x), y - x \rangle +\frac{L_{\nu}}{1 + \nu}\Vert x - y \Vert ^{1+\nu}.
\]
\end{fact}

We next recall the Kurdyka-\L{}ojasiewicz (KL) property, which plays a crucial role in global and linear convergence analysis; see, e.g.,
\cite{absil2005convergence,Ahookhosh21,Ahookhosh2025Minimizing,attouch2010proximal,Attouch2013,bolte2007lojasiewicz,Bolte2007Clarke,Bolte2014,li2023convergence,Li18,Yu2022}.
\begin{definition}[Kurdyka-\L{}ojasiewicz property]\label{def:kldef}
Let $\gh: \R^n \to \Rinf$ be a proper lsc function. We say that $\gh$ satisfies the \textit{Kurdyka-\L{}ojasiewicz (KL) property} at $\ov{x}\in\Dom{\partial \gh}$ if there exist constants $r>0$, $\eta\in (0, +\infty]$, and a desingularizing function $\phi:[0, \eta)\to [0, +\infty)$ such that
$\phi$ is concave and continuous on $[0,\eta)$, continuously differentiable on $(0,\eta)$, $\phi(0)=0$, and $\phi'>0$ on $(0,\eta)$, and
\[
\phi'\left(\vert\gh(x) - \gh(\ov{x})\vert\right)\dist\left(0, \partial \gh(x)\right)\geq 1,
\]
whenever $x\in \mb(\ov{x}; r)\cap\Dom{\partial \gh}$ with $0<\vert \gh(x)-\gh(\ov{x})\vert<\eta$.
We say that $\gh$ has KL exponent $\theta\in[0,1)$ at $\ov{x}$ if $\phi(t)=ct^{1-\theta}$ for some $c>0$.
\end{definition}
The absolute-value formulation in Definition~\ref{def:kldef} is slightly stronger than the classical one and is convenient here because the inexact function values do not need to be monotone;
see, e.g., \cite{Kabgani24itsopt,li2023convergence}.

We say that $\gh$ satisfies the KL property at $\ov{x} \in \Dom{\partial \gh}$ with the \textit{quasi-additivity property}
if there exists a constant $c_{\phi} > 0$ such that the desingularizing function $\phi$ satisfies 
\[
\left[\phi'(s+t)\right]^{-1}\leq c_\phi \left[\left(\phi'(s)\right)^{-1}+\left(\phi'(t)\right)^{-1}\right], \qquad\forall s, t\in (0, \eta)~\text{with}~ s+t<\eta.
\]
For $\phi(t)=ct^{1-\theta}$, this property holds with $c_\phi=1$ \cite{li2023convergence}.

%
\begin{fact}[Uniformized KL property]\label{lem:Uniformized KL property}\cite[Lemma 2.4]{Kabgani24itsopt}
Let $\gh:\R^n\to \Rinf$ be a proper lsc function that is constant on a nonempty compact set $C\subseteq \Dom{\partial \gh}$.
If $\gh$ satisfies the KL property with the quasi-additivity property at each point of $C$, then there exist $r>0$, $\eta>0$, and a desingularizing function $\phi$ that satisfies the quasi-additivity property, such that for all $\ov{x}\in C$ and $x\in X$, where
\[
X:=\Dom{\partial \gh}\cap\left\{x\in \R^n \mid \dist(x, C)<r\right\}\cap\left\{x\in \R^n \mid 0<\vert \gh(x) - \gh(\ov{x})\vert<\eta\right\},
\]
we have
$\phi'\left(\vert \gh(x) - \gh(\ov{x})\vert\right)\dist\left(0, \partial \gh(x)\right)\geq 1$.
\end{fact}

\subsection{ Class of weakly convex functions}\label{sec:weakConvFunc}

We next recall weak convexity and the consequences used in the subsequent analysis; see, e.g.,
\cite{nurminskii1973quasigradient,Atenas2023,Davis2019Stochastic}.
\begin{definition}[Weak convexity]\label{def:weakconvex}
A proper function $\gh: \R^n\to \Rinf$ is said to be \textit{$\rho$-weakly convex} for $\rho>0$ if 
$\gh+\frac{\rho}{2}\Vert \cdot\Vert^2$ is a convex function.
\end{definition}

\begin{remark}[Smooth and nonsmooth subclasses of weakly convex functions]
Weakly convex functions include convex functions and $L$-smooth functions. More generally, the sum of an $L$-smooth function and a proper convex lsc function is weakly convex. Moreover, if 
    $\gh: \R^n\to \R$ is convex and Lipschitz and $\Psi:\R^m\to\R^n$ has Lipschitz continuous Jacobian, then $x\mapsto \gh(\Psi(x))$ is weakly convex; see, e.g., \cite[Lemma~4.2]{Drusvyatskiy19}.
\end{remark}

The following consequences of weak convexity are used repeatedly below; see, e.g., \cite{Atenas2023,daniilidis2005filling,Davis2019Stochastic}.
\begin{fact}[Subgradient inequalities]
\label{pro:weakConvChar}
Let $\gh:\R^n\to\Rinf$ be a proper lsc $\rho$-weakly convex function. Then:
\begin{enumerate}[label=(\textbf{\alph*}),font=\normalfont\bfseries,leftmargin=0.7cm]

\item\label{pro:weakConvChar:d}
for any $x,y\in\R^n$ with $\partial\gh(x)\neq\emptyset$ and any
$\zeta\in\partial\gh(x)$,
\[
\gh(y)\geq\gh(x)+\langle\zeta,y-x\rangle-\frac{\rho}{2}\Vert y-x\Vert^2;
\]

\item\label{pro:weakConvChar:e}
for any $x,y\in\R^n$ with $\partial\gh(x)\neq\emptyset$ and $\partial\gh(y)\neq\emptyset$,
\begin{equation}\label{eq:subGradWeakConv}
\langle\zeta_x-\zeta_y,x-y\rangle\geq -\rho\Vert x-y\Vert^2, \qquad\forall\zeta_x\in\partial\gh(x),\quad\forall\zeta_y\in\partial\gh(y).
\end{equation}
\end{enumerate}
\end{fact}

\section{Further properties of high--order Moreau envelope}
\label{sec:home}
This section recalls the high-order proximal operator (HOPE) and the high-order Moreau envelope (HOME), together with the properties needed in our analysis; see, e.g., \cite{Kabgani24itsopt,Kabganidiff}. We then derive explicit bounds on HOPE that provide the localization needed to establish differentiability and H\"older smoothness of HOME.

\begin{definition}[High-order proximal operator and Moreau envelope]\label{def:Hiorder-Moreau env}
Let $p>1$, $\gamma>0$, and $\gf: \R^n \to \Rinf$ be proper. 
The \textit{high-order proximal operator} (\textit{HOPE}) of $\gf$ with parameter $\gamma$,
$\prox{\gf}{\gamma}{p}: \R^n \rightrightarrows \R^n$, is given by
    \[
       \prox{\gf}{\gamma}{p} (x):=\argmint{y\in \R^n} \left(\gf(y)+\frac{1}{p\gamma}\Vert x- y\Vert^p\right),
    \]    
and the \textit{high-order Moreau envelope} (\textit{HOME}) of $\gf$ with parameter $\gamma$, 
$\fgam{\gf}{p}{\gamma}:\R^n\to \R\cup\{\pm \infty\}$, 
is given by
    \[
    \fgam{\gf}{p}{\gamma}(x):=\mathop{\bs{\inf}}\limits_{y\in \R^n} \left(\gf(y)+\frac{1}{p\gamma}\Vert x- y\Vert^p\right).
    \]
\end{definition}

To guarantee that HOPE is nonempty for suitable $\gamma$,
we recall the high-order prox-boundedness property.
\begin{definition}[High-order prox-boundedness]\label{def:s-prox-bounded}\cite[Definition~3.3]{Kabgani24itsopt}
A proper function $\gf:\R^n\to \Rinf$ is said to be \textit{high-order prox-bounded} 
of order $p> 1$, 
if there exist $\gamma>0$ and $x\in \R^n$ such that
$\fgam{\gf}{p}{\gamma}(x)>-\infty$. 
The supremum of all such $\gamma$ is denoted by $\gamma^{\gf, p}$ and is called the threshold of high-order prox-boundedness for $\gf$.
\end{definition}

\begin{remark}\label{rem:boundedproxbound}
If $\gf$ is bounded below, then it is high-order prox-bounded for every $p>1$, with $\gamma^{\gf,p}=+\infty$.  
This remains true for any lsc, proper convex function $\gf$; see 
\cite[Proposition~12.15]{Bauschke17}.
The converse does not hold. Let us consider the scalar function $\gf: \R\to \R$ given by $\gf(x)= -\sqrt{1+x^2}$,
which is nonconvex, unbounded below, and $1$-weakly convex. Fix any order $p>1$, any $\gamma>0$, and any $x\in\R$, and define
$F_x(y) := \gf(y) + \frac{1}{p\gamma}\,|x-y|^p$.
This function is coercive and continuous. Hence, $\fgam{\gf}{p}{\gamma}(x)>-\infty$, and therefore
$\gamma^{\gf,p}=+\infty$.
\end{remark}

We use the following basic properties of HOPE and HOME throughout the paper.
\begin{fact}[Basic properties of HOME and HOPE]\label{th:level-bound+locally uniform}\cite[Theorem~3.4]{Kabgani24itsopt}
Let $p>1$ and let $\gf: \R^n\to \Rinf$ be a proper lsc function that is high-order prox-bounded with a threshold $\gamma^{\gf, p}>0$. 
Then, for each $\gamma\in (0, \gamma^{\gf, p})$,
\begin{enumerate}[label=(\textbf{\alph*}), font=\normalfont\bfseries, leftmargin=0.7cm]
\item \label{level-bound+locally uniform:proxnonemp} $\prox{\gf}{\gamma}{p}(x)$ is nonempty and compact and 
$\fgam{\gf}{p}{\gamma}$ is finite and continuous on $\R^n$;
    
\item \label{level-bound+locally uniform2:conv} if $y^k\in \prox{\gf}{\gamma}{p}(x^k)$, with $x^k\to \ov{x}$, then the sequence $\{y^k\}_{k\in \mathbb{N}}$ is bounded. Furthermore, all cluster points of this sequence lie in $\prox{\gf}{\gamma}{p}(\ov{x})$.
     \end{enumerate}
\end{fact} 

The following lemma shows that coercivity of $\gf$ is equivalent to coercivity of its HOME.
\begin{lemma}[Coercivity]\label{lem:hiordermor:coer} \cite[Proposition~11]{Kabganidiff}
Let $p>1$ and let $\gf:\R^n\to\Rinf$ be proper and lsc. Then, for any
$\gamma>0$, $\gf$ is coercive if and only if $\fgam{\gf}{p}{\gamma}$ is coercive.
\end{lemma}
\begin{proof}
The forward implication follows from \cite[Proposition~11]{Kabganidiff}.
Conversely, $\fgam{\gf}{p}{\gamma}(x) \leq \gf(x)$ for every $x\in\R^n$.
Hence, if $\fgam{\gf}{p}{\gamma}$ is coercive, then $\gf(x)\geq\fgam{\gf}{p}{\gamma}(x)\to+\infty$ as $\|x\|\to+\infty$
which proves that $\gf$ is coercive.
\end{proof}

Next, we derive an explicit localization bound for HOPE. Whenever a bound involves $\gf(0)$, we assume $0\in\dom{\gf}$, which can always be achieved by a translation of variables.
\begin{proposition}[Boundedness of HOPE: high-order prox-boundedness]\label{prop:findtau:lip} 
Let $p>1$ and let $\gf: \R^n\to \Rinf$ be a proper lsc function that is high-order prox-bounded with a threshold $\gamma^{\gf, p}>0$. Then
\begin{equation*}\label{findtau:main}
\prox{\gf}{\gamma}{p}(x)=\argmint{\{y\in \R^n \mid \Vert y\Vert \leq\tau\}}\left(\gf(y)+\frac{1}{p\gamma}\Vert x- y\Vert^p\right),
\end{equation*}
for each $\gamma\in \left(0, 4^{1-p}\widehat{\gamma}\right)$,
$\widehat{\gamma}\in (0, \gamma^{\gf, p})$, 
and for any $r>0$, $x\in \mb(0; r)$,
and
$\tau:= \left(\frac{2r^p+p\gamma(\gf(0)-\ell_0)}{2^{1-p}-\ell p\gamma}\right)^{\frac{1}{p}}$,
where
$\ell=\frac{2^{p-1}}{p\widehat{\gamma}}$ and $\ell_0$ is a lower bound of $\gf(\cdot)+\ell\Vert \cdot\Vert^p$ on $\R^n$.
\end{proposition}
\begin{proof}
By \cite[Proposition~9]{Kabganidiff}, $\gf(y)+\ell\Vert y\Vert^p\geq\ell_0$ for all $y\in\R^n$.
Let $x\in\mb(0;r)$ and $y\in\prox{\gf}{\gamma}{p}(x)$. Then,
 \begin{equation}\label{findtau:eq2}
-\ell \Vert y\Vert^p +\ell_0+\frac{1}{p\gamma}\Vert x- y\Vert^p\leq   \gf(y)+\frac{1}{p\gamma}\Vert x- y\Vert^p=\fgam{\gf}{p}{\gamma}(x)\leq \gf(0)+\frac{1}{p\gamma}\Vert x\Vert^p.
 \end{equation}
Since $\Vert y\Vert^p\leq 2^{p-1}\left(\Vert x\Vert^p+\Vert y-x\Vert^p\right)$, it is deduced from \eqref{findtau:eq2} that
\[
 -\ell\Vert y\Vert^p+\frac{1}{p\gamma}\left(2^{1-p}\Vert y\Vert^p-\Vert x\Vert^p\right)\leq\frac{1}{p\gamma}\Vert x\Vert^p+(\gf(0)-\ell_0).
\]
Thus, $(2^{1-p}-\ell p\gamma)\Vert y\Vert^p\leq 2\Vert x\Vert^p+p\gamma(\gf(0)-\ell_0)$.
Hence $\Vert y\Vert<\tau$, which proves the claim.
 \end{proof}

\begin{remark}\label{rem:prop:findtau:lip}
A related localization result in Banach spaces is given in \cite[Theorem~3.1]{KecisThibault15}, under different assumptions. Proposition~\ref{prop:findtau:lip} provides an explicit radius $\tau$ in the present setting. The auxiliary parameter $\widehat{\gamma}$ is chosen strictly below $\gamma^{\gf,p}$ because the prox-boundedness threshold need not itself belong to the admissible set, whereas \cite[Proposition~9]{Kabganidiff} requires a value for which HOME is finite.
\end{remark}

For lower-bounded functions, the localization bound simplifies and holds for every $\gamma>0$.
\begin{corollary}[Boundedness of HOPE: lower boundedness]\label{cor:findtau:lip}
Let $p>1$ and let $\gf: \R^n\to \Rinf$ be a proper lsc function that is bounded from below by $\ell_0\in \R$. 
Then
\begin{equation*}\label{cor:findtau:main}
\prox{\gf}{\gamma}{p}(x)=\argmint{\{y \in \R^n \mid \Vert y\Vert \leq\tau\}}\left(\gf(y)+\frac{1}{p\gamma}\Vert x- y\Vert^p\right),
\end{equation*}
for any $\gamma>0$, $r>0$, 
$x\in  \mb(0; r)$,
and
$\tau:=2 r+\left(2^{p-1}p\gamma(\gf(0)-\ell_0)\right)^\frac{1}{p}$.
\end{corollary}
\begin{proof}
Let $y\in\prox{\gf}{\gamma}{p}(x)$. Similar to the proof of Proposition~\ref{prop:findtau:lip}, we obtain
\[
\frac{1}{p\gamma}\left(2^{1-p}\Vert y\Vert^p-\Vert x\Vert^p\right)\leq\frac{1}{p\gamma}\Vert x\Vert^p+(\gf(0)-\ell_0),
\]
leading to $\Vert y\Vert^p\leq 2^{p-1}\left(2\Vert x\Vert^p+p\gamma(\gf(0)-\ell_0)\right)$.
From \eqref{eq:intrp:p01}, we obtain
\begin{align*}
\Vert y\Vert \leq \left[2^{p-1}\left(2\Vert x\Vert^p+p\gamma(\gf(0)-\ell_0)\right)\right]^\frac{1}{p}
\leq 2 \Vert x\Vert +\left(2^{p-1}p\gamma(\gf(0)-\ell_0)\right)^\frac{1}{p}.
\end{align*}
Since $\Vert x\Vert<r$, the result follows.
 \end{proof}

\subsection{Differentiability and weak smoothness of HOME} \label{sec:diff}

This subsection establishes differentiability and weak smoothness of HOME for weakly convex functions. We show that, for any prescribed radius $r>0$, HOME is continuously differentiable on $\mb(0;r)$ when
$\gamma>0$ is sufficiently small, and we subsequently derive a H\"older continuity estimate for its gradient. The analysis is restricted to $p\in(1,2]$; Example~\ref{ex:nondif:prox:p3} indicates that differentiability may fail for $p>2$.

We first recall a characterization of differentiability of HOME.
\begin{fact}[Characterization of differentiability of HOME]\label{th:diffcharact}\cite[Proposition~20]{Kabganidiff} (see also \cite[Proposition~3.1]{KecisThibault15})
Let $p > 1$, and let $\gf: \R^n \to \Rinf$ be proper lsc that is high-order prox-bounded with a threshold $\gamma^{\gf, p} > 0$. Then, for each $\gamma \in (0, \gamma^{\gf, p})$ and any open subset $U \subseteq \R^n$, the following statements are equivalent:
\begin{enumerate}[label=(\textbf{\alph*}), font=\normalfont\bfseries, leftmargin=0.7cm]
\item \label{th:diffcharact:diff} $\fgam{\gf}{p}{\gamma} \in \mathcal{C}^{1}(U)$;
\item \label{th:diffcharact:prox} $\prox{\gf}{\gamma}{p}$ is nonempty, single-valued, and continuous on $U$.
\end{enumerate}
Under these conditions, for any $x \in U$ and $y = \prox{\gf}{\gamma}{p}(x)$, we have $\nabla\fgam{\gf}{p}{\gamma}(x) = \frac{1}{\gamma} \Vert x - y \Vert^{p-2} (x - y)$.
\end{fact}

In general, when $\gf$ is a weakly convex function and $p>2$, there may not exist any $\gamma>0$ such that $\fgam{\gf}{p}{\gamma}$
 is differentiable. The following example, taken from \cite[Example~27]{Kabganidiff}, illustrates this possibility.
 
\begin{example}[Nondifferentiability of HOME for $p > 2$]\label{ex:nondif:prox:p3}
Consider $\gf(x) = x^4 - x^2$ on $\R$, which is weakly convex. For every $p > 2$ and $\gamma > 0$, the set
$\argmin{y \in \R}\left\{\Phi(y):=y^4 - y^2 + \frac{1}{p\gamma}|y|^p\right\}$ has two elements; see \cite[Example~27]{Kabganidiff}.
Hence, $\prox{\gf}{\gamma}{p}(0)$ is not single-valued. 
By Fact~\ref{th:diffcharact}, $\fgam{\gf}{p}{\gamma}$ is not differentiable at $\ov{x}=0$. 
\end{example}
In view of Example~\ref{ex:nondif:prox:p3}, we restrict throughout to $p\in(1,2]$.
 
\begin{remark}\label{rem:boundongamma}
The radius $\tau$ in Proposition~\ref{prop:findtau:lip} may diverge as $\gamma\to4^{1-p}\widehat{\gamma}$.  
We therefore fix $0<\gamma_{\max}<4^{1-p}\widehat{\gamma}$, where
$\widehat{\gamma}\in(0,\gamma^{\gf,p})$, and restrict the admissible values of $\gamma$ to $(0,\gamma_{\max}]$.
\end{remark}

 \begin{theorem}[Differentiability of HOME]\label{th:dif:dif}
Let $p \in (1,2]$, and let $\gf: \R^n \to \Rinf$ be a proper, lsc, and $\rho$-weakly convex function that is high-order prox-bounded with threshold $\gamma^{\gf, p} > 0$. For any $r > 0$, we have that $\fgam{\gf}{p}{\gamma}\in \mathcal{C}^{1}(\mb(0; r))$ and that $\prox{\gf}{\gamma}{p}$ is single-valued and continuous on $\mb(0; r)$, provided that $\gamma \in \left(0, \sigma \right)$, where
$\sigma:=\bs\min\left\{\gamma_{\bs\max}, \frac{p-1}{\rho}(r+\widehat{\tau})^{p-2}\right\}$, 
$\gamma_{\bs\max}$ is as described in Remark \ref{rem:boundongamma},
$\widehat{\tau}= \left(\frac{2r^p+p\gamma_{\bs\max}(\gf(0)-\ell_0)}{2^{1-p}-\ell p\gamma_{\bs\max}}\right)^{\frac{1}{p}}$, $\ell=\frac{2^{p-1}}{p\widehat{\gamma}}$, where $\widehat{\gamma}\in (0, \gamma^{\gf, p})$,  and $\ell_0$ is a lower bound of the function $\gf(\cdot)+\ell\Vert \cdot\Vert^p$ on $\R^n$.
 \end{theorem}
  \begin{proof}
  See Appendix~\ref{proof:th:dif:dif}.
 \end{proof}
 
We next quantify the regularity of HOPE and the gradient of HOME under the same assumptions.
 \begin{theorem}[H\"older regularity of HOPE and HOME]\label{th:dif:weak}
Let the assumptions of Theorem~\ref{th:dif:dif} hold. Then,
\begin{enumerate}[label=(\textbf{\alph*}), font=\normalfont\bfseries, leftmargin=0.7cm]
\item \label{th:dif:weak:a} there exists $L_p>0$ such that
\begin{equation}\label{lochol:maineq:weak}
\Vert \prox{\gf}{\gamma}{p}(x_1) -  \prox{\gf}{\gamma}{p}(x_2) \Vert\leq L_p \Vert  x_1-x_2\Vert^{\frac{1}{2}},\qquad
\forall x_1, x_2\in \mb(0; r);
\end{equation}

\item \label{th:dif:weak:b}  it holds that
$\fgam{\gf}{p}{\gamma}\in\mathcal{C}^{1, \frac{p-1}{2}}_{\mathcal{L}_p}(\mb(0; r))$, i.e., 
there exists $\mathcal{L}_p>0$ such that
  \begin{align*}
\left\Vert \nabla \fgam{\gf}{p}{\gamma}(x_2)-\nabla \fgam{\gf}{p}{\gamma}(x_1)\right\Vert \leq  \mathcal{L}_p\Vert x_2-x_1\Vert^{\frac{p-1}{2}},
\qquad \forall x_1, x_2\in \mb(0; r).
  \end{align*}
\end{enumerate}
\end{theorem}
\begin{proof}
See Appendix~\ref{proof:th:dif:weak}.
\end{proof}

\begin{corollary}[Differentiability and weak smoothness of HOME under lower-boundedness]\label{cor:dif:weak}
Let $p \in (1,2]$, and let $\gf: \R^n \to \Rinf$ be a proper, lsc, and $\rho$-weakly convex function that is
 bounded from below by $\ell_0$.
Let $r>0$ be given and $\gamma\in \left(0, \sigma\right)$, where
$\sigma=\bs\min\left\{\gamma_{\bs\max},\frac{p-1}{\rho}(r+\ov{\tau})^{p-2}\right\}$, $\gamma_{\bs\max}>0$ is fixed, and 
$\ov{\tau}=2 r+\left(2^{p-1}p\gamma_{\bs\max}(\gf(0)-\ell_0)\right)^\frac{1}{p}$.
Then, 
\begin{enumerate}[label=(\textbf{\alph*}), font=\normalfont\bfseries, leftmargin=0.7cm]
\item 
$\fgam{\gf}{p}{\gamma}\in \mathcal{C}^{1}(\mb(0; r))$ and $\prox{\gf}{\gamma}{p}$ is single-valued and continuous on $\mb(0; r)$;
\item There exists $L_p>0$ such that
\[
\Vert \prox{\gf}{\gamma}{p}(x_1) -  \prox{\gf}{\gamma}{p}(x_2) \Vert\leq L_p\Vert  x_1-x_2\Vert^{\frac{1}{2}},\qquad
\forall x_1, x_2\in \mb(0; r);
\]

\item we have
$\fgam{\gf}{p}{\gamma}\in\mathcal{C}^{1, \frac{p-1}{2}}_{\mathcal{L}_p}(\mb(0; r))$, i.e., 
there exists $\mathcal{L}_p>0$ such that
  \begin{align*}
\left\Vert \nabla \fgam{\gf}{p}{\gamma}(x_2)-\nabla \fgam{\gf}{p}{\gamma}(x_1)\right\Vert \leq  \mathcal{L}_p\Vert x_2-x_1\Vert^{\frac{p-1}{2}},
\qquad
\forall x_1, x_2\in \mb(0; r).
  \end{align*}
  \end{enumerate}
\end{corollary}
\begin{proof}
It follows from Corollary~\ref{cor:findtau:lip} and Theorems~\ref{th:dif:dif}~and~\ref{th:dif:weak}.
\end{proof}
For $p=2$, sharper global differentiability and Lipschitz-gradient results are available from the classical Moreau-envelope theory under $0<\gamma<1/\rho$; see, e.g.,
\cite[Lemma~3.1]{Bohm21} and \cite{Davis2019Stochastic}.

\section{ItsDEAL: Inexact two-level smoothing descent algorithms} \label{sec:itsopt}
We here introduce the Inexact two-level smoothing DEscent ALgorithm (ItsDEAL), a generic framework based on HOME for solving weakly convex optimization problems of the form~\eqref{eq:mainproblemconv2}. ItsDEAL follows the inexact two-level smoothing optimization framework
(ItsOPT) introduced in \cite{Kabgani24itsopt}. ItsOPT consists of two levels. At the upper level, a zero-, first-, or
second-order method is applied to a smooth approximation of the original cost function. At the lower level, the associated proximal problem is solved inexactly to generate an inexact oracle for the upper-level
method. The proximal subproblem can be solved approximately using, for example, subgradient methods \cite{Davis2018,Rahimi2024}, Bregman gradient methods \cite{Ahookhosh24}, or BELLA \cite{Ahookhosh21}. For completeness, we summarize the ItsOPT framework below.

\begin{algorithm}[H]
\caption{ItsOPT (Inexact two-level smoothing OPTimization algorithm)}\label{alg:itsopt}
\begin{algorithmic}[1]
\State\label{alg:itsopt:env}\textbf{Smoothing technique} Choose an envelope for generating a smooth variant of $\varphi$;
\State \textbf{Initialization} Start with $x^0\in \R^n$, $\gamma>0$, and set $k=0$;
\While{the stopping criterion is not satisfied}
\State Find an approximate solution for the proximal auxiliary problem; \Comment{lower-level}
\State Generate the inexact zero-, first-, or second-order oracle $\mathcal{O}(\varphi_\gamma,x^k)$; \Comment{lower-level}
\State\label{alg:itsopt:dir}Choose a direction $d^k\in \R^n$ using the inexact oracle $\mathcal{O}(\varphi_\gamma,x^k)$; \Comment{upper-level}
\State Find the step-size $\alpha_k>0$ and update $x^{k+1}:=x^k+\alpha_k d^k$;\Comment{upper-level}
\State $k=k+1$;
\EndWhile
\end{algorithmic}
\end{algorithm}

In this section, we specialize ItsOPT by using HOME as the smoothing approximation and developing inexact first-order descent methods at the upper level, while the corresponding high-order proximal subproblem is solved approximately at the lower level. We consider a H\"olderian inexact gradient descent method, its parameter-free variant, and an inexact descent method using the Armijo line search.

\subsection{Inexact oracle for HOME} 
\label{subsec:inexactorac}
We note that the exact computation of HOPE may be expensive, and we therefore generate approximate proximal points to construct an inexact first-order oracle for HOME. The following assumptions are used throughout this section.
\begin{assumption}\label{assum:approx} Let us consider problem \eqref{eq:mainproblemconv2}. We assume that
\begin{enumerate}[label=(\textbf{\alph*}), font=\normalfont\bfseries, leftmargin=0.7cm]
\item \label{assum:approx:coer weak con}
$p\in(1,2]$, and $\gf:\R^n\to\Rinf$ is proper, lsc, coercive, and $\rho$-weakly convex. Let $x^*$ be a minimizer of $\gf$.


\item \label{assum:approx:eps}  $\{\varepsilon_k\}_{k\in \Nz}$ and $\{\delta_k\}_{k\in \Nz}$ are non-increasing sequences in $[0,+\infty)$ such that $\ov{\varepsilon}=\sum_{k=0}^{\infty} \varepsilon_k<\infty$ and $\delta_k\downarrow 0$;    

\item \label{assum:approx:aproxep} 
Fix $r>0$ and $\gamma>0$ such that Corollary~\ref{cor:dif:weak} applies on $\mb(0;r)$.
For every $x^k\in\mb(0;r)$, we can compute $\prox{\gf}{\gamma}{p,\varepsilon_k}(x^k)$ such that, for some fixed $\mu\in(0,1)$,
\begin{equation}\label{eq:ep-approx:dist}
\left\|\prox{\gf}{\gamma}{p,\varepsilon_k}(x^k)-\prox{\gf}{\gamma}{p}(x^k)\right\|\leq
\bs\min\left\{\delta_k,\,\mu\left\|x^k-\prox{\gf}{\gamma}{p,\varepsilon_k}(x^k)\right\|\right\},
\end{equation}
and
\begin{equation}\label{eq:ep-approx:fun}
\gf(\prox{\gf}{\gamma}{p, \varepsilon_k}(x^k))+\frac{1}{p\gamma}\Vert x^k-\prox{\gf}{\gamma}{p, \varepsilon_k}(x^k)\Vert^p\leq \fgam{\gf}{p}{\gamma}(x^k)+\varepsilon_k.
\end{equation}
\end{enumerate}
\end{assumption}

The two terms in \ref{eq:ep-approx:dist} provide, respectively, an absolute error bound that vanishes asymptotically and a relative error bound with respect to the approximate proximal displacement.

 The inexact function value of  $\fgam{\gf}{p}{\gamma}(x^k)$ is defined as
\[
\fgam{\gf}{p,\varepsilon_k}{\gamma}(x^k):=\gf(\prox{\gf}{\gamma}{p, \varepsilon_k}(x^k))+\frac{1}{p\gamma}\Vert x^k-\prox{\gf}{\gamma}{p, \varepsilon_k}(x^k)\Vert^p.
\]
 By Fact~\ref{th:diffcharact}, if $\fgam{\gf}{p}{\gamma}$ is differentiable, then 
 $\nabla\fgam{\gf}{p}{\gamma}(x^k)=\frac{1}{\gamma}\Vert x^k-\prox{\gf}{\gamma}{p}(x^k)\Vert^{p-2}(x^k-\prox{\gf}{\gamma}{p}(x^k))$. Motivated by this fact,  if $\fgam{\gf}{p}{\gamma}$ is differentiable,  we define the corresponding \textit{inexact gradient} $\nabla\fgam{\gf}{p,\varepsilon_k}{\gamma}(x^k)$ of $\nabla\fgam{\gf}{p}{\gamma}(x^k)$ as
\begin{equation}\label{eq:ep-gard-approx1}
\nabla\fgam{\gf}{p,\varepsilon_k}{\gamma}(x^k):=\frac{1}{\gamma}\Vert x^k-\prox{\gf}{\gamma}{p, \varepsilon_k}(x^k)\Vert^{p-2}(x^k-\prox{\gf}{\gamma}{p, \varepsilon_k}(x^k)).
\end{equation}
The pair $\left(\fgam{\gf}{p,\varepsilon_k}{\gamma}(x^k),\nabla\fgam{\gf}{p,\varepsilon_k}{\gamma}(x^k)\right)$ defines the inexact first-order oracle for HOME at $x^k$.

\begin{remark}\label{rem:relapprox} 
Let us consider Assumption~\ref{assum:approx}.
\begin{enumerate}[label=(\textbf{\alph*}), font=\normalfont\bfseries, leftmargin=0.7cm]
\item \label{rem:relapprox:a} Both \eqref{eq:ep-approx:dist} and \eqref{eq:ep-approx:fun} imply that if $\delta_k =0 $ or $\varepsilon_k=0$, 
 then $\prox{\gf}{\gamma}{p, \varepsilon_k}(x^k)=\prox{\gf}{\gamma}{p}(x^k)$.

\item \label{rem:relapprox:b} Using \cite[Theorem 6.3]{Rodomanov2020}, and analogous to the proof of Theorem~\ref{th:dif:weak}, we can show
\begin{equation}\label{eq:rem:relapprox:b1}
  \left\Vert \nabla\fgam{\gf}{p,\varepsilon_k}{\gamma}(x^k)-\nabla\fgam{\gf}{p}{\gamma}(x^k)\right\Vert
 \leq \frac{2^{2-p}}{\gamma} \left\Vert \prox{\gf}{\gamma}{p, \varepsilon_k}(x^k)-\prox{\gf}{\gamma}{p}(x^k)\right\Vert^{p-1}.
\end{equation}
Combining \eqref{eq:ep-approx:dist} and \eqref{eq:rem:relapprox:b1} leads to
\begin{align}\label{eq:relapprox:relgardgrad}
\left\Vert \nabla\fgam{\gf}{p,\varepsilon_k}{\gamma}(x^k)-\nabla\fgam{\gf}{p}{\gamma}(x^k)\right\Vert \leq \frac{2^{2-p}}{\gamma}\delta_k^{p-1}.
\end{align}
Moreover, it follows from \eqref{eq:ep-approx:dist}, \eqref{eq:ep-gard-approx1}, and \eqref{eq:rem:relapprox:b1} that
\begin{equation}\label{eq:relapprox:b}
\left\Vert \nabla\fgam{\gf}{p,\varepsilon_k}{\gamma}(x^k)-\nabla\fgam{\gf}{p}{\gamma}(x^k)\right\Vert \leq \frac{2^{2-p}\mu^{p-1}}{\gamma} \Vert x^k -\prox{\gf}{\gamma}{p, \varepsilon_k}(x^k)\Vert^{p-1}
 \leq c_p \Vert \nabla\fgam{\gf}{p,\varepsilon_k}{\gamma}(x^k)\Vert,
\end{equation}
where $c_p := 2^{2-p}\mu^{p-1} > 0$. 
This inequality has been extensively studied in the context of inexact gradient methods; see \cite{byrd2012,Carter91,Khanh24,lan2016} and references therein.

 \item  \label{rem:relapprox:c}  Using Fact~\ref{fact:holder:declem} and Corollary~\ref{cor:dif:weak}, we get
 \begin{equation}\label{eq2:lem:ep-prox:p(1,2)}
\fgam{\gf}{p}{\gamma}(y) \leq \fgam{\gf}{p}{\gamma}(x) + \langle \nabla \fgam{\gf}{p}{\gamma}(x), y - x \rangle +\frac{2\mathcal{L}_p}{p+1}\Vert x - y \Vert ^{\frac{p+1}{2}}, \qquad
\forall x, y \in \mb(0; r),
\end{equation}
Consequently, from \eqref{eq:ep-approx:fun} and \eqref{eq2:lem:ep-prox:p(1,2)}, for any $x, y \in \mb(0; r)$ and $\varepsilon_y, \varepsilon_x \geq 0$, we have
\begin{align*}
\fgam{\gf}{p,\varepsilon_y}{\gamma}(y)-\varepsilon_y\leq \fgam{\gf}{p,\varepsilon_x}{\gamma}(x)+\langle \nabla\fgam{\gf}{p}{\gamma}(x), y-x\rangle+\frac{2\mathcal{L}_p}{p+1}\Vert x - y \Vert ^{\frac{p+1}{2}}.
\end{align*}

\item\label{rem:relapprox:e}
Inequality~\eqref{eq:relapprox:b} also yields a computable stopping criterion for the upper-level methods. Indeed, 
\[
\left\|\nabla\gf_\gamma^p(x^k)\right\|\leq\left\|\nabla\gf_\gamma^p(x^k)-\nabla\gf_{\gamma}^{p,\varepsilon_k}(x^k)\right\|+\left\|\nabla\gf_{\gamma}^{p,\varepsilon_k}(x^k)\right\|\leq (1+c_p)\left\|\nabla\gf_{\gamma}^{p,\varepsilon_k}(x^k)\right\|,
\]
i.e., for any prescribed tolerance $\tau_{\rm out}>0$, the computable condition 
\begin{equation}\label{eq:outer-stopping}
\left\|\nabla\gf_{\gamma}^{p,\varepsilon_k}(x^k)\right\|\leq\frac{\tau_{\rm out}}{1+c_p},
\end{equation}
guarantees $\left\|\nabla\gf_\gamma^p(x^k)\right\|\leq\tau_{\rm out}$.
Equivalently, using~\eqref{eq:ep-gard-approx1}, condition~\eqref{eq:outer-stopping} can be written as
\[
\left\|x^k-\prox{\gf}{\gamma}{p, \varepsilon_k}(x^k)\right\|\leq\left(\frac{\gamma\tau_{\rm out}}{1+c_p}\right)^{\frac{1}{p-1}}.
\]
Thus, the upper-level stopping test requires neither the exact HOPE point nor the exact HOME gradient.
\end{enumerate}
\end{remark}
The conditions in Assumption~\ref{assum:approx} are stated in terms of the exact HOPE point and HOME value. We next derive computable certificates for these conditions.

\begin{proposition}[Strong convexity of the proximal model]\label{prop:proxmodel:strong}
Let $p\in(1,2]$, let $\gf:\R^n\to\Rinf$ be a proper, lsc, $\rho$-weakly convex function, and fix $x\in\R^n$ and $\gamma>0$.  Define the proximal model
\[
Q_{p,\gamma}(y;x):=\gf(y)+\frac{1}{p\gamma}\|y-x\|^p.
\]
Then the following statements hold.
\begin{enumerate}[label=(\textbf{\alph*}), font=\normalfont\bfseries, leftmargin=0.7cm]
\item \label{prop:proxmodel:strong:p2}
If $p=2$ and $0<\gamma<1/\rho$, then $Q_{2,\gamma}(\cdot;x)$ is globally $m_2$-strongly convex with
$m_2:=\frac{1}{\gamma}-\rho>0$.

\item \label{prop:proxmodel:strong:subquad}
Let $p\in(1,2)$ and $R>0$, and set $\mathcal{B}_R(x):=\{u\in\R^n\mid \|u-x\|\le R\}$.
If $m_R:=\frac{p-1}{\gamma R^{2-p}}-\rho>0$,
then $Q_{p,\gamma}(\cdot;x)$ is $m_R$-strongly convex on $\mathcal{B}_R(x)$.
\end{enumerate}
\end{proposition}
\begin{proof}
\ref{prop:proxmodel:strong:p2} This is standard; see, e.g., \cite{Bohm21} and \cite{Davis2019Stochastic}.
\\
\ref{prop:proxmodel:strong:subquad}
For any $u,v\in\mathcal{B}_R(x)$, Lemma~\ref{lem:findlowbounknu:lemma} applied to $u-x$ and $v-x$ gives
\[
\left\langle \frac1\gamma\|u-x\|^{p-2}(u-x) -\frac1\gamma\|v-x\|^{p-2}(v-x),u-v \right\rangle \geq \frac{p-1}{\gamma}R^{p-2}\|u-v\|^2.
\]
Hence $u\mapsto\|u-x\|^p/(p\gamma)$ is $m^{\rm reg}_{p,R}$-strongly convex on $\mathcal{B}_R(x)$ where
\[
m^{\rm reg}_{p,R}:=\frac{p-1}{\gamma}R^{p-2} =\frac{p-1}{\gamma R^{2-p}}.
\]
Together with the $\rho$-weak convexity of $\gf$, this yields the asserted modulus $m_R=m^{\rm reg}_{p,R}-\rho$.
\end{proof}
The preceding proposition makes the distance and value requirements in Assumption~\ref{assum:approx} computationally accessible.  The following result gives two useful certificates.

\begin{corollary}[Computable sufficient conditions for an inexact HOPE element]\label{cor:computable:inhope}
Let the assumptions of Proposition~\ref{prop:proxmodel:strong} hold, let $z\in\prox{\gf}{\gamma}{p}(x)$, and suppose that $Q_{p,\gamma}(\cdot;x)$ is $m$-strongly convex on a convex region containing both a candidate $y$ and $z$, where $m>0$.  
\begin{enumerate}[label=(\textbf{\alph*}), font=\normalfont\bfseries, leftmargin=0.7cm]
\item \label{cor:computable:inhope:residual}
Suppose that $y\in\Dom{\partial\gf}$ and that a vector $\xi_y\in\partial\gf(y)$ is available.  Define the lower-level stationarity residual
\[
r_{\gamma,p}(y;x):=\xi_y+\frac1\gamma\|y-x\|^{p-2}(y-x).
\]
Then
\begin{equation}\label{eq:computable:residual:dist}
\|y-z\|\leq \frac{\|r_{\gamma,p}(y;x)\|}{m},
\end{equation}
and
\begin{equation}\label{eq:computable:residual:value}
Q_{p,\gamma}(y;x)-\fgam{\gf}{p}{\gamma}(x) \leq \frac{\|r_{\gamma,p}(y;x)\|^2}{2m}.
\end{equation}
Consequently, if
\[
\frac{\|r_{\gamma,p}(y;x)\|}{m}\le\bs\min\{\delta_k,\mu\|x-y\|\},
\qquad
\frac{\|r_{\gamma,p}(y;x)\|^2}{2m}\le\varepsilon_k,
\]
then $y$ satisfies \eqref{eq:ep-approx:dist}--\eqref{eq:ep-approx:fun}.

\item \label{cor:computable:inhope:lowerbound}
Suppose that a computable lower bound
\[
\underline Q\le \fgam{\gf}{p}{\gamma}(x)=Q_{p,\gamma}(z;x)
\]
is available.  Define
\[
\widehat\varepsilon:=Q_{p,\gamma}(y;x)-\underline Q,
\qquad
\widehat\delta:=\sqrt{\frac{2\widehat\varepsilon}{m}}.
\]
Then
\begin{equation}\label{eq:computable:LB:cert}
Q_{p,\gamma}(y;x)-\fgam{\gf}{p}{\gamma}(x)\le\widehat\varepsilon,\qquad \|y-z\|\le\widehat\delta.
\end{equation}
Hence, if $\widehat\varepsilon\le\varepsilon_k$, $\widehat\delta\le\bs\min\{\delta_k,\mu\|x-y\|\}$, then $y$ satisfies \eqref{eq:ep-approx:dist}--\eqref{eq:ep-approx:fun}.
\end{enumerate}
\end{corollary}
\begin{proof}
\ref{cor:computable:inhope:residual} The first-order optimality condition for $z$ gives a vector $\xi_z\in\partial\gf(z)$ satisfying
\[
\xi_z+\frac1\gamma\|z-x\|^{p-2}(z-x)=0.
\]
In light of the hypomonotonicity inequality \eqref{eq:subGradWeakConv} and the strong monotonicity of the $p$-order regularizer on the stated region, we come to
\[
\langle r_{\gamma,p}(y;x),y-z\rangle\ge m\|y-z\|^2,
\]
which proves \eqref{eq:computable:residual:dist}.  Moreover, the weak convexity of $\gf$ together with strong convexity of the regularizer yields
\[
Q_{p,\gamma}(z;x)\ge Q_{p,\gamma}(y;x) +\langle r_{\gamma,p}(y;x),z-y\rangle +\frac{m}{2}\|z-y\|^2.
\]
Maximizing the resulting quadratic upper bound in $\|z-y\|$ gives \eqref{eq:computable:residual:value}.
\\
\ref{cor:computable:inhope:lowerbound} Strong convexity and optimality of $z$ imply
\[
Q_{p,\gamma}(y;x)-Q_{p,\gamma}(z;x)\ge\frac{m}{2}\|y-z\|^2.
\]
Since $\underline Q\le Q_{p,\gamma}(z;x)$, the definition of $\widehat\varepsilon$ gives both inequalities in \eqref{eq:computable:LB:cert}.
\end{proof}

\begin{remark}[Practical interpretation of Assumption~\ref{assum:approx}]\label{rem:computable:inhope}
Corollary~\ref{cor:computable:inhope} does not require access to an exact point in $\prox{\gf}{\gamma}{p}(x)$.  It shows that the lower-level conditions in Assumption~\ref{assum:approx} can be verified whenever the inner solver supplies either a computable stationarity residual as in Corollary~\ref{cor:computable:inhope}~\ref{cor:computable:inhope:residual}, or a valid lower bound as in Corollary~\ref{cor:computable:inhope}~\ref{cor:computable:inhope:lowerbound}, together with a positive strong-convexity modulus on a region containing the candidate and the exact HOPE point.  In the subquadratic case, the latter modulus is localization-dependent.  Bounds such as Corollary~\ref{cor:findtau:lip} can be used to localize the exact HOPE point, while the candidate $y$ is known explicitly.  The inner iteration is then continued until the resulting computable bounds satisfy the prescribed $\varepsilon_k$ and $\delta_k$ tolerances together with the relative-distance
condition in \eqref{eq:ep-approx:dist}.  
\end{remark}

The estimate~\eqref{eq:relapprox:b} allows descent directions for HOME to be certified using only the inexact gradient.
\begin{lemma}[Descent direction for $\fgam{\gf}{p}{\gamma}$]\label{lem:disdir}
Let Assumption~\ref{assum:approx} hold, $c_p := 2^{2-p}\mu^{p-1}$, $\fgam{\gf}{p}{\gamma} \in \mathcal{C}^{1}(U)$ for an open set $U$, and let $x^k \in U$ with $\nabla\fgam{\gf}{p}{\gamma}(x^k)\neq 0$. Then $d^k\neq 0$ is a descent direction if either of the following holds:
\begin{enumerate}[label=(\textbf{\alph*}), font=\normalfont\bfseries, leftmargin=0.7cm]
\item \label{lem:disdir:a} there exists some $c\in(0,1]$ such that $c_p<c$ and
\begin{equation}\label{eq:disdirapp}
\langle \nabla\fgam{\gf}{p,\varepsilon_k}{\gamma}(x^k), d^k \rangle \leq -c \Vert \nabla\fgam{\gf}{p,\varepsilon_k}{\gamma}(x^k)\Vert  \Vert d^k\Vert; 
\end{equation}
\item \label{lem:disdir:b}  there exists some $c_1, c_2>0$ and $\vartheta>0$, such that
$c_pc_2<c_1$ and
\begin{equation}\label{eq:dir:th:conv:alg:first:tem}
\langle\nabla\fgam{\gf}{p,\varepsilon_k}{\gamma}(x^k), d^k\rangle\leq -c_1  \Vert \nabla\fgam{\gf}{p,\varepsilon_k}{\gamma}(x^k)\Vert^{1+\vartheta},
\qquad \Vert d^k\Vert\leq c_2 \Vert \nabla\fgam{\gf}{p,\varepsilon_k}{\gamma}(x^k)\Vert^\vartheta. 
\end{equation} 
\end{enumerate}
\end{lemma}

\begin{proof} By expanding the inner product and \eqref{eq:relapprox:b}, we have
  \begin{align}
 \langle \nabla\fgam{\gf}{p}{\gamma}(x^k), d^k\rangle& = \langle \nabla\fgam{\gf}{p}{\gamma}(x^k)-\nabla\fgam{\gf}{p,\varepsilon_k}{\gamma}(x^k)
+\nabla\fgam{\gf}{p,\varepsilon_k}{\gamma}(x^k), d^k\rangle\nonumber\\
 &\leq \Vert  \nabla\fgam{\gf}{p}{\gamma}(x^k) - \nabla\fgam{\gf}{p,\varepsilon_k}{\gamma}(x^k)\Vert \Vert d^k\Vert
 +\langle\nabla\fgam{\gf}{p,\varepsilon_k}{\gamma}(x^k), d^k\rangle\nonumber\\
 &\leq c_p \Vert \nabla\fgam{\gf}{p,\varepsilon_k}{\gamma}(x^k)\Vert  \Vert d^k\Vert
 +\langle\nabla\fgam{\gf}{p,\varepsilon_k}{\gamma}(x^k), d^k\rangle.\label{eq:lem:disdir}
\end{align}
If $d^k$ satisfies \eqref{eq:disdirapp}
and $c_p<c$, then \eqref{eq:lem:disdir} gives
\[
\langle \nabla\fgam{\gf}{p}{\gamma}(x^k), d^k\rangle< c\Vert \nabla\fgam{\gf}{p,\varepsilon_k}{\gamma}(x^k)\Vert  \Vert d^k\Vert
 +\langle\nabla\fgam{\gf}{p,\varepsilon_k}{\gamma}(x^k), d^k\rangle\leq 0.
 \]
Moreover, if $d^k$ satisfies \eqref{eq:dir:th:conv:alg:first:tem}
and $c_pc_2<c_1$, then \eqref{eq:lem:disdir} yields
\begin{align*}
\langle \nabla\fgam{\gf}{p}{\gamma}(x^k), d^k\rangle\leq 2^{2-p}\mu^{p-1}c_2 \Vert \nabla\fgam{\gf}{p,\varepsilon_k}{\gamma}(x^k)\Vert^{1+\vartheta}
-c_1  \Vert \nabla\fgam{\gf}{p,\varepsilon_k}{\gamma}(x^k)\Vert^{1+\vartheta}<0,
\end{align*}
establishing our desired result.
\end{proof}
\begin{remark}\label{rem:examofdisdir}
Let the assumptions of Lemma~\ref{lem:disdir} hold.
\begin{enumerate}[label=(\textbf{\alph*}), font=\normalfont\bfseries, leftmargin=0.7cm]
\item \label{rem:examofdisdir:a} In Assertion~\ref{lem:disdir:a} of Lemma~\ref{lem:disdir}, for $c\leq 1$ and for a given $x^k\in U$ with $\nabla\fgam{\gf}{p}{\gamma}(x^k)\neq 0$, if the direction $d^k$ is of the form
$d^k= -D_k \nabla\fgam{\gf}{p,\varepsilon_k}{\gamma}(x^k)$ with $D_k>0$, then $d^k$ satisfies \eqref{eq:disdirapp}. 
Thus, if $\mu<\left(\frac{c}{2^{2-p}}\right)^{\frac{1}{p-1}}$, then $\langle \nabla\fgam{\gf}{p}{\gamma}(x^k), d^k\rangle<0$.
The {\it steepest descent directions}, where $D_k=1$, {\it scaled gradient descent directions}, where $D_k\in [D_{\bs \min}, D_{\bs \max}]$, with
$D_{\bs \min}, D_{\bs \max}>0$, and the directions obtained in {\it Barzilai-Borwein gradient methods} \cite{Barzilai1988Two} are instances of directions that satisfy \eqref{eq:disdirapp}.
\item The conditions stated in \eqref{eq:dir:th:conv:alg:first:tem} with $\vartheta=1$, are well-known also as {\it sufficient descent conditions} and the direction $d^k$ satisfying them is said to be the {\it sufficient descent direction}. Evidently, the directions mentioned in Assertion~\ref{rem:examofdisdir:a} are sufficient descent directions.
\end{enumerate}
\end{remark}

Since the differentiability and H\"older regularity of HOME are established on a prescribed ball, we need to choose $\gamma$ and a localization region that contain the generated iterates. The following remark shows how these quantities can be selected from the initial point.

\begin{remark}\label{rem:levelsetch}
Let $\gf: \R^n\to \Rinf$ be proper, lsc, and coercive, and let  $x^0\in\dom\gf$. 
Since $\mathcal{L}\bigl(\gf,\gf(x^0)+2\ov{\varepsilon}\bigr)$ is compact, where $\ov{\varepsilon}$ is defined in Assumption~\ref{assum:approx}~$\ref{assum:approx:eps}$, we can choose $R>0$ such that
$\mathcal{L}\bigl(\gf,\gf(x^0)+2\ov{\varepsilon}\bigr)\subseteq \mb(0;R)$. By \cite[Proposition 12]{Kabganidiff}, there exists $\ov{\gamma}>0$ such that, for every $\gamma\in(0,\ov{\gamma}]$,
\begin{align*}
\mathcal{S}:=\left\{x \mid \fgam{\gf}{p}{\gamma}(x)\leq \fgam{\gf}{p}{\gamma}(x^0)+2\ov{\varepsilon}\right\}
\subseteq \left\{x \mid \fgam{\gf}{p}{\gamma}(x)\leq \gf(x^0)+2\ov{\varepsilon}\right\}
\subseteq \mb(0; R).
\end{align*}
Fix $\gamma_{\bs\max}>0$ and let $\ell_0$ be a lower bound of $\gf$. Define
$\ov{\tau}=2 R+\left(2^{p-1}p\gamma_{\bs\max}(\gf(0)-\ell_0)\right)^\frac{1}{p}$.
Then, by Corollary~\ref{cor:findtau:lip}, for every $x\in\mb(0;R)$ and $\gamma\leq\gamma_{\bs\max}$,
$\prox{\gf}{\gamma}{p}(x)\subseteq \mb(0; \ov{\tau})$.
Hence, for any $r\geq R$, if 
$\gamma\in \left(0, \ov{\sigma}\right)$, where
$\ov{\sigma}=\bs\min\left\{\gamma_{\bs\max},\ov{\gamma},\frac{p-1}{\rho}(r+\ov{\tau})^{p-2}\right\}$,
then $\mathcal{S}\subseteq \mb(0; R)\subseteq \mb(0; r)$
and Corollary~\ref{cor:dif:weak} gives
$\fgam{\gf}{p}{\gamma}\in\mathcal{C}^{1, \frac{p-1}{2}}_{\mathcal{L}_p}(\mb(0; r))$.
\end{remark}
Assumptions~\ref{assum:rgamma2} and~\ref{assum:rgamma} below use this construction with different choices of $r$, adapted respectively to Algorithms~\ref{alg:ingrad}--\ref{alg:ingrad2} and Algorithm~\ref{alg:first}.

Building on the framework given in Algorithm~\ref{alg:itsopt}, Assumption~\ref{assum:approx}, and Remark~\ref{rem:levelsetch}, we next introduce 
the generic ItsDEAL framework.

\begin{algorithm}[H]
\caption{ItsDEAL (Inexact two-level Smoothing DEscent ALgorithm)}\label{alg:fram:nonmono}
\begin{algorithmic}[1]
\State\label{alg:fram:nonmono:init} \textbf{Initialization} 
Given $x^0\in\R^n$, choose $r\geq R$ and $\gamma>0$ as in Remark~\ref{rem:levelsetch}
such that $\mathcal{S}\subseteq \mb(0; R)$ and
$\fgam{\gf}{p}{\gamma}\in\mathcal{C}^{1, \frac{p-1}{2}}_{\mathcal{L}_p}(\mb(0; r))$. Pick $\vartheta, c_1, c_2>0$  such that $c_pc_2<c_1$;
\While{the stopping criterion is not satisfied}
    \State Compute an approximate solution of the proximal subproblem; \Comment{lower-level}
    \State Generate the inexact first-order oracle $\mathcal{O}(\fgam{\gf}{p}{\gamma},x^k)$; \Comment{lower-level}
    \State Choose $d^k$ using $\mathcal{O}(\fgam{\gf}{p}{\gamma},x^k)$ such that \eqref{eq:dir:th:conv:alg:first:tem} holds.
    \State Choose $\alpha_k>0$ such that $\varrho\leq\alpha_k\leq\varsigma$ for some fixed $\varrho,\varsigma>0$;
    \State\label{alg:fram:nonmono:newpoint} Generate $x^{k+1}=x^k+\alpha_k d^k$ such that $x^{k+1}\in \mathcal{S}$ and for some fixed $\widehat{\varrho}>0$,
    \begin{align}\label{eq:genstract:ineq}
    \fgam{\gf}{p,\varepsilon_{k+1}}{\gamma}(x^{k+1})\leq \fgam{\gf}{p,\varepsilon_k}{\gamma}(x^k)
        -\widehat{\varrho}\Vert\nabla\fgam{\gf}{p,\varepsilon_k}{\gamma}(x^k)\Vert^{1+\vartheta}
    +\varepsilon_{k+1}.
    \end{align}
    \State Set $k=k+1$;
\EndWhile
\end{algorithmic}
\end{algorithm}

Algorithm~\ref{alg:fram:nonmono} is a generic framework. The constants $\varrho$, $\varsigma$, and $\widehat{\varrho}$ represent uniform step-size and sufficient-decrease bounds verified by each particular realization. A computable upper-level stopping criterion is given in Remark~\ref{rem:relapprox}~\ref{rem:relapprox:e}. For Algorithms~\ref{alg:ingrad} and~\ref{alg:ingrad2}, these bounds follow from the corresponding fixed-step and adaptive rules. For Algorithm~\ref{alg:first}, the same framework applies when the accepted step-sizes are bounded away from zero, as assumed later in Corollary~\ref{cor:conv:alg:first}.

\subsection{H\"olderian inexact gradient descent methods}
\label{subsec:HiGDA}

We first consider a fixed-step H\"olderian inexact gradient descent method, which requires a valid H\"older constant $\mathcal{L}_p$ for HOME. We then replace this requirement by a backtracking strategy, leading to the parameter-free variant in Algorithm~\ref{alg:ingrad2}.

For convenience, let us set
\[
A_p:=c_1-c_pc_2,    
\]
where $c_p := 2^{2-p}\mu^{p-1}$ as introduced in Remark~\ref{rem:relapprox}~\ref{rem:relapprox:b}.
Throughout this subsection, we assume $A_p>0$.

\begin{algorithm}
\caption{HiGDA (H\"olderian inexact gradient descent algorithm)}\label{alg:ingrad}
\begin{algorithmic}[1]
\State \textbf{Initialization} Start with $x^0\in \R^n$ and set $k=0$. Choose $\gamma>0$ and $r>0$ such that
Assumption~\ref{assum:rgamma2} holds. Pick $c_1, c_2>0$ such that $c_pc_2<c_1$.
\State Choose $\alpha\leq\bs\min\left\{ \frac{\gamma^{\frac{2}{p-1}}}{c_2}, \left(\frac{(p+1)A_p}{2c_2^{\frac{p+1}{2}}\mathcal{L}_p}\right)^{\frac{2}{p-1}}\right\}$;
\While{the stopping criterion is not satisfied}
\State Find an approximate solution for the proximal auxiliary problem; \Comment{lower-level}
\State  Generate the inexact first-order oracle $\mathcal{O}(\fgam{\gf}{p}{\gamma},x^k)$; \Comment{lower-level}
\State Choose $d^k\in \R^n$ satisfying \eqref{eq:th:welldfalg:dir};
\State\label{alg:ingrad:xk+1} Set $x^{k+1}=x^k+\alpha d^k$, $k=k+1$;
\EndWhile
\end{algorithmic}
\end{algorithm}

We consider the following assumption, necessary for proving the well-definedness of Algorithms~\ref{alg:ingrad}~and~\ref{alg:ingrad2}.
\begin{assumption}\label{assum:rgamma2}
Considering the notations in Remark~\ref{rem:levelsetch}, for a given $x^0\in \R^n$, we set $r=R+\left(\frac{R+\ov{\tau}}{1-\mu}\right)^2$ and choose $\gamma\in \left(0, \ov{\sigma}\right)$ such that $\mathcal{S}\subseteq \mb(0; R)$ and
$\fgam{\gf}{p}{\gamma}\in\mathcal{C}^{1, \frac{p-1}{2}}_{\mathcal{L}_p}(\mb(0; r))$.
\end{assumption}

\begin{theorem}[Well-definedness of Algorithm~\ref{alg:ingrad}]\label{th:welldfalg} 
Let Assumptions~\ref{assum:approx}~and~\ref{assum:rgamma2} hold, and let $\{x^k\}_{k\in \Nz}$ be generated by Algorithm~\ref{alg:ingrad}. Suppose that each direction $d^k$ satisfies 
\begin{equation}\label{eq:th:welldfalg:dir}
\langle\nabla\fgam{\gf}{p,\varepsilon_k}{\gamma}(x^k), d^k\rangle\leq -c_1  \Vert \nabla\fgam{\gf}{p,\varepsilon_k}{\gamma}(x^k)\Vert^{\frac{p+1}{p-1}},
\qquad \Vert d^k\Vert\leq c_2 \Vert \nabla\fgam{\gf}{p,\varepsilon_k}{\gamma}(x^k)\Vert^{\frac{2}{p-1}}, 
\end{equation}
where $c_pc_2<c_1$. Then Algorithm~\ref{alg:ingrad} is well-defined.
\end{theorem}
\begin{proof}
See Appendix~\ref{proof:th:welldfalg}.
\end{proof}

From the proof of Theorem~\ref{th:welldfalg}, for every admissible step-size $\alpha>0$, we obtain
 \begin{equation}\label{eq:discentforseq}
   \fgam{\gf}{p,\varepsilon_{k+1}}{\gamma}(x^{k+1})\leq \fgam{\gf}{p,\varepsilon_k}{\gamma}(x^k)
    -\alpha \left(A_p -\frac{2\mathcal{L}_p}{p+1} \alpha^{\frac{p-1}{2}}c_2^{\frac{p+1}{2}}\right)\Vert\nabla\fgam{\gf}{p,\varepsilon_k}{\gamma}(x^k)\Vert^{\frac{p+1}{p-1}}
+\varepsilon_{k+1}.
 \end{equation}
Algorithm~\ref{alg:ingrad} requires a valid H\"older constant $\mathcal{L}_p$. To avoid this requirement, we replace $\mathcal{L}_p$ by a trial value and accept the corresponding step only when the predicted decrease is verified by the inexact HOME values.
Fix $\tau>1$ and $L_0>0$. At iteration $k$, let us set
\begin{equation}\label{eq:chooseL}
L_{k+1,0} =
\begin{cases}
L_0, & k=0,\\[1mm]
L_k/\tau, & k\geq1,
\end{cases}
\qquad
L_{k+1,j}:=\tau^jL_{k+1,0},
\quad j=0,1,\ldots.
\end{equation}
For each trial value $L_{k+1,j}$, we define
\begin{equation}\label{eq:pfhigda:trial-step}
\alpha_{k,j} :=\bs\min\left\{\frac{\gamma^{\frac{2}{p-1}}}{c_2},\left(\frac{A_p}{c_2^{\frac{p+1}{2}}L_{k+1,j}}\right)^{\frac{2}{p-1}}\right\},
\end{equation}
and set $\widehat{x}^{k+1}_j:=x^k+\alpha_{k,j}d^k$. The trial is accepted if
  \begin{equation}\label{eq:discentforseq:2}
    \fgam{\gf}{p,\varepsilon_{k+1}}{\gamma}(\widehat{x}_j^{k+1})\leq \fgam{\gf}{p,\varepsilon_k}{\gamma}(x^k)
    -\alpha_{k,j}\left(A_p -\frac{2L_{k+1,j}}{p+1}\alpha_{k,j}^{\frac{p-1}{2}}c_2^{\frac{p+1}{2}}\right)\Vert\nabla\fgam{\gf}{p,\varepsilon_k}{\gamma}(x^k)\Vert^{\frac{p+1}{p-1}}
+\varepsilon_{k+1}.
 \end{equation}

\begin{algorithm}
\caption{Parameter-free HiGDA (Parameter-free H\"olderian inexact gradient descent algorithm)}\label{alg:ingrad2}
\begin{algorithmic}[1]
\State \textbf{Initialization} Start with $x^0\in \R^n$ and set $k=0$. Choose $\gamma>0$ and $r>0$ such that
Assumption~\ref{assum:rgamma2} holds. Pick $c_1, c_2>0$, $L_0>0$, and $\tau>1$ such that $c_pc_2<c_1$.
\While{the stopping criterion is not satisfied}
\State Find an approximate solution for the proximal auxiliary problem; \Comment{lower-level}
\State  Generate the inexact first-order oracle $\mathcal{O}(\fgam{\gf}{p}{\gamma},x^k)$; \Comment{lower-level}
\State Choose $d^k\in \R^n$ satisfying \eqref{eq:th:welldfalg:dir};
\State Set $L_{k+1,0}$ as in \eqref{eq:chooseL} and $j=0$;
\Loop\label{alg:ingrad2:startinner}
\State Compute $\alpha_{k,j}$ from \eqref{eq:pfhigda:trial-step}, set $\widehat{x}^{k+1}_j=x^k+\alpha_{k,j}d^k$, and evaluate $\fgam{\gf}{p,\varepsilon_{k+1}}{\gamma}(\widehat{x}^{k+1}_j)$;
\If{\eqref{eq:discentforseq:2} holds}
\State Set $L_{k+1}=L_{k+1,j}$, $\alpha_k=\alpha_{k,j}$, and $x^{k+1}=\widehat{x}^{k+1}_j$;
\State \textbf{break};
\EndIf
\State Set $j= j+1$ and $L_{k+1,j}=\tau L_{k+1,j-1}$;
\EndLoop\label{alg:ingrad2:endinner}
\State $k=k+1$;
\EndWhile
\end{algorithmic}
\end{algorithm}

The following result establishes the well-definedness of the backtracking rule and provides the uniform decrease estimate needed in the convergence analysis.

\begin{theorem}[Well-definedness and uniform decrease of Algorithm~\ref{alg:ingrad2}]\label{th:pfhigda:well}
Let Assumptions~\ref{assum:approx} and~\ref{assum:rgamma2} hold, and suppose that the directions satisfy
\eqref{eq:th:welldfalg:dir}, with $A_p>0$. Then:
\begin{enumerate}[label=(\textbf{\alph*}), font=\normalfont\bfseries, leftmargin=0.7cm]
\item \label{th:pfhigda:well:a} every inner loop of Algorithm~\ref{alg:ingrad2} terminates after finitely many trials, and the generated iterates satisfy $x^k\in\mathcal{S}$.

\item \label{th:pfhigda:well:b} $L_k\leq L_{\max} :=\max\{L_0,\tau\mathcal{L}_p\}$ for $k\geq1$ and every accepted step-size satisfies
\begin{equation}\label{eq:alpha-min}
\alpha_k\geq\underline{\alpha}, \qquad \underline{\alpha}:=\bs\min\left\{\frac{\gamma^{\frac{2}{p-1}}}{c_2},\left(\frac{A_p}{c_2^{\frac{p+1}{2}}L_{\max}}\right)^{\frac{2}{p-1}}\right\}>0.
\end{equation}

\item \label{th:pfhigda:well:c} it holds that
\begin{equation}\label{eq:pfhigda:uniform-decrease}
\fgam{\gf}{p,\varepsilon_{k+1}}{\gamma}(x^{k+1}) \leq \fgam{\gf}{p,\varepsilon_k}{\gamma}(x^k)-\widehat{\varrho}_1
\left\|\nabla\fgam{\gf}{p,\varepsilon_k}{\gamma}(x^k)\right\|^{\frac{p+1}{p-1}}+\varepsilon_{k+1},
\end{equation}
where
$\widehat{\varrho}_1:=\underline{\alpha}\,\frac{p-1}{p+1}A_p>0$.
\end{enumerate}
\end{theorem}
\begin{proof}
See Appendix~\ref{proof:th:pfhigda:well}.
\end{proof}
 

\subsection{Descent algorithm with inexact Armijo line search}
\label{subsec:armijo}

We next consider an inexact Armijo line search based on the descent directions in Lemma~\ref{lem:disdir}.

\begin{theorem}[Well-definedness of inexact Armijo rule]\label{th:lem:disdir}
Let the assumptions of Lemma~\ref{lem:disdir} hold. If for a given $x^k\in U$ with $\nabla\fgam{\gf}{p}{\gamma}(x^k)\neq 0$ and a direction $d^k\in \R^n$,  either \eqref{eq:disdirapp} or \eqref{eq:dir:th:conv:alg:first:tem} holds, then for each $\lambda\in (0, 1)$, there exists some $\ov{\alpha}>0$ such that,
\[
 \fgam{\gf}{p,\varepsilon_{k+1}}{\gamma}(x^{k}+\alpha d^k)\leq \fgam{\gf}{p,\varepsilon_{k}}{\gamma}(x^{k})+\alpha \lambda \langle \nabla\fgam{\gf}{p}{\gamma}(x^k), d^k\rangle+\varepsilon_{k+1}, \qquad \forall \alpha\in (0, \ov{\alpha}).
\]
\end{theorem}
\begin{proof}
See Appendix~\ref{proof:th:lem:disdir}.
\end{proof}

The following assumption is considered regarding Algorithm~\ref{alg:first}.
\begin{assumption}\label{assum:rgamma}
Using the notation from Remark \ref{rem:levelsetch}, for a given
 $x^0\in \R^n$, we set $r=R$ and choose $\gamma\in \left(0, \ov{\sigma}\right)$ such that $\mathcal{S}\subseteq \mb(0; R)$ and
$\fgam{\gf}{p}{\gamma}\in\mathcal{C}^{1, \frac{p-1}{2}}_{\mathcal{L}_p}(\mb(0; r))$.
\end{assumption}

\begin{algorithm}[h]
\caption{IDEALS (Inexact DEscent Armijo Line Search)}\label{alg:first}
\begin{algorithmic}[1]
\State \textbf{Initialization} Start with $x^0\in\R^n$ and set $k=0$. Choose $\gamma>0$ and $r>0$ such that Assumption~\ref{assum:rgamma} holds. Pick $\vartheta,c_1,c_2>0$ such that $A_p>0$, and choose $\lambda,\upsilon\in(0,1)$.
\While{the stopping criterion is not satisfied}
\State Find an approximate solution for the proximal auxiliary problem; \Comment{lower-level}
\State  Generate the inexact first-order oracle $\mathcal{O}(\fgam{\gf}{p}{\gamma},x^k)$; \Comment{lower-level}
\State Choose $d^k\in \R^n$ satisfying \eqref{eq:dir:th:conv:alg:first:tem};
\State Set $\widehat{\alpha}_k=1$ and compute $\fgam{\gf}{p,\varepsilon_{k+1}}{\gamma}(x^{k}+\widehat{\alpha}_k d^k)$;
\While{$\fgam{\gf}{p,\varepsilon_{k+1}}{\gamma}(x^{k}+\widehat{\alpha}_k  d^k)> \fgam{\gf}{p,\varepsilon_{k}}{\gamma}(x^{k})-\widehat{\alpha}_k \lambda 
A_p\Vert \nabla\fgam{\gf}{p,\varepsilon_k}{\gamma}(x^k)\Vert^{1+\vartheta}+\varepsilon_{k+1}$} \label{alg:first:inwh}
\State Set $\widehat{\alpha}_k=\upsilon\widehat{\alpha}_k$ and evaluate $\fgam{\gf}{p,\varepsilon_{k+1}}{\gamma}(x^{k}+\widehat{\alpha}_k d^k)$
\EndWhile \label{alg:first:inendwh}
\State  \label{alg:first:xk+1} Set $\alpha_k=\widehat{\alpha}_k$ and $x^{k+1}=x^k+\alpha_k d^k$, $k=k+1$.
\EndWhile
\end{algorithmic}
\end{algorithm}

\begin{theorem}[Well-definedness of Algorithm~\ref{alg:first}]\label{th:welldf:alg:first}
Let Assumptions~\ref{assum:approx}~and~\ref{assum:rgamma} hold.
If the directions generated by Algorithm~\ref{alg:first} satisfy \eqref{eq:dir:th:conv:alg:first:tem} with $A_p>0$, then Algorithm~\ref{alg:first} is well-defined and $x^k\in\mathcal{S}$, for each $k\in\Nz$. 
\end{theorem}
\begin{proof}
See Appendix~\ref{proof:th:welldf:alg:first}.
\end{proof}

\subsection{Subsequential convergence analysis} \label{subsec:glconv}
We now collect the convergence consequences of the sufficient-decrease condition in Algorithm~\ref{alg:fram:nonmono}. We set
\[
\Delta_0 :=\fgam{\gf}{p,\varepsilon_0}{\gamma}(x^0) - \fgam{\gf}{p}{\gamma}(x^*) + \overline{\varepsilon}.
\]
We recall that $\ov{x}\in\R^n$ is a \textit{proximal fixed point}, denoted as $\ov{x}\in \bs{\rm Fix}(\prox{\gf}{\gamma}{p})$, if $\ov{x}\in \prox{\gf}{\gamma}{p}(\ov{x})$.

\begin{theorem}[Stationarity, complexity, and subsequential convergence]\label{th:conv:alg:nonmono}\label{th:comcon}
Let Assumption~\ref{assum:approx} hold and let $\{x^k\}_{k\in\Nz}$ be generated by Algorithm~\ref{alg:fram:nonmono}. Then the following statements hold.
\begin{enumerate}[label=(\textbf{\alph*}), font=\normalfont\bfseries, leftmargin=0.7cm]
\item\label{th:conv:alg:nonmono:a} 
$\sum_{k=0}^{\infty} \|\nabla\fgam{\gf}{p,\varepsilon_k}{\gamma}(x^k)\|^{1+\vartheta} \leq \Delta_0/\widehat{\varrho}$,
$\Vert \nabla\fgam{\gf}{p,\varepsilon_k}{\gamma}(x^k)\Vert\to 0$,
and $\Vert \nabla\fgam{\gf}{p}{\gamma}(x^k)\Vert\to 0$.

\item\label{th:conv:alg:nonmono:b} We have,
\begin{equation}\label{eq:conv:alg:nonmono1}
\bs\min_{0\leq k\leq N} \Vert \nabla\fgam{\gf}{p,\varepsilon_k}{\gamma}(x^k)\Vert\leq \left(
\frac{\Delta_0}{(N+1)\widehat{\varrho}}
\right)^{\frac{1}{1+\vartheta}},
\end{equation}
and
\begin{equation}\label{eq:conv:alg:nonmono2}
\bs\min_{0\leq k\leq N} \Vert \nabla\fgam{\gf}{p}{\gamma}(x^k)\Vert\leq (1+c_p) \left(
\frac{\Delta_0}{(N+1)\widehat{\varrho}}
\right)^{\frac{1}{1+\vartheta}}.
\end{equation}

\item\label{th:conv:alg:nonmono:c} $\{x^k\}_{k\in\Nz}$ has cluster points, and every cluster point belongs to $\bs{\rm Fix}(\prox{\gf}{\gamma}{p})$.
\end{enumerate}
\end{theorem}
\begin{proof}
\ref{th:conv:alg:nonmono:a} From \eqref{eq:genstract:ineq}, we obtain
\begin{align*}
\widehat{\varrho} \Vert \nabla\fgam{\gf}{p,\varepsilon_k}{\gamma}(x^k)\Vert^{1+\vartheta}\leq 
   \fgam{\gf}{p,\varepsilon_{k}}{\gamma}(x^{k})-\fgam{\gf}{p,\varepsilon_{k+1}}{\gamma}(x^{k+1})
+\varepsilon_{k+1}.
\end{align*}
Together with $\fgam{\gf}{p}{\gamma}(x^*)\leq \fgam{\gf}{p}{\gamma}(x^{N+1})\leq \fgam{\gf}{p, \varepsilon_{N+1}}{\gamma}(x^{N+1})$ for $N\in \Nz$, this ensures
$\sum_{k=0}^{N} \Vert \nabla\fgam{\gf}{p,\varepsilon_k}{\gamma}(x^k)\Vert^{1+\vartheta} \leq 
\frac{\Delta_0}{\widehat{\varrho}}$.
Letting $N\to\infty$ proves the first claim. This implies
$\Vert \nabla\fgam{\gf}{p,\varepsilon_k}{\gamma}(x^k)\Vert\to 0$.
Then, \eqref{eq:relapprox:b} implies $\Vert \nabla\fgam{\gf}{p}{\gamma}(x^k)\Vert\to 0$.
\\
\ref{th:conv:alg:nonmono:b}
Assertion~\ref{th:conv:alg:nonmono:a} implies \eqref{eq:conv:alg:nonmono1}. 
Reusing \eqref{eq:relapprox:b}, we have
$\Vert  \nabla\fgam{\gf}{p}{\gamma}(x^k)\Vert\leq (1+c_p)\Vert \nabla\fgam{\gf}{p,\varepsilon_k}{\gamma}(x^k)\Vert$. Hence,
\[
 \bs\min_{0\leq k\leq N} \Vert \nabla\fgam{\gf}{p}{\gamma}(x^k)\Vert\leq  (1+c_p)\bs\min_{0\leq k\leq N} \Vert \nabla\fgam{\gf}{p,\varepsilon_k}{\gamma}(x^k)\Vert.
\]
Together with \eqref{eq:conv:alg:nonmono1}, this implies that \eqref{eq:conv:alg:nonmono2} holds.
\\
\ref{th:conv:alg:nonmono:c}
By Step~\ref{alg:fram:nonmono:newpoint} of Algorithm~\ref{alg:fram:nonmono}, $x^k\in\mathcal{S}$ for every $k$.
Since $\fgam{\gf}{p}{\gamma}$ is continuous and coercive, $\mathcal{S}$ is compact. Hence $\{x^k\}_{k\in\Nz}$ has cluster points. Combining \eqref{eq:ep-gard-approx1} with $\Vert \nabla\fgam{\gf}{p,\varepsilon_k}{\gamma}(x^k)\Vert\rightarrow 0$ leads to $\Vert x^k-\prox{\gf}{\gamma}{p, \varepsilon_k}(x^k)\Vert\to 0$.
Additionally, from \eqref{eq:ep-approx:dist} and $\delta_k\downarrow 0$, it holds that
$\Vert \prox{\gf}{\gamma}{p, \varepsilon_k}(x^k) - \prox{\gf}{\gamma}{p}(x^k)\Vert\to 0$.
Let $\widehat{x}$ be an arbitrary cluster point, and let $\{x^j\}_{j\in J\subseteq \Nz}$ be a subsequence such that $x^j\to \widehat{x}$.
For this subsequence, we obtain
\[
\Vert \prox{\gf}{\gamma}{p}(x^j) - \widehat{x}\Vert \leq \Vert x^j -  \widehat{x}\Vert+ \Vert x^j -\prox{\gf}{\gamma}{p, \varepsilon_j}(x^j)  \Vert+\Vert  \prox{\gf}{\gamma}{p, \varepsilon_j}(x^j) - \prox{\gf}{\gamma}{p}(x^j)\Vert \to 0. 
\]
Therefore, $\prox{\gf}{\gamma}{p}(x^j)\to \widehat{x}$.
By Fact~\ref{th:level-bound+locally uniform}~$\ref{level-bound+locally uniform2:conv}$, it follows that 
$\widehat{x}\in \prox{\gf}{\gamma}{p}(\widehat{x})$, which implies $\widehat{x}\in \bs{\rm Fix}(\prox{\gf}{\gamma}{p})$.
\end{proof}

\begin{corollary}[Consequences for Algorithms~\ref{alg:ingrad2} and~\ref{alg:first}]
\label{cor:specific:rates}\label{cor:con1}\label{cor:conv:alg:first}
The following statements hold.
\begin{enumerate}[label=(\textbf{\alph*}),font=\normalfont\bfseries,leftmargin=0.7cm]
\item\label{cor:specific:rates:a} Let $\{x^k\}_{k\in\Nz}$ be generated by Algorithm~\ref{alg:ingrad2} under the assumptions of Theorem~\ref{th:pfhigda:well}. Then, conclusions of Theorem~\ref{th:conv:alg:nonmono} hold
with $\vartheta=\frac{2}{p-1}$ and $\widehat{\varrho}=\widehat{\varrho}_1$. In particular,
\[
\bs\min_{0\leq k\leq N} \Vert \nabla\fgam{\gf}{p,\varepsilon_k}{\gamma}(x^k)\Vert=O\left(N^{-\frac{p-1}{p+1}}\right),
\quad 
\bs\min_{0\leq k\leq N} \Vert \nabla\fgam{\gf}{p}{\gamma}(x^k)\Vert=O\left(N^{-\frac{p-1}{p+1}}\right).
\]

\item\label{cor:specific:rates:b} Let $\{x^k\}_{k\in\Nz}$ be generated by Algorithm~\ref{alg:first}, and suppose that $\alpha_k\geq\underline{\alpha}>0$, for $k\in\Nz$. Then, conclusions of Theorem~\ref{th:conv:alg:nonmono} hold
with $\widehat{\varrho} = \underline{\alpha}\lambda A_p$. In particular,
\[
\bs\min_{0\leq k\leq N} \Vert \nabla\fgam{\gf}{p,\varepsilon_k}{\gamma}(x^k)\Vert=O\left(N^{-\frac{1}{1+\vartheta}}\right),
\quad 
\bs\min_{0\leq k\leq N} \Vert \nabla\fgam{\gf}{p}{\gamma}(x^k)\Vert=O\left(N^{-\frac{1}{1+\vartheta}}\right).
\]
\end{enumerate}
\end{corollary}
\begin{proof}
\ref{cor:specific:rates:a} For Algorithm~\ref{alg:ingrad2}, Theorem~\ref{th:pfhigda:well}~\ref{th:pfhigda:well:c} shows that the sufficient-decrease condition of Algorithm~\ref{alg:fram:nonmono} holds with $\vartheta=\frac{2}{p-1}$ and $\widehat{\varrho}=\widehat{\varrho}_1$. Since
$\frac{1}{1+\vartheta}=\frac{p-1}{p+1}$, the result follows from Theorem~\ref{th:conv:alg:nonmono}.
\\
\ref{cor:specific:rates:b}
From  Step~\ref{alg:first:inwh} of Algorithm~\ref{alg:first}, we obtain
\begin{align*}
   \fgam{\gf}{p,\varepsilon_{k+1}}{\gamma}(x^{k+1})\leq \fgam{\gf}{p,\varepsilon_{k}}{\gamma}(x^{k})-\underline{\alpha}\lambda A_p \Vert \nabla\fgam{\gf}{p,\varepsilon_k}{\gamma}(x^k)\Vert^{1+\vartheta}
+\varepsilon_{k+1}.
\end{align*}
Setting $\widehat{\varrho}:=\underline{\alpha}\lambda A_p$ in Theorem \ref{th:conv:alg:nonmono}, the desired results are valid.
\end{proof}

\subsection{Global and linear convergence}\label{sec:globalLinConv}
We next study the global and linear convergence under the KL property.
In the following, the set of cluster points of $\{x^k\}_{k\in \mathbb{N}_0}$ is denoted by $\Omega(\{x^k\})$.

\begin{assumption}\label{assum:eps:weak} 
Let $\{x^k\}_{k\in\Nz}$ be generated by Algorithm~\ref{alg:fram:nonmono}.
\begin{enumerate}[label=(\textbf{\alph*}), font=\normalfont\bfseries, leftmargin=0.7cm]
\item \label{assum:eps:weak:b} The function $\fgam{\gf}{p}{\gamma}$ satisfies the KL property at each point of $\Omega(\{x^k\})\subseteq \Dom{\partial \fgam{\gf}{p}{\gamma}}$
with a desingularizing function $\phi:[0,\eta)\to[0,+\infty)$ that possesses the quasi-additivity property described in 
Fact~\ref{lem:Uniformized KL property}.

\item \label{assum:eps:weak:c}  
Define $w_k:=2\widehat{\varrho}\left(\frac{2^{2-p}}{\gamma}\right)^{1+\vartheta}\sum_{j=k}^{\infty}\delta_j^{(p-1)(1+\vartheta)}+2\sum_{j=k}^{\infty}\varepsilon_j$ with $k\in \Nz$. Since $w_k\to0$, let $k_0$ be sufficiently large that $w_k<\eta$ for $k\geq k_0$, and assume that $\sum_{k=k_0}^{\infty} ([\phi'(w_k)]^{-1})^\vartheta<\infty$.
\end{enumerate}
\end{assumption}

\begin{theorem}[Global convergence under the KL property]\label{th:conKL}
Let Assumptions~\ref{assum:approx}~and~\ref{assum:eps:weak} hold.
Then the sequence $\{x^k\}_{k\in \Nz}$ generated by Algorithm~\ref{alg:fram:nonmono} 
converges to a proximal fixed point.
\end{theorem}
\begin{proof}
See Appendix~\ref{proof:th:conKL}.
\end{proof}


We next impose a stronger condition on the lower-level errors to obtain linear convergence.

\begin{assumption}\label{assum:linconv:weak}
Let $\{\beta_k\}_{k\in \Nz}$ be a sequence of non-increasing positive scalars such that $\ov{\beta}:=\sum_{k=0}^{\infty} \beta_k<\infty$ and  $\widehat{\beta}:=\sum_{k=0}^{\infty} \beta_k^{\frac{1}{\vartheta(p-1)}}<\infty$. We select $\delta_k$ and $\varepsilon_k$ to satisfy the condition
\begin{equation}\label{eq:ep-approx:oper:b}
\varepsilon_k^{\frac{\vartheta}{1+\vartheta}}, \delta_k^{\vartheta(p-1)}\leq 
\mathop{\bs{\min}}\left\{\beta_j\Vert x^i - \prox{\gf}{\gamma}{p, \varepsilon_i}(x^i)\Vert^{\vartheta(p-1)}\mid 0\leq i\leq j\leq k\right\}.
\end{equation}
\end{assumption}

\begin{theorem}[Linear convergence under the KL property]\label{th:linconv}
Let Assumptions~\ref{assum:approx},~\ref{assum:eps:weak}, and~\ref{assum:linconv:weak} hold, and let $\{x^k\}_{k\in \Nz}$ be the sequence generated by Algorithm~\ref{alg:fram:nonmono}.
If $\fgam{\gf}{p}{\gamma}$ satisfies the KL property with an exponent $\theta=\frac{1}{1+\vartheta}$ at each point of $\Omega(\{x^k\})$, then the sequence converges R-linearly to a proximal fixed point.
\end{theorem}
\begin{proof}
See Appendix~\ref{proof:th:linconv}.
\end{proof}

\begin{corollary}[Consequences for Algorithms~\ref{alg:ingrad2} and~\ref{alg:first}]
\label{th:conKL:oth}\label{cor:linconv:gen}
Suppose that Assumption~\ref{assum:approx} holds and that one of the following cases applies:
\begin{enumerate}[label=(\textbf{\alph*}),font=\normalfont\bfseries,leftmargin=0.7cm]

\item\label{cor:KL:specific:a} $\{x^k\}_{k\in\Nz}$ is generated by Algorithm~\ref{alg:ingrad2} under the assumptions of Theorem~\ref{th:pfhigda:well}, with
$\vartheta=\frac{2}{p-1}$ and $\widehat{\varrho}=\widehat{\varrho}_1$.

\item\label{cor:KL:specific:b}
$\{x^k\}_{k\in\Nz}$ is generated by Algorithm~\ref{alg:first} under the assumptions of Theorem~\ref{th:welldf:alg:first}, and
$\alpha_k\geq\underline{\alpha}>0$ for all $k$, with $\widehat{\varrho}=\underline{\alpha}\lambda A_p$. 
\end{enumerate}

If Assumption~\ref{assum:eps:weak} holds, then $\{x^k\}_{k\in\Nz}$ converges to a proximal fixed point.
If, in addition, Assumption~\ref{assum:linconv:weak} holds and $\fgam{\gf}{p}{\gamma}$ has KL exponent $\theta=1/(1+\vartheta)$ on $\Omega(\{x^k\})$, then the convergence is R-linear. In case~\ref{cor:KL:specific:a}, the required exponent is $\theta=\frac{p-1}{p+1}$. 
\end{corollary}
\begin{proof}
The two cases satisfy the generic sufficient-decrease condition by Corollary~\ref{cor:specific:rates}. The conclusions therefore follow from Theorems~\ref{th:conKL} and~\ref{th:linconv}.
\end{proof}

\section{Preliminary numerical experiments}
\label{sec:numerical}
This section examines the numerical behavior of the proposed methods and the role of the HOME order in two settings. The hinge-loss experiment is convex and nonsmooth and uses a primal--dual certificate for the lower problem. The matrix-completion experiment is nonconvex and weakly convex and uses the residual certificate of Corollary~\ref{cor:computable:inhope} to verify every accepted lower-level point. In both cases, parameters are calibrated first and then fixed for separate confirmation instances. 
All experiments are conducted in Python on a laptop equipped with a 12th Gen Intel$\circledR$ Core$^{\text{TM}}$ i7-12800H CPU (1.80 GHz) and 16 GB of RAM.

\subsection{Classification with hinge-loss}
\label{sec:num:svm}
We first consider the ridge-regularized hinge-loss problem
\[
 \mathop{\bs\min}\limits_{x\in \mathbb {R}^n}\ \Phi(x) :=\frac{1}{m}\sum_{i=1}^{m}\bs\max\left\{0,1-b_i a_i^\top x\right\} +\frac{\lambda}{2}\|x\|^2,
\]
where $a_i\in\R^n$ and $b_i\in\{-1,1\}$. This is the standard primal $\ell_2$-regularized support vector machine formulation with hinge loss; see, e.g., \cite{Chapelle2007,ShalevShwartz2011}.
The objective is nonsmooth and $\lambda$-strongly convex, and hence belongs to the convex subclass of the weakly convex setting considered in this paper.

We use $m=192$, $n=24$, $\lambda=10^{-2}$, and $x^0=0$. For each random seed, the rows of $A\in\R^{m\times n}$ are the feature vectors $a_i^\top$. We generate
$A=\sqrt{m}\,U\operatorname{diag}(s_1,\ldots,s_n)V^\top$ with
$s_j=30^{(j-1)/(n-1)}$,
where $U$ has orthonormal columns, and $V$ is orthogonal. A Gaussian vector $w_\star$ is generated,
the scores $q=Aw_\star$ are rescaled to have standard deviation $3$, and the labels are sampled independently according to the model
\[
\mathbb{P}(b_i=1\mid a_i) =\frac{1}{1+\exp(-q_i)}.
\]
For a center $c$, the lower-level problem is
\begin{equation}\label{eq:svm_lower}
 Q_{p,\gamma}(y;c)=\Phi(y)+\frac{1}{p\gamma}\|y-c\|^p.
\end{equation}
Since $\Phi\geq0$, its exact minimizer $z$ belongs to $\mb(c; R)$, where $R=(p\gamma\Phi(c))^{1/p}$. On this ball, $Q_{p,\gamma}(\cdot;c)$ is strongly convex with the modulus
$m_R =\lambda+\frac{p-1}{\gamma R^{2-p}}$. We solve \eqref{eq:svm_lower} by the accelerated primal--dual method \cite{ChambollePock2011}. If $D(v)$ is a feasible dual lower bound and
$ \Delta=Q_{p,\gamma}(y;c)-D(v)$, then
\[
 Q_{p,\gamma}(y;c)-\mathop{\bs\min}\limits_{u} Q_{p,\gamma}(u;c)\leq\Delta,
 \qquad
 \|y-z\|\leq\widehat\delta:=\sqrt{\frac{2\Delta}{m_R}},
\]
where $z$ is the exact lower minimizer.
A lower solution at the outer iteration $k$ is accepted whenever
\[
 \Delta\leq\varepsilon_k, \qquad \widehat\delta\leq\bs\min\{\mu\|y-c\|,(k+1)^{-2}\},
 \qquad  Q_{p,\gamma}(y;c)\leq\Phi(c),
\]
with
\[
 \varepsilon_k= \frac{\varepsilon_{\rm sc}\bs\max\{1,\Phi(x^0)\}}{(k+1)^2},
\]
where $\varepsilon_{\rm sc}$ is a scaling parameter for the accuracy at the lower-level.  
These conditions provide the value and distance guarantees required in Assumption~\ref{assum:approx}; see Corollary~\ref{cor:computable:inhope}.

In all selected ItsDEAL configurations, $\varepsilon_{\rm sc}=10^{-3}$. The initial primal and dual step-sizes of the lower solver are $t_0=\frac{r_0\sqrt{m}}{\|A\|}$ and $\sigma_0=\frac{0.99}{t_0\|A\|^2}$,
where $r_0$ is calibrated together with $\gamma$.
Let us set
\[
g^k  = \frac{1}{\gamma} \|x^k-y^k\|^{p-2}(x^k-y^k),
\]
be the inexact HOME gradient. We choose $\mu$ so that $ c_p:=2^{2-p}\mu^{p-1}=0.95$ 
for both orders, giving $\mu=0.45125$ for $p=1.5$ and $\mu=0.95$ for $p=2$.
For the parameter-free HiGDA (PFHiGDA) in Algorithm~\ref{alg:ingrad2}, we use
$d^k=-\|g^k\|^{(3-p)/(p-1)}g^k$.
The backtracking starts from $L_0=10^{-4}$ and uses the factor $2$. For IDEALS in Algorithm~\ref{alg:first}, we use
\[
d^k  =  -D_p\|g^k\|^{(2-p)/(p-1)}g^k,\qquad  D_p=0.75\,\gamma^{1/(p-1)},
\]
so that $d^k=0.75(y^k-x^k)$. The Armijo parameter and the backtracking factor are both $1/2$.
The parameters are selected on seeds $3701$--$3703$, and the final results are reported on seeds $3801$--$3820$. We use the common stopping condition
\begin{equation}\label{eq:svm_target_common}
\Phi(x) \leq  L+\tau\bigl(\Phi(x^0)-L\bigr),  \qquad  \tau\in\{10^{-2},10^{-5}\},
\end{equation}
where $L\leq\Phi^\star$ is a dual lower bound computed offline. The largest reference primal--dual gap on the $20$ reported instances is $9.6\times10^{-14}$. The condition~\eqref{eq:svm_target_common} is used only as a common time-to-target benchmarking criterion; the certified upper-level stationarity test is given in Remark~\ref{rem:relapprox}~\ref{rem:relapprox:e}.
For $\tau=10^{-2}$, the selected PFHiGDA parameters $(\gamma,r_0)$ are $(10,10)$ for $p=1.5$ and $(3,10)$ for $p=2$.
The selected IDEALS parameters are $(30,10)$ for $p=1.5$ and $(3,10)$ for $p=2$. For $\tau=10^{-5}$, the PFHiGDA settings are unchanged, while IDEALS uses $(\gamma,r_0)=(100,10)$ for $p=1.5$ and $(10,1)$ for $p=2$.
We also compare with normalized subgradient methods. At $\tau=10^{-2}$, the selected geometric rule uses $\alpha_k=0.1(0.99)^k$,
while at $\tau=10^{-5}$ it uses $\alpha_k=(0.999)^k$. The Polyak method uses multiplier $1$ and the same lower bound $L$ appearing
in \eqref{eq:svm_target_common}.

All methods receive the same budget of $2$ process-CPU seconds per instance. Each run is repeated three times. For each instance, we use the median CPU time over the three repetitions and count the instance as solved when at least two repetitions reach the target. A failed run is assigned the full $2$-second budget.
The results for the moderate target $\tau=10^{-2}$ are reported in Table~\ref{tab:svm_tau2}. PFHiGDA with $p=1.5$ solves more instances and is faster than its quadratic counterpart. Both IDEALS variants and both subgradient methods solve all instances, with the subgradient methods being the fastest at
this accuracy. Within IDEALS, $p=2$ is faster than $p=1.5$.

\begin{table}[h]
\centering
\caption{Results on 20 instances for $\tau=10^{-2}$ with a common CPU budget of 2 seconds per instance. Failed runs are assigned the full budget.}
\label{tab:svm_tau2}
\begin{tabular}{lcc}
\toprule
Method & Target reached & Median CPU (s) \\
\midrule
PFHiGDA, $p=1.5$ & $19/20$ & $0.797$ \\
PFHiGDA, $p=2$   & $11/20$ & $1.785$ \\
IDEALS, $p=1.5$  & $20/20$ & $0.115$ \\
IDEALS, $p=2$    & $20/20$ & $0.062$ \\
Geometric subgradient & $20/20$ & $0.033$ \\
Polyak subgradient    & $20/20$ & $\textbf{0.005}$ \\
\bottomrule
\end{tabular}
\end{table}

The results for the higher-accuracy target $\tau=10^{-5}$ are reported in Table~\ref{tab:svm_tau5}. Neither PFHiGDA variant reaches this target within the
CPU budget. IDEALS with $p=1.5$ solves $17/20$ instances, compared with $14/20$ for $p=2$, and also has the lower median capped CPU time. The geometric subgradient method solves $15/20$ instances, while the Polyak method solves only $3/20$.

\begin{table}[h]
\centering
\caption{Results on 20 instances for $\tau=10^{-5}$ with a common CPU budget of 2 seconds per instance. Failed runs are assigned the full budget.}
\label{tab:svm_tau5}
\begin{tabular}{lcc}
\toprule
Method & Target reached & Median CPU (s) \\
\midrule
PFHiGDA, $p=1.5$ & $0/20$  & $2.000$ \\
PFHiGDA, $p=2$   & $0/20$  & $2.000$ \\
IDEALS, $p=1.5$  & $17/20$ & $\textbf{0.886}$ \\
IDEALS, $p=2$    & $14/20$ & $1.348$ \\
Geometric subgradient & $15/20$ & $1.839$ \\
Polyak subgradient    & $3/20$  & $2.000$ \\
\bottomrule
\end{tabular}
\end{table}

Tables~\ref{tab:svm_tau2} and~\ref{tab:svm_tau5} show that the effect of the HOME order depends on the algorithm and the requested accuracy. In these experiments, the subquadratic order performs better for PFHiGDA at the moderate target and for IDEALS at the higher-accuracy target, whereas the quadratic IDEALS realization is faster at the moderate target. Figure~\ref{fig:svm_profiles} complements Tables~\ref{tab:svm_tau2}
and~\ref{tab:svm_tau5} by showing the fraction of confirmation
instances that reach the target as the CPU budget increases.

\begin{figure}[h]
    \centering

    \begin{subfigure}[t]{0.49\textwidth}
        \centering
        \includegraphics[width=\linewidth]{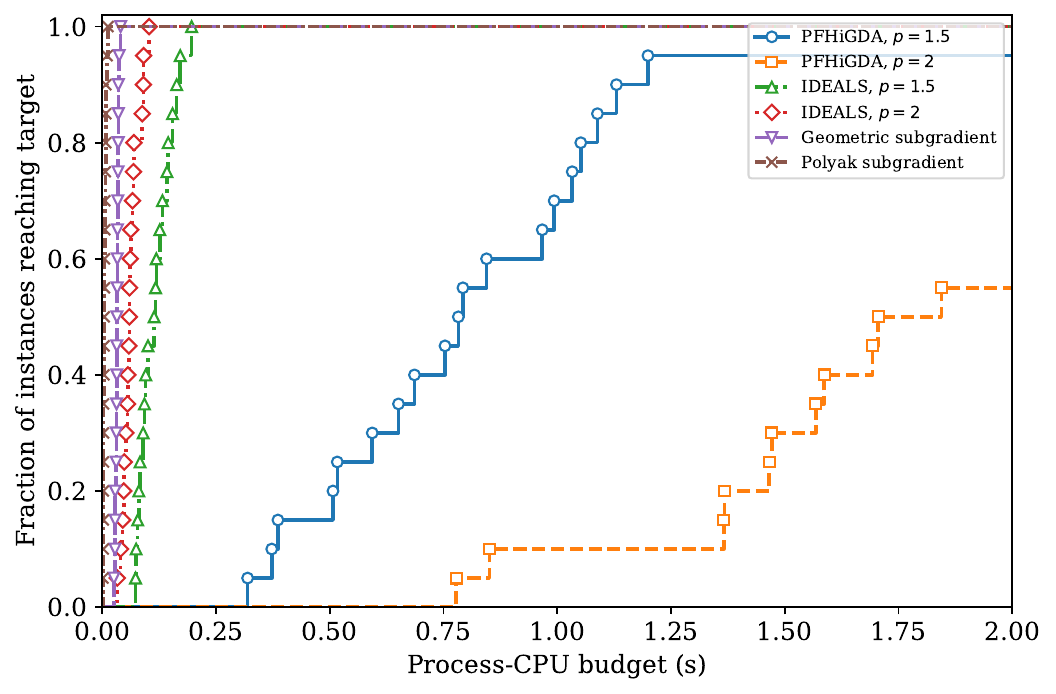}
        \caption{$\tau=10^{-2}$.}\label{fig:svm_profile_tau2}
    \end{subfigure}
    \hfill
    \begin{subfigure}[t]{0.49\textwidth}
        \centering
        \includegraphics[width=\linewidth]{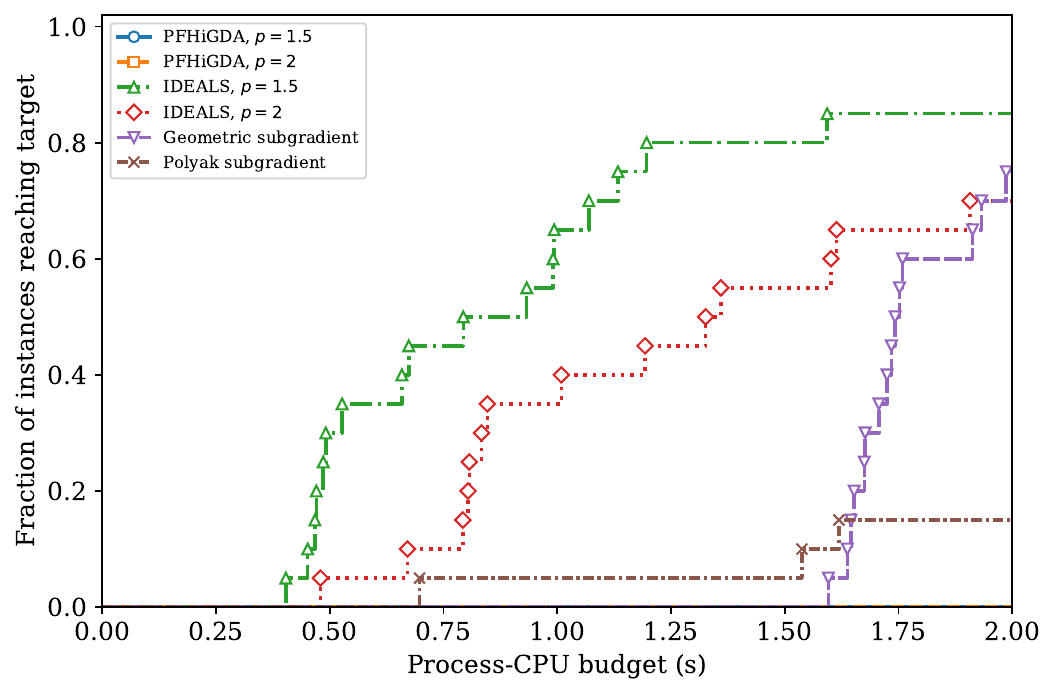}
        \caption{$\tau=10^{-5}$.}
        \label{fig:svm_profile_tau5}
    \end{subfigure}

    \caption{Data profiles for the ridge-regularized hinge-loss experiment. Each curve shows the fraction of the $20$ confirmation instances that reach the target as a function of the process-CPU budget. For each instance and method, the CPU time is the median over three repetitions. }
    \label{fig:svm_profiles}
\end{figure}

Let us now turn to an MCP-regularized hinge-loss problem. Nonconvex penalties, such as MCP and SCAD, are commonly used in sparse high-dimensional SVMs because they can reduce the estimation bias associated with convex sparsity penalties while performing variable selection \cite{Zhang2010MCP,ZhangWuWangLi2016}. Here, we consider the function
\begin{equation}\label{eq:svm:mcp:highdim}
\Phi_{\rm MCP}(x) = \frac1m\sum_{i=1}^m \max\{0,1-b_i a_i^\top x\} +\frac{\lambda_2}{2}\|x\|^2 +\sum_{j=1}^n P_{\lambda_{\rm M},\theta}(x_j),
\end{equation}
with
\[
P_{\lambda,\theta}(t)= \begin{cases} 
\lambda |t|-\dfrac{t^2}{2\theta}, & |t|\leq \theta\lambda,\\[1mm]
\dfrac{\theta\lambda^2}{2}, & |t|>\theta\lambda.
\end{cases}
\]
Since the MCP term is $1/\theta$-weakly convex, \eqref{eq:svm:mcp:highdim} is $\rho$-weakly convex with $\rho=1/\theta-\lambda_2$ when $\lambda_2<1/\theta$.

We use $m=100$ observations and $n=200$ features. The feature covariance is $\Sigma_{ij}=0.5^{|i-j|}$. The true coefficient vector has five nonzero entries $(1.5,-1.2,1,-0.8,0.6)$, with all remaining entries equal to zero. The training features are standardized, and the labels are sampled from a logistic model based on the sparse score. We set $\theta=3$, $\lambda_{\rm M}=0.05$, and $\lambda_2=5\times10^{-3}$. Motivated by the use of suitable initial estimators for nonconvex penalized SVMs \cite{ZhangWuWangLi2016}, every method starts from the same convex $\ell_1$-regularized hinge-loss solution. An independent sample of $5000$ observations is used only to report held-out classification error.

For the HOME methods, the lower problem is certified with Corollary~\ref{cor:computable:inhope} using a computable lower bound for the strongly convex localized proximal model. The lower solver is warm-started from the previous primal and dual variables. We compare IDEALS and PFHiGDA for $p\in\{1.25,1.5,1.75,2\}$ with geometric and normalized subgradient methods. The two subgradient baselines use $\alpha_k=\alpha_0q^k$, with calibrated $(\alpha_0,q)$ equal to $(5\times10^{-4},0.9995)$ and $(2\times10^{-4},0.9995)$, respectively.
All methods start from the same initial point and are run under a common process-CPU budget of $0.30$ seconds. We report the original training objective and the held-out classification error. This experiment examines the finite-budget effect of the HOME order.

\begin{table}[t]
\centering
\small
\caption{High-dimensional MCP-regularized hinge-loss experiment on $20$
confirmation instances under a common process-CPU budget. Lower objective
and test-error values are better.}
\label{tab:svm:mcp:highdim}
\begin{tabular}{lccc}
\toprule
Method & $p$ & Median objective & Median test error \\
\midrule
IDEALS & $1.25$ & $0.2214$ & $0.3415$ \\
IDEALS & $1.5$  & $0.2188$ & $0.3419$ \\
IDEALS & $1.75$ & $0.2050$ & $0.3466$ \\
IDEALS & $2$    & $\mathbf{0.1958}$ & $0.3482$ \\
PFHiGDA & $1.25$ & $0.2214$ & $0.3415$ \\
PFHiGDA & $1.5$  & $0.2210$ & $0.3414$ \\
PFHiGDA & $1.75$ & $0.2201$ & $\mathbf{0.3413}$ \\
PFHiGDA & $2$    & $0.2169$ & $0.3421$ \\
Geometric subgradient & -- & $0.2119$ & $0.3427$ \\
Normalized subgradient & -- & $0.2158$ & $0.3427$ \\
\bottomrule
\end{tabular}
\end{table}

Table~\ref{tab:svm:mcp:highdim} shows that IDEALS with $p=2$ gives the lowest training objective, whereas the subquadratic choices give lower held-out error under the same computational budget. In particular, IDEALS with $p=1.5$ has lower test error than $p=2$ on $18$ of the $20$ confirmation instances. A similar effect appears for PFHiGDA, for which the subquadratic orders remain closer to the predictive behavior of the common sparse warm start while decreasing the MCP objective. The direct subgradient methods decrease the objective efficiently but do not improve the held-out error relative to the best subquadratic HOME variants in this regime.

\subsection{Robust matrix completion}
We consider a nonsmooth robust low-rank matrix-completion problem. Let $\Omega\subseteq\{1,\ldots,n\}^2$ denote the set of observed indices,
and let $Y_{ij}$ be the observed value at $(i,j)\in\Omega$. The variable $X\in\R^{n\times n}$ is the matrix to be estimated. We consider
\begin{equation}\label{eq:mc:l1:obj}
\Phi(X) = \frac{1}{|\Omega|}\sum_{(i,j)\in\Omega}|X_{ij}-Y_{ij}|+ \lambda R_a(X),
\end{equation}
where
\[
R_a(X) = \frac{a^2}{2}\sum_{\ell}\log\left(1+\frac{\sigma_\ell(X)^2}{a^2}\right),
\]
$\sigma_\ell(X)$ denotes the $\ell$-th singular value of $X$, and $a>0$ controls the shape of the rank surrogate. The first term in \eqref{eq:mc:l1:obj} is convex and nonsmooth. Related nonsmooth $\ell_1$ formulations are used in robust low-rank recovery; see, e.g., \cite{Li20Low}. The spectral term is a smooth nonconvex log-det-type rank surrogate motivated by \cite{FazelHindiBoyd2003}.

Let us set
\[
h_a(t)=\frac{a^2}{2}\log\left(1+\frac{t^2}{a^2}\right).
\]
Since
\[
h_a''(t) = \frac{a^2(a^2-t^2)}{(a^2+t^2)^2}\geq -\frac18,
\]
the function $t\mapsto h_a(t)+t^2/16$ is convex and nondecreasing on $[0,\infty)$. Hence $R_a(X)+\frac1{16}\|X\|_F^2$ is convex. Therefore, $R_a$ is $1/8$-weakly convex, and \eqref{eq:mc:l1:obj} is $\rho$-weakly convex with $\rho=\frac{\lambda}{8}$.

We generate $20\times20$ matrices of rank $4$ with singular values $8,6,4,3$, and rescale each planted matrix to have unit standard deviation. Each entry is observed independently with probability $0.4$. Gaussian noise with standard deviation $0.02$ is added to the observations, and approximately $15\%$ of the observed entries receive an additional Student-$t_2$ perturbation with scale $4$. We set $\lambda=0.05$ and $a=1.2$.

All methods use the same deliberately rank-underestimated warm start. Let $Q_{0.6}$ denote the empirical $60\%$ quantile of the absolute observed entries and set
$\tau_{\rm clip} = \max\left\{0.5,\,1.5Q_{0.6}\bigl(\{|Y_{ij}|:(i,j)\in\Omega\}\bigr)\right\}$.
The observed entries are clipped at $\tau_{\rm clip}$, and the resulting sparse matrix is scaled by the inverse realized observation fraction. Its rank-$2$ truncated SVD is used to initialize a rank-$2$ factorization. We then refine this factorization for at most $80$ L-BFGS-B iterations using a pseudo-Huber data term and ridge coefficient $0.02$. The resulting rank-$2$ matrix is the common initial point $X^0$. 

At an outer center $C$, the lower problem is 
\[
Q_{p,\gamma}(Y;C) = \Phi(Y)+\frac1{p\gamma}\|Y-C\|_F^p.
\]
We solve this problem by forward--backward splitting. The smooth part is $\lambda R_a$, while the nonsmooth proximal part is
\[
\frac1{|\Omega|} \sum_{(i,j)\in\Omega}|Y_{ij}-Y_{ij}^{\rm obs}|+\frac1{p\gamma}\|Y-C\|_F^p.
\]
For a fixed radial coefficient $\beta=\gamma^{-1}\|Y-C\|_F^{p-2}$, the corresponding proximal step reduces to entrywise soft-thresholding on the observed entries. The
coefficient $\beta$ is obtained from a one-dimensional scalar equation.

The internal stopping test of the lower solver is not used as a certificate. Let $Z$ be an exact minimizer of the lower model. Since $\Phi\geq 0$,
\[
Q_{p,\gamma}(Z;C) \leq Q_{p,\gamma}(C;C) = \Phi(C),
\]
and therefore
\[
\|Z-C\|_F \leq R_C:=\bigl(p\gamma\Phi(C)\bigr)^{1/p}.
\]
We also require the candidate $Y$ to satisfy $Q_{p,\gamma}(Y;C)\leq\Phi(C)$. By Proposition~\ref{prop:proxmodel:strong}, a valid strong-convexity modulus
of the lower model on this ball is
\[
m_C = \frac{p-1}{\gamma R_C^{2-p}}-\rho, \qquad 1<p<2,
\]
and $m_C=1/\gamma-\rho$ for $p=2$. The certificate is used only when $m_C>0$. This condition was verified at every accepted outer call in the reported experiments.
For a candidate $Y$, we choose $\xi_Y\in\partial\Phi(Y)$ so that the norm of 
\[
r_{\gamma,p}(Y;C) = \xi_Y+ \frac1\gamma\|Y-C\|_F^{p-2}(Y-C),
\]
is minimized. The separability of the absolute-loss term makes this choice explicit on the observed entries. Corollary~\ref{cor:computable:inhope}
gives the computable bounds
\[
\widehat\delta = \frac{\|r_{\gamma,p}(Y;C)\|_F}{m_C}, \qquad \widehat\varepsilon = \frac{\|r_{\gamma,p}(Y;C)\|_F^2}{2m_C}.
\]
The candidate is accepted only if
\begin{equation}\label{eq:mc:l1:certificate}
\widehat\varepsilon\leq\varepsilon_k^{\rm loc},
\qquad 
\widehat\delta \leq \min\{\mu\|Y-C\|_F,\delta_k\}, \qquad
Q_{p,\gamma}(Y;C)\leq\Phi(C),
\end{equation}
i.e., every lower point used by the outer algorithms satisfies the computable sufficient conditions of Corollary~\ref{cor:computable:inhope}.

The implementation uses
\[
\varepsilon_k^{\rm loc} = \frac{10^{-3}\max\{1,\Phi(C)\}}{(k+1)^2}, \qquad \delta_k=(k+1)^{-2}.
\]
For all reported runs, we verified that the objective along the accepted outer sequence did not exceed its initial value. Hence, every accepted lower solution also satisfies Assumption~\ref{assum:approx} with the fixed summable sequence
\[
\varepsilon_k = \frac{10^{-3}\bs\max\{1,\Phi(X^0)\}}{(k+1)^2}.
\]
We choose $\mu$ so that $c_p=0.95$,
which gives $\mu=0.45125$, $0.74123$, and $0.95$ for $p=1.5,1.75,2$, respectively.

Let $Y^k$ be the certified lower point at $X^k$ and define the inexact HOME gradient
\[
g^k = \frac1\gamma \|X^k-Y^k\|_F^{p-2}(X^k-Y^k).
\]
For IDEALS, we use
\[
d^k = -D_p\|g^k\|_F^{\frac{2-p}{p-1}}g^k, \qquad D_p=D_{\rm scale}\gamma^{1/(p-1)}.
\]
Then
\[
\langle g^k,d^k\rangle = -D_p\|g^k\|_F^{p/(p-1)}, \qquad
\|d^k\|_F = D_p\|g^k\|_F^{1/(p-1)},
\]
i.e., the direction condition of Algorithm~\ref{alg:first} holds with $\vartheta=1/(p-1)$ and $c_1=c_2=D_p$. Moreover,
\[
Y^k-X^k = -\gamma^{1/(p-1)}\|g^k\|_F^{\frac{2-p}{p-1}}g^k,
\]
and therefore $d^k=D_{\rm scale}(Y^k-X^k)$. The line search starts from $\alpha=1$ and halves the step when necessary. Each trial point is evaluated using a newly certified lower solve.
For PFHiGDA, we use
\[
d^k = -\|g^k\|_F^{\frac{3-p}{p-1}}g^k.
\]
consequently,
\[
\langle g^k,d^k\rangle = -\|g^k\|_F^{\frac{p+1}{p-1}}, \qquad \|d^k\|_F = \|g^k\|_F^{\frac{2}{p-1}},
\]
which is the direction used in Algorithm~\ref{alg:ingrad2}. We use the adaptive step-size rule of Algorithm~\ref{alg:ingrad2}; each trial point must also satisfy \eqref{eq:mc:l1:certificate}.

We also compare with two subgradient baselines applied directly to \eqref{eq:mc:l1:obj}. The first uses
\[
X^{k+1} = X^k-\alpha_0q^kG^k,\qquad G^k\in\partial\Phi(X^k),
\]
which is a standard baseline for nonsmooth weakly convex robust low-rank recovery \cite{Li20Low}. The second uses the same geometrically decreasing step after normalizing the subgradient direction.

Parameters are selected using only seeds $71,\ldots,80$ by minimizing the median value of the original objective \eqref{eq:mc:l1:obj}. Ground-truth recovery error is not used for parameter selection. The selected parameters are frozen before evaluation on the disjoint confirmation seeds $201,\ldots,220$. The IDEALS pairs $(\gamma,D_{\rm scale})$ are $(15,3)$, $(30,1.5)$, and $(30,1.5)$ for $p=1.5,1.75,2$, respectively. The selected PFHiGDA pairs $(\gamma,L_0)$ are $(15,10^{-3})$, $(30,10^{-3})$, and $(30,10^{-3})$. The geometric subgradient method uses $(\alpha_0,q)=(10,0.8)$, while the normalized variant uses $(0.5,0.9)$.

The experiment is designed to examine the early correction of a rank-underestimated initialization. We therefore fix, before evaluating the confirmation instances, a common budget of two updates per method. For the ItsDEAL methods, these are two updates per method. Process CPU time is reported using the median of three repetitions per method and instance.

For a computed matrix $X$, we report the relative recovery error
\[
\operatorname{Rec}(X)=\frac{\|X-M_\star\|_F}{\|M_\star\|_F},
\]
where $M_\star$ is the planted rank-$4$ matrix.

Table~\ref{tab:mc:l1} summarizes the results after two updates per method.
All methods reduce the median recovery error from the common warm start. Among the IDEALS realizations, $p=1.75$ gives the lowest median recovery error and objective value. Relative to $p=2$, it has a lower objective value on $19$ of the $20$ confirmation instances and  a lower recovery error on $12$ instances. On the unobserved entries, its median relative recovery error is $0.8704$, compared with $0.8710$ for $p=2$, and it is lower on $14$ of the $20$ instances. Thus, the advantage over the quadratic order is modest in recovery error but highly consistent in the original objective.
IDEALS with $p=1.75$ has lower full recovery error and lower objective value than both calibrated subgradient methods on all $20$ confirmation instances. On the unobserved entries, it is better on $17$ of the $20$ instances. The median recovery correction from the common warm start is $0.0195$ for IDEALS with $p=1.75$, compared with $0.0103$ for the geometric subgradient method and $0.0089$ for the normalized variant.

\begin{table}[h]
\centering
\caption{Median results over the $20$ confirmation instances after two updates per method. CPU time is the process CPU time based on three
repetitions per method and instance. Lower values are better.}
\label{tab:mc:l1}
\begin{tabular}{lcccc}
\toprule
Method & $p$ & Recovery error & Objective & CPU (s)\\
\midrule
Initial rank-$2$ warm start & --   & 0.8014 & 1.3224 & --\\
IDEALS                      & 1.5  & 0.7846 & 1.2634 & 0.0111\\
IDEALS                      & 1.75 & \textbf{0.7819} & \textbf{1.2565} & 0.0090\\
IDEALS                      & 2    & 0.7824 & 1.2571 & 0.0046\\
PFHiGDA                     & 1.5  & 0.7967 & 1.2954 & 0.0084\\
PFHiGDA                     & 1.75 & 0.7894 & 1.2676 & 0.0095\\
PFHiGDA                     & 2    & 0.7901 & 1.2687 & 0.0048\\
Geometric subgradient       & --   & 0.7911 & 1.2790 & \textbf{0.00049}\\
Normalized subgradient      & --   & 0.7924 & 1.2843 & \textbf{0.00049}\\
\bottomrule
\end{tabular}
\end{table}

Figure~\ref{fig:mc:l1:trajectories} complements Table~\ref{tab:mc:l1} by showing how these differences develop over the two updates per method. Starting from the common rank-$2$ warm start, the IDEALS realizations produce a larger early correction than the two subgradient baselines. Among the IDEALS variants, $p=1.75$ finishes slightly below $p=2$ in both median recovery error and objective value.

\begin{figure}[h]
    \centering

    \begin{subfigure}[t]{0.49\textwidth}
        \centering
        \includegraphics[width=\linewidth]{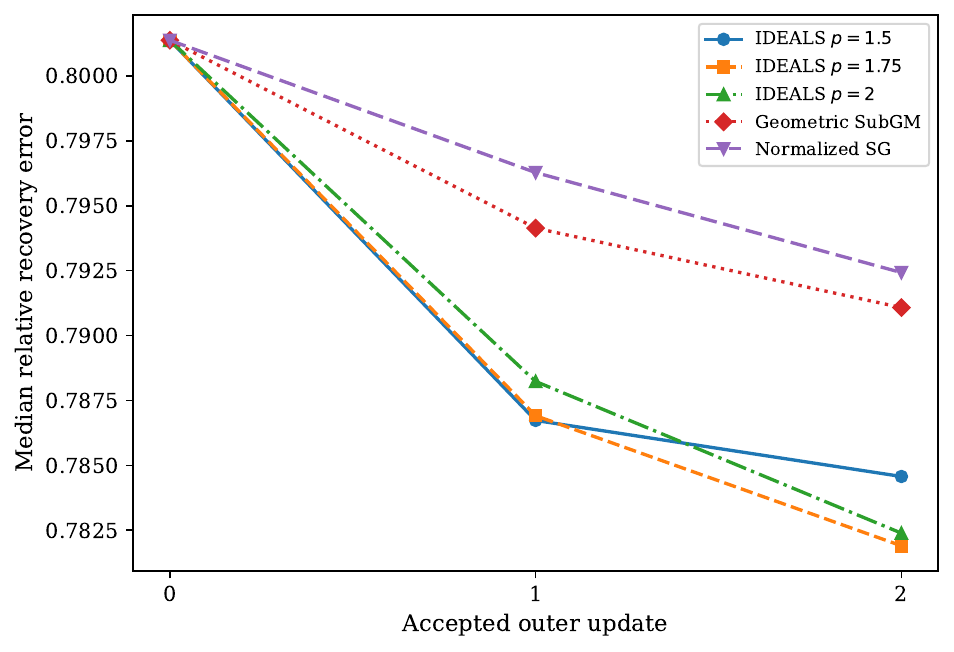}
        \caption{Median relative recovery error.}
        \label{fig:mc_recovery}
    \end{subfigure}
    \hfill
    \begin{subfigure}[t]{0.49\textwidth}
        \centering
        \includegraphics[width=\linewidth]{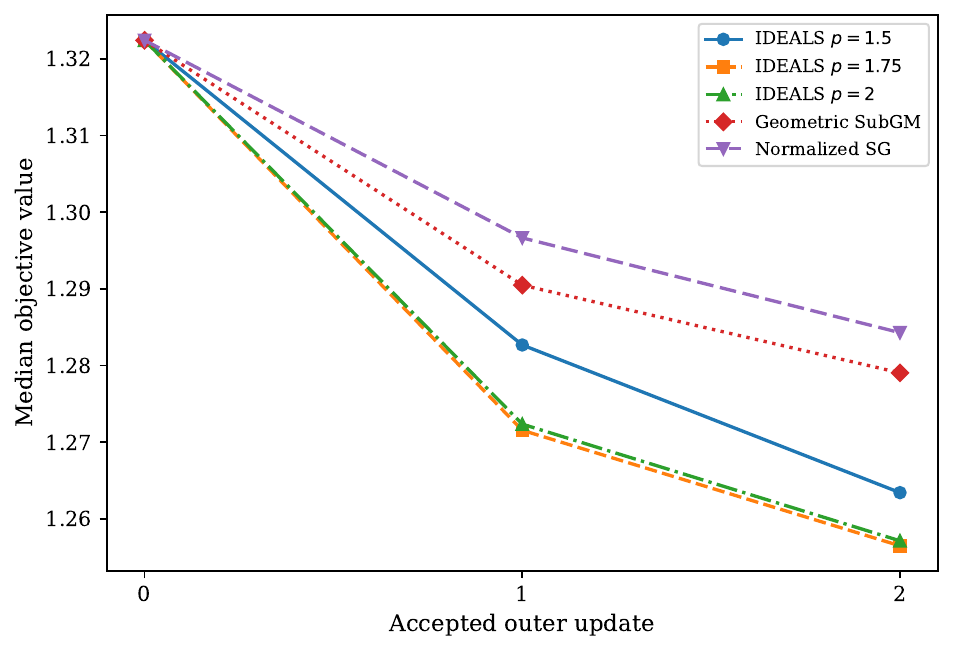}
        \caption{Median objective value.}
        \label{fig:mc_objective}
    \end{subfigure}

    \caption{Median trajectories over the $20$ fresh confirmation instances for the nonsmooth robust matrix-completion experiment. All methods start from the same rank-$2$ warm start and are run
    for two updates per method.}
    \label{fig:mc:l1:trajectories}
\end{figure}

PFHiGDA shows a smaller but similar subquadratic effect. The $p=1.75$ variant has lower recovery error than $p=2$ on $15$ of the $20$ confirmation instances and lower objective value on $19$ instances. These results indicate that the HOME order can affect the finite-work behavior of both upper-level methods, although the preferred order remains algorithm-dependent.

The CPU values in Table~\ref{tab:mc:l1} should be interpreted only as a measure of the per-update cost. A direct subgradient update is much cheaper than a certified HOME-based outer update on these small matrices. Longer-run checks show that the subgradient methods can catch up when allowed many more of their inexpensive updates. We therefore interpret this experiment as a finite-update comparison, i.e., allowing $p<2$ provides a useful additional 
algorithmic degree of freedom, and the certified IDEALS steps can produce a larger early correction from a deliberately rank-underestimated initialization. The experiment does not establish CPU-time superiority of the proposed methods.

\section{Concluding remarks}\label{sec:conc}
In this work, we introduced ItsDEAL, an Inexact two-level smoothing DEscent ALgorithm, to address weakly convex optimization problems. Indeed, this framework leverages the high-order Moreau envelope (HOME) to create a suitable smooth approximation of such a nonsmooth and nonconvex cost function. Crucially, new results to address the differentiability and weak smoothness of HOME under weak convexity assumptions were derived, paving the way toward developing first- or second-order optimization methods. By solving the underlying proximity problem approximately, an inexact oracle for HOME is provided. Further, a novel descent condition using this inexact oracle was established that is particularly well-suited to select proper descent directions for the minimization of weakly convex functions.

A key strength of the ItsDEAL framework is its versatility. We have demonstrated its applicability by developing specific inexact first-order methods, namely, a H\"{o}lderian inexact gradient descent algorithm (HiGDA) with both fixed and dynamically adaptive step-sizes, and an inexact descent algorithm with Armijo line search (IDEALS). These algorithms are specifically tailored to minimize the smooth approximations generated by HOME using its inexact oracle.
The numerical behavior of the ItsDEAL framework was examined on ridge-regularized hinge-loss classification, high-dimensional MCP-regularized classification, and robust matrix completion. The experiments show that the HOME order can materially affect the behavior of the proposed methods. In the convex ridge experiment, IDEALS with $p=1.5$ is more competitive at the higher-accuracy target. In the nonconvex MCP experiment, the quadratic order decreases the training objective more aggressively, while subquadratic orders preserve better held-out classification performance under the short common CPU budget. The matrix-completion experiment shows an early-update advantage for $p=1.75$, although direct subgradient updates remain much cheaper. These results do not establish uniform CPU-time superiority over subgradient methods. They instead illustrate that the HOME order provides a meaningful algorithmic degree of freedom, while the main contribution of the paper remains the theoretical framework and the computable lower-level certification.

\section*{Funding statement} The Research Foundation Flanders (FWO) research project G081222N and UA BOF DocPRO4 projects with ID 46929 and 48996 partially supported the paper's authors.


\appendix

\section{Deferred proofs of Section~\ref{sec:prelim}}
\subsection{Proof of Lemma~\ref{lem:findlowbounknu:lemma}}
\label{proof:lem:findlowbounknu:lemma}
\begin{proof}
For $1<p<2$, the proof of \cite[Lemma A.1]{Foghem2025} gives
\[
\left\langle \Vert a\Vert^{p-2}a-\Vert b\Vert^{p-2}b,\, a-b \right\rangle \geq (p-1)\Vert a-b\Vert^2 \int_0^1 \Vert ta+(1-t)b\Vert^{p-2}\,dt.
\]
Since $\Vert a\Vert\leq r$ and $\Vert b\Vert\leq r$, we have $\Vert ta+(1-t)b\Vert\leq r$ for all $t\in[0,1]$. Hence, for almost every $t\in[0,1]$, $\Vert ta+(1-t)b\Vert^{p-2}\geq r^{p-2}$.
If the segment joining $a$ and $b$ passes through the origin, the corresponding singularity is integrable since $p-2>-1$. Thus,
\[
\int_0^1\Vert ta+(1-t)b\Vert^{p-2}\,dt \geq r^{p-2},
\]
which proves the result. For $p=2$, the result is immediate.
For $p=1$, define $T_1(x)=x/\|x\|$ for $x\neq0$ and $T_1(0)=0$. For every fixed $x\in\R^n$,
\[
\lim_{q\downarrow1}\|x\|^{q-2}x=T_1(x).
\]
Applying the established inequality with $q\in(1,2]$ gives
\[
\left\langle \|a\|^{q-2}a-\|b\|^{q-2}b,a-b \right\rangle\geq (q-1)r^{q-2}\|a-b\|^2.
\]
Letting $q\downarrow1$ yields
\[
\langle T_1(a)-T_1(b),a-b\rangle\geq 0,
\]
which is precisely the asserted inequality for $p=1$.
\end{proof}

\section{Deferred proofs of Section~\ref{sec:home}}\label{app:proof}

\subsection{Proof of Theorem~\ref{th:dif:dif}}
\label{proof:th:dif:dif}
 \begin{proof}
   By Fact~\ref{th:level-bound+locally uniform}~$\ref{level-bound+locally uniform:proxnonemp}$, the proximal set $\prox{\gf}{\gamma}{p}(x)$ is nonempty. Let $x_1, x_2 \in \mb(0; r)$, and choose $y_j \in \prox{\gf}{\gamma}{p}(x_j)$ for $j = 1, 2$. From Fact~\ref{pro:weakConvChar}~\ref{pro:weakConvChar:d} and $\frac{1}{\gamma}\Vert x_j-y_j\Vert^{p-2}(x_j-y_j) \in \partial \gf(y_j)$  for $j = 1, 2$, it holds that
\begin{equation*}
\gf(y_i)\geq \gf(y_j)+\frac{1}{\gamma}\Vert x_j-y_j\Vert^{p-2}\langle x_j-y_j , y_i-y_j\rangle-\frac{\rho}{2}\Vert y_i-y_j\Vert^2,
\end{equation*}
for $i, j=1, 2$ with $i\neq j$, 
which implies
\begin{align}\label{eq:a:th:dif:dif}
\langle \Vert x_2-y_2\Vert^{p-2} (x_2-y_2)- \Vert x_1-y_1\Vert^{p-2} (x_1-y_1) , y_2-y_1\rangle&\geq  -\gamma\rho\Vert y_1-y_2\Vert^2.
\end{align}
Since $\gamma<\gamma_{\bs\max}$, we have 
$\tau= \left(\frac{2r^p+p\gamma(\gf(0)-\ell_0)}{2^{1-p}-\ell p\gamma}\right)^{\frac{1}{p}}\leq \widehat{\tau}$. Thus, by Proposition~\ref{prop:findtau:lip},
$\Vert y_i\Vert\leq \tau\leq \widehat{\tau}$.
Consequently, $\Vert x_i -y_i\Vert\leq r_1:=r+\widehat{\tau}$ (for $i=1,2$) and $\Vert y_2-y_1\Vert\leq 2\widehat{\tau}$.
Setting $r_2:=(p-1) r_1^{p-2}$, Lemma~\ref{lem:findlowbounknu:lemma} yields
\begin{align}\label{eq:b:th:dif:dif}
\langle \Vert x_2-y_2\Vert^{p-2}(x_2-y_2) - \Vert x_1-y_1\Vert^{p-2}(x_1-y_1), (x_2-y_2)-(x_1-y_1)\rangle&\geq  r_2\Vert (x_2-y_2) - (x_1-y_1)\Vert^{2}.
\end{align}
Adding the inequalities \eqref{eq:a:th:dif:dif} and \eqref{eq:b:th:dif:dif}, we come to
\begin{align}\label{eq:c:th:dif:dif}
 r_2\Vert (x_2-y_2) - (x_1-y_1)\Vert^{2} -\gamma\rho\Vert y_1-y_2\Vert^2
 &\leq \langle \Vert x_2-y_2\Vert^{p-2}(x_2-y_2) - \Vert x_1-y_1\Vert^{p-2}(x_1-y_1), x_2-x_1\rangle\nonumber\\
 &\leq \left(\Vert x_2-y_2\Vert^{p-1}+\Vert x_1-y_1\Vert^{p-1}\right)\Vert x_2-x_1\Vert.
\end{align}
In addition, it holds that
\[
\Vert (x_2-y_2) - (x_1-y_1)\Vert^{2}\geq \Vert y_2-y_1\Vert^{2}-2\Vert y_2-y_1\Vert\Vert x_2-x_1\Vert\geq \Vert y_2-y_1\Vert^2-4\widehat{\tau}\Vert x_2-x_1\Vert.
\]
Together with \eqref{eq:c:th:dif:dif}, this implies
\[(r_2-\gamma\rho)\Vert y_2 - y_1\Vert^{2}\leq  (4 r_2\widehat{\tau}+2r_1^{p-1}) \Vert x_2-x_1\Vert.\]
Since $\gamma< \frac{p-1}{\rho}(r+\widehat{\tau})^{p-2}$, we get $r_2-\rho\gamma>0$, leading to
\begin{equation}\label{lochol:weak}
\Vert y_2 - y_1\Vert\leq L_p  \Vert x_1-x_2\Vert^{\frac{1}{2}},
\end{equation}
where $L_p:=\left(\frac{4 r_2\widehat{\tau}+2r_1^{p-1}}{r_2-\rho\gamma}\right)^{\frac{1}{2}}$.
Therefore, \eqref{lochol:weak} confirms the single-valuedness and continuity of $\prox{\gf}{\gamma}{p}(x)$ for any $x\in \mb(0; r)$.  By Fact~\ref{th:diffcharact}, we conclude that
$\fgam{\gf}{p}{\gamma}\in \mathcal{C}^{1}(\mb(0; r))$.
 \end{proof}

\subsection{Proof of Theorem~\ref{th:dif:weak}}
\label{proof:th:dif:weak}
\begin{proof}
$\ref{th:dif:weak:a}$ 
We obtain the result directly from \eqref{lochol:weak},
where an admissible constant is
\[
L_p:=\left(\frac{4 (p-1) (r+\widehat{\tau})^{p-2}\widehat{\tau}+2(r+\widehat{\tau})^{p-1}}{(p-1) (r+\widehat{\tau})^{p-2}-\rho\gamma}\right)^{\frac{1}{2}}.
\]
This proves the H\"older continuity of $\prox{\gf}{\gamma}{p}$.
\\
$\ref{th:dif:weak:b}$
Using \cite[Theorem~6.3 ]{Rodomanov2020} and \eqref{lochol:maineq:weak}, for $x_1, x_2 \in  \mb(0; r)$, we get
  \begin{align*}
 \left\Vert \nabla \fgam{\gf}{p}{\gamma}(x_2)-\nabla \fgam{\gf}{p}{\gamma}(x_1)\right\Vert &=\left\Vert \nabla\left(\frac{1}{p\gamma}\Vert\cdot\Vert^{p}\right)\left(x_2 - \prox{\gf}{\gamma}{p}(x_2)\right)-\nabla\left(\frac{1}{p\gamma}\Vert\cdot\Vert^{p}\right)\left(x_1 - \prox{\gf}{\gamma}{p}(x_1)\right)\right\Vert\\
 &\leq \frac{2^{2-p}}{\gamma} \left\Vert (x_2-x_1) - (\prox{\gf}{\gamma}{p}(x_2)-\prox{\gf}{\gamma}{p}(x_1))\right\Vert^{p-1},
  \\
 &\leq \frac{2^{2-p}}{\gamma}\left(\Vert x_2-x_1\Vert + \Vert \prox{\gf}{\gamma}{p}(x_2)-\prox{\gf}{\gamma}{p}(x_1)\Vert\right)^{p-1}\\
 &\leq  \frac{2^{2-p}}{\gamma}\left(\Vert x_2-x_1\Vert^{\frac{1}{2}}\Vert x_2-x_1\Vert^{\frac{1}{2}} +L_p  \Vert x_2-x_1\Vert^{\frac{1}{2}}\right)^{p-1}\\
&\leq \frac{2^{2-p}}{\gamma}\left((2r)^{\frac{1}{2}}+L_p \right)^{p-1} \Vert x_2-x_1\Vert^{\frac{p-1}{2}},
  \end{align*}
yielding our desired result with $\mathcal{L}_p:= \frac{2^{2-p}}{\gamma}\left((2r)^{\frac{1}{2}}+L_p\right)^{p-1}$.
\end{proof}

\section{Deferred proofs of Section~\ref{sec:itsopt}}\label{app:proof:sec:itsopt}

\subsection{Proof of Theorem~\ref{th:welldfalg}}
\label{proof:th:welldfalg}
\begin{proof}
Since $x^0\in \mb(0; R)\subseteq \mb(0; r)$, the function $\fgam{\gf}{p}{\gamma}$ is differentiable at $x^0$. 
From Algorithm~\ref{alg:ingrad}, we obtain $x^{1}=x^0+\alpha d^0$, i.e., 
\begin{align}\label{eq:b:th:welldfalg}
\Vert x^{1}\Vert &\leq \Vert x^0\Vert +\alpha\Vert d^0\Vert
\nonumber\\& \leq  \Vert x^0\Vert +\gamma^{\frac{2}{p-1}}\Vert \nabla\fgam{\gf}{p,\varepsilon_0}{\gamma}(x^0)\Vert^{\frac{2}{p-1}}
\nonumber\\& =  \Vert x^0\Vert +\gamma^{\frac{2}{p-1}}\gamma^{\frac{2}{1-p}}\Vert  x^0-\prox{\gf}{\gamma}{p, \varepsilon_0}(x^0)\Vert^2
=  \Vert x^0\Vert +\Vert  x^0-\prox{\gf}{\gamma}{p, \varepsilon_0}(x^0)\Vert^2.
\end{align}
Moreover, by \eqref{eq:ep-approx:dist}, we get
\begin{align*}
\Vert x^0-\prox{\gf}{\gamma}{p, \varepsilon_0}(x^0)\Vert
&\leq \Vert x^0-\prox{\gf}{\gamma}{p}(x^0)\Vert+\Vert\prox{\gf}{\gamma}{p}(x^0)-\prox{\gf}{\gamma}{p, \varepsilon_0}(x^0)\Vert\\
&\leq \Vert x^0-\prox{\gf}{\gamma}{p}(x^0)\Vert+\mu \Vert x^0-\prox{\gf}{\gamma}{p, \varepsilon_0}(x^0)\Vert.
\end{align*}
Together with Corollary~\ref{cor:findtau:lip}, this ensures
\begin{equation}\label{eq:a:th:welldfalg}
\Vert x^0-\prox{\gf}{\gamma}{p, \varepsilon_0}(x^0)\Vert \leq \frac{\Vert x^0-\prox{\gf}{\gamma}{p}(x^0)\Vert}{1-\mu} 
\leq \frac{\Vert x^0\Vert+\Vert \prox{\gf}{\gamma}{p}(x^0)\Vert}{1-\mu} < \frac{R+\ov{\tau}}{1-\mu}.
\end{equation}
From \eqref{eq:b:th:welldfalg} and \eqref{eq:a:th:welldfalg}, we obtain
$\Vert x^{1}\Vert <R+\left(\frac{R+\ov{\tau}}{1-\mu}\right)^2$.
Hence, by Assumption~\ref{assum:rgamma2}, $x^{1}\in \mb(0; r)$ and $\fgam{\gf}{p}{\gamma}$ is also differentiable at $x^1$.
As such, it follows from Remark~\ref{rem:relapprox}~$\ref{rem:relapprox:b}$~and~$\ref{rem:relapprox:c}$ that
\begin{align*}
\fgam{\gf}{p,\varepsilon_1}{\gamma}(x^{1})&
\leq \fgam{\gf}{p,\varepsilon_0}{\gamma}(x^0)+\langle \nabla\fgam{\gf}{p}{\gamma}(x^0), x^{1}-x^0\rangle +\frac{2\mathcal{L}_p}{p+1}\Vert x^{1}-x^0 \Vert ^{\frac{p+1}{2}}+\varepsilon_{1}
\\&= \fgam{\gf}{p,\varepsilon_0}{\gamma}(x^0) +\alpha
\langle \nabla\fgam{\gf}{p}{\gamma}(x^0), d^0\rangle 
+\frac{2\mathcal{L}_p}{p+1} \alpha^{\frac{p+1}{2}}\Vert d^0\Vert ^{\frac{p+1}{2}}+\varepsilon_{1}
\\&\leq \fgam{\gf}{p,\varepsilon_0}{\gamma}(x^0)+\alpha \left(
2^{2-p}\mu^{p-1} \Vert \nabla\fgam{\gf}{p,\varepsilon_0}{\gamma}(x^0)\Vert  \Vert d^0\Vert
 +\langle\nabla\fgam{\gf}{p,\varepsilon_0}{\gamma}(x^0), d^0\rangle\right)+\frac{2\mathcal{L}_p}{p+1} \alpha^{\frac{p+1}{2}}\Vert d^0\Vert ^{\frac{p+1}{2}}+\varepsilon_{1}
 \\&\leq \fgam{\gf}{p,\varepsilon_0}{\gamma}(x^0)+\alpha \left(
2^{2-p}\mu^{p-1}c_2 \Vert \nabla\fgam{\gf}{p,\varepsilon_0}{\gamma}(x^0)\Vert^{\frac{p+1}{p-1}}
 -c_1  \Vert \nabla\fgam{\gf}{p,\varepsilon_0}{\gamma}(x^0)\Vert^{\frac{p+1}{p-1}}\right)
 \\&\qquad+\frac{2\mathcal{L}_p}{p+1} \alpha^{\frac{p+1}{2}}c_2^{\frac{p+1}{2}}\Vert \nabla\fgam{\gf}{p,\varepsilon_0}{\gamma}(x^0)\Vert ^{\frac{p+1}{p-1}}+\varepsilon_{1}
 \\&\leq \fgam{\gf}{p,\varepsilon_0}{\gamma}(x^0)-\alpha \left(c_1-
2^{2-p}\mu^{p-1}c_2 
-\frac{2\mathcal{L}_p}{p+1} \alpha^{\frac{p-1}{2}}c_2^{\frac{p+1}{2}}\right)\Vert \nabla\fgam{\gf}{p,\varepsilon_0}{\gamma}(x^0)\Vert ^{\frac{p+1}{p-1}}+\varepsilon_{1}.
\end{align*}
Since $\alpha\leq \left(\frac{(p+1)\left(c_1-2^{2-p}\mu^{p-1}c_2\right)}{2c_2^{\frac{p+1}{2}}\mathcal{L}_p}\right)^{\frac{2}{p-1}}$, we obtain
$\fgam{\gf}{p,\varepsilon_1}{\gamma}(x^{1})\leq\fgam{\gf}{p,\varepsilon_0}{\gamma}(x^0)+\varepsilon_{1}$, leading to
\begin{align*}
\fgam{\gf}{p}{\gamma}(x^{1})&\leq\fgam{\gf}{p}{\gamma}(x^0)
+\varepsilon_0+\varepsilon_1\leq \fgam{\gf}{p}{\gamma}(x^0)
+2\ov{\varepsilon},
\end{align*}
i.e., $x^1\in \mathcal{S}$. 
Hence, $\fgam{\gf}{p}{\gamma}$ is differentiable at $x^1$.
 By repeating this approach, we obtain
\begin{align*}
\fgam{\gf}{p}{\gamma}(x^{k+1})\leq\fgam{\gf}{p}{\gamma}(x^k)
+\varepsilon_k+\varepsilon_{k+1}
\leq \fgam{\gf}{p}{\gamma}(x^0)+\varepsilon_0+2\left(\varepsilon_1+\ldots+\varepsilon_k\right)+\varepsilon_{k+1}
\leq \fgam{\gf}{p}{\gamma}(x^0)+2\ov{\varepsilon}.
\end{align*}
Therefore, the sequence $\{x^k\}_{k\geq 0}\subseteq \mathcal{S}$ is well-defined.
\end{proof}

\subsection{Proof of Theorem~\ref{th:pfhigda:well}}
\label{proof:th:pfhigda:well}
\begin{proof}
\ref{th:pfhigda:well:a}
Assume that $x^k\in\mathcal{S}$.
Since $\alpha_{k,j}\leq(\gamma^{\frac{2}{p-1}})/c_2$, the same localization argument as in the proof of
Theorem~\ref{th:welldfalg} shows that every trial point $\widehat{x}^{k+1}_j$ belongs to $\mb(0;r)$. Hence the H\"older descent estimate is valid at every trial point. If $L_{k+1,j}\geq\mathcal{L}_p$, then the estimate \eqref{eq:discentforseq}, with $\alpha=\alpha_{k,j}$, yields
\[
\fgam{\gf}{p,\varepsilon_{k+1}}{\gamma}(\widehat{x}^{k+1}_j)\leq\fgam{\gf}{p,\varepsilon_k}{\gamma}(x^k)-\alpha_{k,j}\left(A_p-\frac{2L_{k+1,j}}{p+1}\alpha_{k,j}^{\frac{p-1}{2}}c_2^{\frac{p+1}{2}}\right)
\left\|\nabla\fgam{\gf}{p,\varepsilon_k}{\gamma}(x^k)\right\|^{\frac{p+1}{p-1}}+\varepsilon_{k+1},
\]
resulting in \eqref{eq:discentforseq:2}. Since the trial values $L_{k+1,j}$ increase geometrically, the inner loop terminates after finitely many trials. Moreover, by the definition of $\alpha_{k,j}$,
\begin{equation}\label{eq:pfhigda:coeff}
L_{k+1,j}\alpha_{k,j}^{\frac{p-1}{2}}c_2^{\frac{p+1}{2}}\leq A_p,
\end{equation}
and consequently
\[
A_p-\frac{2L_{k+1,j}}{p+1}\alpha_{k,j}^{\frac{p-1}{2}}c_2^{\frac{p+1}{2}}\geq \frac{p-1}{p+1}A_p>0.
\]
Hence, an accepted step satisfies
\[
\fgam{\gf}{p,\varepsilon_{k+1}}{\gamma}(x^{k+1}) \leq\fgam{\gf}{p,\varepsilon_k}{\gamma}(x^k)+\varepsilon_{k+1}.
\]
The same argument as in the proof of Theorem~\ref{th:welldfalg} then gives $x^{k+1}\in\mathcal{S}$. By induction, $x^k\in\mathcal{S}$ for all $k$.
\\
\ref{th:pfhigda:well:b} At the first outer iteration, if $L_0\geq\mathcal{L}_p$, then the first trial is accepted and $L_1=L_0\leq L_{\max}$. If $L_0<\mathcal{L}_p$, then the first trial value that reaches $\mathcal{L}_p$ is strictly smaller than $\tau\mathcal{L}_p$, and hence $L_1\leq\tau\mathcal{L}_p\leq L_{\max}$. Suppose now that $L_k\leq L_{\max}$. The next backtracking search starts from $L_k/\tau$. If $L_k/\tau\geq\mathcal{L}_p$, the first trial is accepted. Otherwise, the first trial reaching $\mathcal{L}_p$ is strictly smaller than $\tau\mathcal{L}_p$. In either case, $L_{k+1}\leq L_{\max}$. Thus, by induction, $L_k\leq L_{\max}$ for every $k\in \Nz$.
Since the function
\[
L\mapsto\left(\frac{A_p}{c_2^{\frac{p+1}{2}}L}\right)^{\frac{2}{p-1}}
\]
is decreasing on $(0,+\infty)$, every accepted step-size satisfies $\alpha_k \geq \underline{\alpha}>0$.
\\
\ref{th:pfhigda:well:c} A combination of the acceptance condition \eqref{eq:discentforseq:2}, \eqref{eq:pfhigda:coeff}, and $\alpha_k\geq\underline{\alpha}$ implies
\[
\fgam{\gf}{p,\varepsilon_{k+1}}{\gamma}(x^{k+1}) \leq \fgam{\gf}{p,\varepsilon_k}{\gamma}(x^k) -\underline{\alpha}\frac{p-1}{p+1}A_p
\left\|\nabla\fgam{\gf}{p,\varepsilon_k}{\gamma}(x^k)\right\|^{\frac{p+1}{p-1}}+\varepsilon_{k+1},
\]
which adjust \eqref{eq:pfhigda:uniform-decrease}.
\end{proof}

\subsection{Proof of Theorem~\ref{th:lem:disdir}}
\label{proof:th:lem:disdir}
\begin{proof}
From Lemma \ref{lem:disdir}, we have $\langle \nabla\fgam{\gf}{p}{\gamma}(x^k), d^k\rangle<0$. Thus, for each $\lambda\in (0, 1)$,
we get
\[
\mathop{\bs\lim}\limits_{t\downarrow 0}\frac{\fgam{\gf}{p}{\gamma}(x^{k}+t d^k)-\fgam{\gf}{p}{\gamma}(x^{k})}{t}
=\langle \nabla\fgam{\gf}{p}{\gamma}(x^k), d^k\rangle< \lambda \langle \nabla\fgam{\gf}{p}{\gamma}(x^k), d^k\rangle.
\]
Hence, there exists some $\ov{\alpha}>0$ such that,
\[
\fgam{\gf}{p}{\gamma}(x^{k}+\alpha d^k)-\fgam{\gf}{p}{\gamma}(x^{k})<
\alpha\lambda \langle \nabla\fgam{\gf}{p}{\gamma}(x^k), d^k\rangle, \qquad \forall \alpha\in (0, \ov{\alpha}).
\]
Since $\fgam{\gf}{p,\varepsilon_{k+1}}{\gamma}(x^{k}+\alpha d^k)\leq \fgam{\gf}{p}{\gamma}(x^{k}+\alpha d^k)+\varepsilon_{k+1}$ and
$\fgam{\gf}{p}{\gamma}(x^{k})\leq \fgam{\gf}{p,\varepsilon_{k}}{\gamma}(x^{k})$, we have
\begin{align*}
 \fgam{\gf}{p,\varepsilon_{k+1}}{\gamma}(x^{k}+\alpha d^k)\leq \fgam{\gf}{p,\varepsilon_{k}}{\gamma}(x^{k})+\alpha \lambda \langle \nabla\fgam{\gf}{p}{\gamma}(x^k), d^k\rangle+\varepsilon_{k+1}, \qquad \forall \alpha\in (0, \ov{\alpha}),
\end{align*}
establishing the claim.
\end{proof}

\subsection{Proof of Theorem~\ref{th:welldf:alg:first}}
\label{proof:th:welldf:alg:first}
\begin{proof}
Suppose that $x^k\in\mathcal{S}$. By Assumption~\ref{assum:rgamma}, $\fgam{\gf}{p}{\gamma}$ is differentiable at $x^k$. First assume that
$\nabla\fgam{\gf}{p}{\gamma}(x^k)\neq 0$. Using \eqref{eq:relapprox:b} and \eqref{eq:dir:th:conv:alg:first:tem}, 
\[
\left\langle\nabla\fgam{\gf}{p}{\gamma}(x^k),d^k\right\rangle\leq c_p\left\|\nabla\fgam{\gf}{p,\varepsilon_k}{\gamma}(x^k)\right\|\|d^k\|
+\left\langle\nabla\fgam{\gf}{p,\varepsilon_k}{\gamma}(x^k),d^k\right\rangle \leq-A_p\left\|\nabla\fgam{\gf}{p,\varepsilon_k}{\gamma}(x^k)\right\|^{1+\vartheta}.
\]
Invoking Theorem~\ref{th:lem:disdir} implies that the Armijo condition in Algorithm~\ref{alg:first} holds for all sufficiently small step-sizes. Since the trial step-sizes decrease geometrically, the backtracking loop terminates after finitely many reductions. If $\nabla\fgam{\gf}{p}{\gamma}(x^k)=0$,  the same estimate leads to
\[
0 \leq -A_p \left\|\nabla\fgam{\gf}{p,\varepsilon_k}{\gamma}(x^k)\right\|^{1+\vartheta}.
\]
Since $A_p>0$, it follows that $\nabla\fgam{\gf}{p,\varepsilon_k}{\gamma}(x^k)=0$. Then \eqref{eq:dir:th:conv:alg:first:tem} yields $d^k=0$. Moreover,
\[
\fgam{\gf}{p,\varepsilon_{k+1}}{\gamma}(x^k)\leq \fgam{\gf}{p}{\gamma}(x^k)+\varepsilon_{k+1} \leq \fgam{\gf}{p,\varepsilon_k}{\gamma}(x^k)
+\varepsilon_{k+1},
\]
i.e., the Armijo condition is satisfied immediately. Therefore, in either case, the accepted step satisfies
\[
\fgam{\gf}{p,\varepsilon_{k+1}}{\gamma}(x^{k+1})
\leq
\fgam{\gf}{p,\varepsilon_k}{\gamma}(x^k)
+\varepsilon_{k+1}.
\]
As in the proof of Theorem~\ref{th:welldfalg}, $\fgam{\gf}{p}{\gamma}(x^{k+1})\leq\fgam{\gf}{p}{\gamma}(x^k)+\varepsilon_k+\varepsilon_{k+1}$,
and hence $\fgam{\gf}{p}{\gamma}(x^{k+1}) \leq \fgam{\gf}{p}{\gamma}(x^0) +2\ov{\varepsilon}$,
i.e., $x^{k+1}\in\mathcal{S}$. Since $x^0\in\mathcal{S}$, the result follows by induction.
\end{proof}

\subsection{Proof of Theorem~\ref{th:conKL}}
\label{proof:th:conKL}
\begin{proof}
From the inequality \eqref{eq:relapprox:relgardgrad}, for each $k\in \Nz$, we have
\[
\left\Vert \nabla\fgam{\gf}{p}{\gamma}(x^k)\right\Vert \leq \Vert \nabla\fgam{\gf}{p,\varepsilon_k}{\gamma}(x^k)\Vert+\frac{2^{2-p}}{\gamma}\delta_k^{p-1},
\qquad \Vert \nabla\fgam{\gf}{p,\varepsilon_k}{\gamma}(x^k)\Vert\leq  \left\Vert \nabla\fgam{\gf}{p}{\gamma}(x^k)\right\Vert+\frac{2^{2-p}}{\gamma}\delta_k^{p-1}.
\]
Using these inequalities, we conclude that
\begin{equation}\label{eq:th:conKL:ineqa:arm}
2^{-\vartheta}\left\Vert \nabla\fgam{\gf}{p}{\gamma}(x^k)\right\Vert^{1+\vartheta} \leq 
\Vert \nabla\fgam{\gf}{p,\varepsilon_k}{\gamma}(x^k)\Vert^{1+\vartheta}+\left(\frac{2^{2-p}}{\gamma}\right)^{1+\vartheta}\delta_k^{(p-1)(1+\vartheta)},
\end{equation}
and
\begin{equation}\label{eq:th:conKL:ineqb:arm}
2^{-\vartheta} \Vert \nabla\fgam{\gf}{p,\varepsilon_k}{\gamma}(x^k)\Vert^{1+\vartheta}\leq 
\left\Vert \nabla\fgam{\gf}{p}{\gamma}(x^k)\right\Vert^{1+\vartheta}+\left(\frac{2^{2-p}}{\gamma}\right)^{1+\vartheta}\delta_k^{(p-1)(1+\vartheta)}.
\end{equation}
In conjunction with \eqref{eq:ep-approx:fun},
the inequalities \eqref{eq:th:conKL:ineqa:arm} and \eqref{eq:th:conKL:ineqb:arm} lead to
\begin{align}\label{eq:th:conKL:a}
\fgam{\gf}{p}{\gamma}(x^{k+1})
&\leq \fgam{\gf}{p}{\gamma}(x^k)+\varepsilon_k -\widehat{\varrho}\Vert\nabla\fgam{\gf}{p,\varepsilon_k}{\gamma}(x^k)\Vert^{1+\vartheta}
+\varepsilon_{k+1}
\nonumber \\
&\leq \fgam{\gf}{p}{\gamma}(x^k)-\widehat{\varrho} 2^{-\vartheta}\left\Vert \nabla\fgam{\gf}{p}{\gamma}(x^k)\right\Vert^{1+\vartheta}
+\widehat{\varrho} \left(\frac{2^{2-p}}{\gamma}\right)^{1+\vartheta}\delta_k^{(p-1)(1+\vartheta)}+\varepsilon_k+\varepsilon_{k+1}
\nonumber \\
&= \fgam{\gf}{p}{\gamma}(x^k)-\widehat{\varrho} 2^{-(1+\vartheta)}\left\Vert \nabla\fgam{\gf}{p}{\gamma}(x^k)\right\Vert^{1+\vartheta}
-\widehat{\varrho} 2^{-(1+\vartheta)}\left\Vert \nabla\fgam{\gf}{p}{\gamma}(x^k)\right\Vert^{1+\vartheta}
+\widehat{\varrho} \left(\frac{2^{2-p}}{\gamma}\right)^{1+\vartheta}\delta_k^{(p-1)(1+\vartheta)}+\varepsilon_k+\varepsilon_{k+1}
\nonumber \\
&\leq \fgam{\gf}{p}{\gamma}(x^k)-\widehat{\varrho} 2^{-(1+\vartheta)}\left\Vert \nabla\fgam{\gf}{p}{\gamma}(x^k)\right\Vert^{1+\vartheta}
-\widehat{\varrho} 2^{-(1+\vartheta)}2^{-\vartheta} \Vert \nabla\fgam{\gf}{p,\varepsilon_k}{\gamma}(x^k)\Vert^{1+\vartheta}
\nonumber\\&\quad+ \widehat{\varrho} 2^{-(1+\vartheta)}\left(\frac{2^{2-p}}{\gamma}\right)^{1+\vartheta}\delta_k^{(p-1)(1+\vartheta)}+\widehat{\varrho} \left(\frac{2^{2-p}}{\gamma}\right)^{1+\vartheta}\delta_k^{{(p-1)(1+\vartheta)}}+\varepsilon_k+\varepsilon_{k+1}.
\end{align}
Let us define $v_k:=\widehat{c}\widehat{\varrho}\sum_{j=k}^{\infty}\delta_j^{(p-1)(1+\vartheta)}+\sum_{j=k}^{\infty}\varepsilon_j+\sum_{j=k+1}^{\infty}\varepsilon_j$,
where $\widehat{c}:=\left(1+ 2^{-(1+\vartheta)}\right)\left(\frac{2^{2-p}}{\gamma}\right)^{1+\vartheta}$.
Together with \eqref{eq:th:conKL:a}, this implies
\begin{equation}\label{eq:th:conKL:zz1}
\fgam{\gf}{p}{\gamma}(x^{k+1})+v_{k+1}\leq \fgam{\gf}{p}{\gamma}(x^k)+v_k -\widehat{\varrho} 2^{-(1+\vartheta)}\left\Vert \nabla\fgam{\gf}{p}{\gamma}(x^k)\right\Vert^{1+\vartheta}
-\widehat{\varrho} 2^{-(1+\vartheta)}2^{-\vartheta} \Vert \nabla\fgam{\gf}{p,\varepsilon_k}{\gamma}(x^k)\Vert^{1+\vartheta}.
\end{equation}
Thus, the sequence $\{\fgam{\gf}{p}{\gamma}(x^k)+v_k\}_{k\in \Nz}$ is non-increasing and bounded from below by $\gf^*$, i.e., there exists some $\fv\in \R$ such that
\begin{equation}\label{eq:th:conKL:zz1-1}
\lim_{k\to \infty}\fgam{\gf}{p}{\gamma}(x^k)=\lim_{k\to \infty}\fgam{\gf}{p}{\gamma}(x^k)+v_k=\fv.
\end{equation}
Assume that $\widehat{x}\in \Omega(\{x^k\})$ is arbitrary and is regarding the subsequence $\{x^j\}_{j\in J\subseteq\Nz}$. Note that $\widehat{x}\in \bs{\rm Fix}(\prox{\gf}{\gamma}{p})$.
Moreover, from the proof of Theorem~\ref{th:comcon}, $\prox{\gf}{\gamma}{p}(x^j)\to \widehat{x}$ with $j\in J$, i.e., from the lsc property of $\gf$, the continuity of $\fgam{\gf}{p}{\gamma}$, and 
\[\gf(\prox{\gf}{\gamma}{p}(x^j))\leq \gf(\prox{\gf}{\gamma}{p}(x^j))+\frac{1}{p\gamma}\Vert x^j- \prox{\gf}{\gamma}{p}(x^j)\Vert^p=\fgam{\gf}{p}{\gamma}(x^j),\]
 it holds that
\begin{equation}\label{eqb:th:conKL:zz1-1}
\gf(\widehat{x})\leq
\liminf_{\scriptsize{\begin{matrix}j\to\infty\\j\in J \end{matrix}}}\gf(\prox{\gf}{\gamma}{p}(x^j))
\leq \limsup_{\scriptsize{\begin{matrix}j\to\infty\\j\in J \end{matrix}}}\gf(\prox{\gf}{\gamma}{p}(x^j))
\leq \limsup_{\scriptsize{\begin{matrix}j\to\infty\\j\in J \end{matrix}}} \fgam{\gf}{p}{\gamma}(x^j)=\fgam{\gf}{p}{\gamma}(\widehat{x})\leq \gf(\widehat{x}).
\end{equation}
Combining  \eqref{eq:th:conKL:zz1-1} and \eqref{eqb:th:conKL:zz1-1} leads to
\begin{equation}\label{eqc:th:conKL:zz1-1}
\gf(\widehat{x})=\fgam{\gf}{p}{\gamma}(\widehat{x})=\lim_{\scriptsize{\begin{matrix}j\to\infty\\j\in J \end{matrix}}}\fgam{\gf}{p}{\gamma}(x^j)=\lim_{k\to\infty}\fgam{\gf}{p}{\gamma}(x^k)=\fv,
\end{equation}
which means $\fgam{\gf}{p}{\gamma}$ is constant on $\Omega(\{x^k\})$.
Hence, from Fact~\ref{lem:Uniformized KL property},
Step~\ref{alg:fram:nonmono:init} of Algorithm~\ref{alg:fram:nonmono}, and that $\Omega(\{x^k\})\subseteq \mathcal{S}$,
there exist $\widehat{r}>0$, $\eta>0$, and a desingularizing function $\phi$  such that for all $x\in X$ with
\[
X:=\left\{x\in \R^n \mid \dist(x, \Omega(\{x^k\}))<\widehat{r}\right\}\cap\left\{x\in \R^n \mid 0<\vert \fgam{\gf}{p}{\gamma}(x) - \fgam{\gf}{p}{\gamma}(\widehat{x})\vert<\eta\right\}\subseteq \mb(0; r),
\]
we obtain
\begin{equation}\label{eq:th:globconv:klI:a}
\phi'\left(\vert \fgam{\gf}{p}{\gamma}(x) - \fgam{\gf}{p}{\gamma}(\widehat{x})\vert\right)\Vert \nabla \fgam{\gf}{p}{\gamma}(x)\Vert\geq 1.
\end{equation}
Since \eqref{eq:dir:th:conv:alg:first:tem} and Theorem~\ref{th:conv:alg:nonmono} yield $\Vert d^k\Vert\to 0$, 
and since $\alpha_k$ is bounded from above,
we have  $\Vert x^{k+1}- x^k\Vert=\alpha_k \Vert d^k\Vert\to 0$.  Together with \eqref{eqc:th:conKL:zz1-1}, for some $\widehat{k}>0$ and for each $k\geq \widehat{k}$, we have $\dist(x^k, \Omega(\{x^k\}))<\widehat{r}$, $v_k<\frac{\eta}{2}$, and 
$\vert\fgam{\gf}{p}{\gamma}(x^k)-\fgam{\gf}{p}{\gamma}(\widehat{x})\vert<\frac{\eta}{2}$.
Defining $\Delta_k:=\phi\left(\fgam{\gf}{p}{\gamma}(x^k)-\fv+v_k\right)$ and
utilizing concavity of $\phi$, monotonically decreasing of $\phi'$, and \eqref{eq:th:conKL:zz1} lead to
\begin{align}\label{eq:th:conKL:z1}
  \Delta_k-\Delta_{k+1} & \geq \phi'\left(\fgam{\gf}{p}{\gamma}(x^k)-\fv+v_k\right)\left[\fgam{\gf}{p}{\gamma}(x^k)+v_k-\fgam{\gf}{p}{\gamma}(x^{k+1})-v_{k+1}\right]
\nonumber \\&\geq \phi'\left(\vert\fgam{\gf}{p}{\gamma}(x^k)-\fv\vert+v_k\right)\left[\fgam{\gf}{p}{\gamma}(x^k)+v_k-\fgam{\gf}{p}{\gamma}(x^{k+1})-v_{k+1}\right]
\nonumber  \\&\geq  \phi'\left(\vert\fgam{\gf}{p}{\gamma}(x^k)-\fv\vert+v_k\right)
\left[\widehat{\varrho} 2^{-(1+\vartheta)}\left\Vert \nabla\fgam{\gf}{p}{\gamma}(x^k)\right\Vert^{1+\vartheta}
+\widehat{\varrho} 2^{-(1+\vartheta)}2^{-\vartheta} \Vert \nabla\fgam{\gf}{p,\varepsilon_k}{\gamma}(x^k)\Vert^{1+\vartheta}\right].
\end{align}
 Note that
since $\phi' > 0$,  this together with \eqref{eq:th:conKL:z1} implies $\Delta_k - \Delta_{k+1} \geq 0$.
We consider two cases: (i) $\fgam{\gf}{p}{\gamma}(x^k)-\fv=0$; (ii) $\fgam{\gf}{p}{\gamma}(x^k)-\fv\neq 0$. 
In Case~(i), $\fgam{\gf}{p}{\gamma}(x^k)-\fv=0$, then 
from \eqref{eq:th:conKL:z1}, we have
\begin{align*}
\widehat{\varrho} 2^{-(1+\vartheta)}\left\Vert \nabla\fgam{\gf}{p}{\gamma}(x^k)\right\Vert^{1+\vartheta}
+\widehat{\varrho} 2^{-(1+\vartheta)}2^{-\vartheta} \Vert \nabla\fgam{\gf}{p,\varepsilon_k}{\gamma}(x^k)\Vert^{1+\vartheta}&\leq 
\left(\Delta_k-\Delta_{k+1} \right)\left(\phi'(v_k)\right)^{-1}\nonumber\\
&\leq \left(\Delta_k-\Delta_{k+1} \right)\left(\left\Vert \nabla\fgam{\gf}{p}{\gamma}(x^k)\right\Vert+\left(\phi'(v_k)\right)^{-1}\right).
\end{align*}
In Case~(ii), $\fgam{\gf}{p}{\gamma}(x^k)-\fv\neq 0$, then \eqref{eq:th:conKL:z1} and quasi-additivity property, by increasing $\widehat{k}$ if necessary, imply the existence of some $c_\phi>0$ such that
\[
  \Delta_k-\Delta_{k+1}\geq  \frac{1}{c_\phi}\frac{1}{[\phi'\left(\vert\fgam{\gf}{p}{\gamma}(x^k)-\fv\vert\right)]^{-1}+\left(\phi'(v_k)\right)^{-1}}\left[\widehat{\varrho} 2^{-(1+\vartheta)}\left\Vert \nabla\fgam{\gf}{p}{\gamma}(x^k)\right\Vert^{1+\vartheta}
+\widehat{\varrho} 2^{-(1+\vartheta)}2^{-\vartheta} \Vert \nabla\fgam{\gf}{p,\varepsilon_k}{\gamma}(x^k)\Vert^{1+\vartheta}\right].
\]
Together with \eqref{eq:th:globconv:klI:a}, this ensures
\begin{align*}
\widehat{\varrho} 2^{-(1+\vartheta)}\left\Vert \nabla\fgam{\gf}{p}{\gamma}(x^k)\right\Vert^{1+\vartheta}
+\widehat{\varrho} 2^{-(1+\vartheta)}2^{-\vartheta} \Vert \nabla\fgam{\gf}{p,\varepsilon_k}{\gamma}(x^k)\Vert^{1+\vartheta}
&\leq c_\phi\left(\Delta_k-\Delta_{k+1} \right)\left([\phi'\left(\vert\fgam{\gf}{p}{\gamma}(x^k)-\fv\vert\right)]^{-1}+\left(\phi'(v_k)\right)^{-1}\right)
\nonumber\\&\leq c_\phi\left(\Delta_k-\Delta_{k+1} \right)\left(\left\Vert \nabla\fgam{\gf}{p}{\gamma}(x^k)\right\Vert+\left(\phi'(v_k)\right)^{-1}\right).
\end{align*}
In both cases and by setting $\widehat{m}:=\bs\max\{1, c_\phi\}$,  for each $k\geq \widehat{k}$, we obtain
\begin{align}\label{eq:th:conKL:z4}
  \left\Vert \nabla\fgam{\gf}{p}{\gamma}(x^k)\right\Vert^{1+\vartheta} +2^{-\vartheta}
   \Vert \nabla\fgam{\gf}{p,\varepsilon_k}{\gamma}(x^k)\Vert^{1+\vartheta}
&\leq \widehat{m}\widehat{\varrho}^{-1} 2^{1+\vartheta}\left(\Delta_k-\Delta_{k+1} \right)\left(\left\Vert \nabla\fgam{\gf}{p}{\gamma}(x^k)\right\Vert+\left(\phi'(v_k)\right)^{-1}\right),
\end{align}
leading to
\[
2^{-\frac{1}{\vartheta}}\left(\Vert \nabla\fgam{\gf}{p}{\gamma}(x^k)\Vert^{\vartheta} +2^{-\frac{\vartheta^2}{1+\vartheta}}
\Vert \nabla\fgam{\gf}{p,\varepsilon_k}{\gamma}(x^k)\Vert^{\vartheta}\right)^{\frac{1+\vartheta}{\vartheta}}\leq\widehat{m}\widehat{\varrho}^{-1} 2^{1+\vartheta}\left(\Delta_k-\Delta_{k+1} \right)\left(\left\Vert \nabla\fgam{\gf}{p}{\gamma}(x^k)\right\Vert+\left(\phi'(v_k)\right)^{-1}\right).
\]
Accordingly, we come to
\[
\Vert \nabla\fgam{\gf}{p}{\gamma}(x^k)\Vert^{\vartheta}+2^{-\frac{\vartheta^2}{1+\vartheta}}
\Vert \nabla\fgam{\gf}{p,\varepsilon_k}{\gamma}(x^k)\Vert^{\vartheta}\leq
\left[\frac{\widehat{m}}{\widehat{\varrho}} 2^{\frac{\vartheta^2+\vartheta+1}{\vartheta}}\left(\Delta_k-\Delta_{k+1} \right)
\left(\left\Vert \nabla\fgam{\gf}{p}{\gamma}(x^k)\right\Vert+\left(\phi'(v_k)\right)^{-1}\right)\right]^{\frac{\vartheta}{1+\vartheta}}.
\]
We recall that if $a\geq 0$, $b\geq 0$ are non-negative numbers, and $q>1$ such that $\frac{\vartheta}{1+\vartheta}+\frac{1}{q}=1$, then via the Young inequality and \eqref{eq:intrp:p01},
we obtain
\[
(ab)^{\frac{\vartheta}{1+\vartheta}}\leq \left(\frac{a^{\frac{1+\vartheta}{\vartheta}}}{\frac{1+\vartheta}{\vartheta}}+\frac{b^q}{q}\right)^{\frac{\vartheta}{1+\vartheta}}\leq \frac{a}{(\frac{1+\vartheta}{\vartheta})^{\frac{\vartheta}{1+\vartheta}}}+\frac{b^{\vartheta}}{(1+\vartheta)^{\frac{\vartheta}{1+\vartheta}}}\leq a+b^{\vartheta}.
\]
If $\vartheta> 1$, then
\begin{align*}
\Vert \nabla\fgam{\gf}{p}{\gamma}(x^k)\Vert^{\vartheta}+2^{-\frac{\vartheta^2}{1+\vartheta}}
\Vert \nabla\fgam{\gf}{p,\varepsilon_k}{\gamma}(x^k)\Vert^{\vartheta}
&\leq
\left[\frac{\widehat{m}}{\widehat{\varrho}}  2^{\frac{\vartheta^2+\vartheta+1}{\vartheta}}\left(\Delta_k-\Delta_{k+1} \right)
\left(\left\Vert \nabla\fgam{\gf}{p}{\gamma}(x^k)\right\Vert+\left(\phi'(v_k)\right)^{-1}\right)\right]^{\frac{\vartheta}{1+\vartheta}}
\\
&=
\left[\frac{\widehat{m}}{\widehat{\varrho}}  2^{2+\vartheta}\left(\Delta_k-\Delta_{k+1} \right)
2^{\frac{1-\vartheta}{\vartheta}}\left(\left\Vert \nabla\fgam{\gf}{p}{\gamma}(x^k)\right\Vert+\left(\phi'(v_k)\right)^{-1}\right)\right]^{\frac{\vartheta}{1+\vartheta}}
\\
&\leq
\frac{\widehat{m}}{\widehat{\varrho}}  2^{2+\vartheta}\left(\Delta_k-\Delta_{k+1} \right)+
2^{1-\vartheta}\left(\left\Vert \nabla\fgam{\gf}{p}{\gamma}(x^k)\right\Vert+\left(\phi'(v_k)\right)^{-1}\right)^{\vartheta}
\\
&\leq
\frac{\widehat{m}}{\widehat{\varrho}} 2^{2+\vartheta}\left(\Delta_k-\Delta_{k+1} \right)+
\left\Vert \nabla\fgam{\gf}{p}{\gamma}(x^k)\right\Vert^{\vartheta}+\left(\left(\phi'(v_k)\right)^{-1}\right)^{\vartheta}.
\end{align*}
If $\vartheta\in(0, 1]$, then
\begin{align*}
\Vert \nabla\fgam{\gf}{p}{\gamma}(x^k)\Vert^{\vartheta}+2^{-\frac{\vartheta^2}{1+\vartheta}}
\Vert \nabla\fgam{\gf}{p,\varepsilon_k}{\gamma}(x^k)\Vert^{\vartheta}
&\leq
\left[\frac{\widehat{m}}{\widehat{\varrho}}  2^{\frac{\vartheta^2+\vartheta+1}{\vartheta}}\left(\Delta_k-\Delta_{k+1} \right)
\left(\left\Vert \nabla\fgam{\gf}{p}{\gamma}(x^k)\right\Vert+\left(\phi'(v_k)\right)^{-1}\right)\right]^{\frac{\vartheta}{1+\vartheta}}
\\
&\leq
\frac{\widehat{m}}{\widehat{\varrho}}  2^{\frac{\vartheta^2+\vartheta+1}{\vartheta}}\left(\Delta_k-\Delta_{k+1} \right)
+\left(\left\Vert \nabla\fgam{\gf}{p}{\gamma}(x^k)\right\Vert+\left(\phi'(v_k)\right)^{-1}\right)^{\vartheta}
\\
&\leq
\frac{\widehat{m}}{\widehat{\varrho}}  2^{\frac{\vartheta^2+\vartheta+1}{\vartheta}}\left(\Delta_k-\Delta_{k+1} \right)+
\left\Vert \nabla\fgam{\gf}{p}{\gamma}(x^k)\right\Vert^{\vartheta}+\left(\left(\phi'(v_k)\right)^{-1}\right)^{\vartheta}.
\end{align*}
Hence, for each $k \geq \widehat{k}$, we come to
\begin{equation}\label{eq:th:conKL:d}
\Vert \nabla\fgam{\gf}{p,\varepsilon_k}{\gamma}(x^k)\Vert^{\vartheta}
\leq  \varpi(p)
\left(\Delta_k-\Delta_{k+1} \right)+2^{\frac{\vartheta^2}{1+\vartheta}}\left(\left(\phi'(v_k)\right)^{-1}\right)^{\vartheta}.
\end{equation}
where $\varpi(p):=\frac{\widehat{m}}{\widehat{\varrho}} 2^{\frac{2\vartheta^2+3\vartheta+2}{1+\vartheta}}$ if $\vartheta>1$ and
$\varpi(p):=\frac{\widehat{m}}{\widehat{\varrho}} 2^{\frac{2\vartheta\left(\vartheta^2+\vartheta+1\right)+1}{\vartheta(1+\vartheta)}}$ if $\vartheta\in (0, 1]$.
Since $v_k\leq w_k$, we have $\phi'(w_k)\leq \phi'(v_k)$, i.e., $\left(\phi'(v_k)\right)^{-1}\leq \left(\phi'(w_k)\right)^{-1}$. Thus, from Assumption~\ref{assum:eps:weak}~$\ref{assum:eps:weak:c}$, we obtain
$\sum_{k=0}^{\infty}(\left(\phi'(v_k)\right)^{-1})^\vartheta<\infty$. 
Note that the continuity of $\phi$ and $\fgam{\gf}{p}{\gamma}(x^k) - \fv + v_k\to 0$ imply $\Delta_k \to 0$. Thus, from \eqref{eq:th:conKL:d}, we obtain
\begin{align*}
\sum_{k=\widehat{k}}^{\infty}\Vert x^{k+1}-x^k\Vert\leq \varsigma c_2
\sum_{k=\widehat{k}}^{\infty}\Vert \nabla\fgam{\gf}{p,\varepsilon_k}{\gamma}(x^k)\Vert^{\vartheta}&\leq \varsigma c_2\varpi(p)\sum_{k=\widehat{k}}^{\infty}\left(\Delta_k-\Delta_{k+1} \right)+\varsigma c_2 2^{\frac{\vartheta^2}{1+\vartheta}} \sum_{k=\widehat{k}}^{\infty}(\left(\phi'(v_k)\right)^{-1})^{\vartheta}
\\& =\varsigma c_2\varpi(p)\Delta_{\widehat{k}}+\varsigma c_2 2^{\frac{\vartheta^2}{1+\vartheta}} \sum_{k=\widehat{k}}^{\infty}(\left(\phi'(v_k)\right)^{-1})^{\vartheta}<\infty.
\end{align*}
This implies that the sequence $\{x^k\}_{k\in \Nz}$ is a Cauchy sequence and hence for the sequence $\{x^k\}_{k\in \Nz}$, we have  $x^k\to \widehat{x}$. Thus, the convergence of this sequence is global.
\end{proof}

\subsection{Proof of Theorem~\ref{th:linconv}}
\label{proof:th:linconv}
\begin{proof}
By Theorem~\ref{th:conKL},  the sequence  $\{x^k\}_{k\in \Nz}$ converges  to a proximal fixed point $\widehat{x}$. By defining
$B_k:=\sum_{i\geq k}\Vert \nabla\fgam{\gf}{p,\varepsilon_i}{\gamma}(x^i)\Vert^{\vartheta}$, 
\begin{align*}
\Vert x^k - \widehat{x}\Vert\leq \sum_{i\geq k} \Vert x^{i+1} - x^i\Vert \leq \varsigma c_2 B_k.
\end{align*}
Hence, it is enough to show that the sequence $\{B_k\}_{k\in \Nz}$ converges at an asymptotically linear rate. Let $\phi(t) := c t^{1 - \theta}$, where $c>0$. 
From \eqref{eq:th:globconv:klI:a}, for enough large $\widehat{k}\in \mathbb{N}$ and for each $k\geq \widehat{k}$, we obtain
\begin{equation}\label{eq:th:linconv:d1}
\left[(1-\theta)c\Vert \nabla\fgam{\gf}{p}{\gamma}(x^k)\Vert\right]^{\frac{1}{\theta}}\geq \vert \fgam{\gf}{p}{\gamma}(x^k) - \fv\vert.
\end{equation}
Defining $\Delta_k := \phi\left(\fgam{\gf}{p}{\gamma}(x^k) - \fv + v_k\right)$, where $v_k$ is defined in the proof of Theorem \ref{th:conKL}, and applying \eqref{eq:intrp:p01} and \eqref{eq:th:linconv:d1}, we get
\begin{align}\label{eq:th:linconv:d2}
\Delta_k&= c (\fgam{\gf}{p}{\gamma}(x^k)-\fv+v_k)^{1 - \theta}
\leq c (\vert\fgam{\gf}{p}{\gamma}(x^k)-\fv\vert+v_k)^{1 - \theta}
\nonumber\\&\leq c \vert\fgam{\gf}{p}{\gamma}(x^k)-\fv\vert^{1 - \theta}+c v_k^{1 - \theta}
\leq c \left[(1-\theta)c\Vert \nabla\fgam{\gf}{p}{\gamma}(x^k)\Vert\right]^{\frac{1-\theta}{\theta}}+c v_k^{1 - \theta}
\nonumber\\
&\leq  c \left[(1-\theta)c\Vert \nabla\fgam{\gf}{p}{\gamma}(x^k)\Vert\right]^{\vartheta}+c v_k^{\frac{\vartheta}{1+\vartheta}}
\nonumber\\&=  c \left[\frac{(1-\theta)c}{\gamma}\right]^{\vartheta}\Vert x^k- \prox{\gf}{\gamma}{p}(x^k)\Vert^{\vartheta(p-1)}+c v_k^{\frac{\vartheta}{1+\vartheta}}.
\end{align}
Since $\frac{\vartheta}{1+\vartheta}<1$, it follows from \eqref{eq:intrp:p01} that
\begin{align*}
v_k^{\frac{\vartheta}{1+\vartheta}}&\leq\left[\widehat{c}\widehat{\varrho}\sum_{j=k}^{\infty}\delta_j^{(p-1)(1+\vartheta)}\right]^{\frac{\vartheta}{1+\vartheta}}
+\left[2\sum_{j=k}^{\infty}\varepsilon_j\right]^{\frac{\vartheta}{1+\vartheta}}
\leq
\left(\widehat{c}\widehat{\varrho}\right)^{\frac{\vartheta}{1+\vartheta}}\sum_{j=k}^{\infty}\delta_j^{\vartheta(p-1)}
+2^{\frac{\vartheta}{1+\vartheta}}\sum_{j=k}^{\infty}\varepsilon_j^{\frac{\vartheta}{1+\vartheta}}
\end{align*}
where $\widehat{c}:=\left(1+ 2^{-(1+\vartheta)}\right)\left(\frac{2^{2-p}}{\gamma}\right)^{1+\vartheta}$.
Moreover, from \eqref{eq:ep-approx:oper:b}, we have 
\[
\sum_{j=k}^{\infty}\delta_j^{\vartheta(p-1)}\leq \Vert x^k - \prox{\gf}{\gamma}{p, \varepsilon_k}({x}^k)\Vert^{\vartheta(p-1)} \sum_{j=k}^{\infty}\beta_j\leq \ov{\beta}\Vert x^k - \prox{\gf}{\gamma}{p, \varepsilon_k}({x}^k)\Vert^{\vartheta(p-1)},
\]
and
\[
\sum_{j=k}^{\infty}\varepsilon_j^{\frac{\vartheta}{1+\vartheta}}\leq \ov{\beta}\Vert x^k - \prox{\gf}{\gamma}{p, \varepsilon_k}({x}^k)\Vert^{\vartheta(p-1)}.
\]
These lead to
\begin{equation}\label{eq:th:linconv:k1}
v_k^{\frac{\vartheta}{1+\vartheta}}\leq \left(\left(\widehat{c}\widehat{\varrho}\right)^{\frac{\vartheta}{1+\vartheta}}+2^{\frac{\vartheta}{1+\vartheta}}\right)\ov{\beta}\Vert x^k -\prox{\gf}{\gamma}{p, \varepsilon_k}({x}^k)\Vert^{\vartheta(p-1)}.
\end{equation}
Together with \eqref{eq:th:linconv:d2}, this implies
\begin{align}\label{eq:th:linconv:d3}
\Delta_k&\leq  c  \left[\frac{(1-\theta)c}{\gamma}\right]^{\vartheta}\Vert x^k- \prox{\gf}{\gamma}{p}(x^k)\Vert^{\vartheta(p-1)}+c\left(\left(\widehat{c}\widehat{\varrho}\right)^{\frac{\vartheta}{1+\vartheta}}
+2^{\frac{\vartheta}{1+\vartheta}}\right)\ov{\beta}\Vert x^k -\prox{\gf}{\gamma}{p, \varepsilon_k}({x}^k)\Vert^{\vartheta(p-1)}
\nonumber\\&\overset{(i)}{\leq}  c \left(1+\widehat{\beta}\right)^{\vartheta(p-1)}\left[\frac{(1-\theta)c}{\gamma}\right]^{\vartheta}\Vert x^k-\prox{\gf}{\gamma}{p, \varepsilon_k}({x}^k)\Vert^{\vartheta(p-1)}+c\left(\left(\widehat{c}\widehat{\varrho}\right)^{\frac{\vartheta}{1+\vartheta}}+2^{\frac{\vartheta}{1+\vartheta}}\right)\ov{\beta}\Vert x^k -\prox{\gf}{\gamma}{p, \varepsilon_k}({x}^k)\Vert^{\vartheta(p-1)}
\nonumber\\&= b_1 \Vert x^k -\prox{\gf}{\gamma}{p, \varepsilon_k}({x}^k)\Vert^{\vartheta(p-1)},
\end{align}
where $(i)$ follows from the inequalities
\begin{align*}
\Vert x^k - \prox{\gf}{\gamma}{p}({x}^k)\Vert &\leq \Vert x^k -\prox{\gf}{\gamma}{p, \varepsilon_k}({x}^k)\Vert+\Vert\prox{\gf}{\gamma}{p, \varepsilon_k}({x}^k) - \prox{\gf}{\gamma}{p}({x}^k)\Vert
\\&\leq (1+\widehat{\beta})\Vert x^k - \prox{\gf}{\gamma}{p, \varepsilon_k}({x}^k)\Vert,
\end{align*}
and $b_1:= 
 c \left(1+\widehat{\beta}\right)^{\vartheta(p-1)}\left[\frac{(1-\theta)c}{\gamma}\right]^{\vartheta}
 +c\left(\left(\widehat{c}\widehat{\varrho}\right)^{\frac{\vartheta}{1+\vartheta}}+2^{\frac{\vartheta}{1+\vartheta}}\right)\ov{\beta}
$.

In addition, since $\fgam{\gf}{p}{\gamma}$ satisfies the KL property with an exponent, we have $c_\phi=1$. Thus, we have
 $\widehat{m}:=\bs\max\{1, c_\phi\}=1$ in  \eqref{eq:th:conKL:z4}. Thus, from  \eqref{eq:th:conKL:z4}, we obtain
\begin{equation}\label{eq:th:conKL:z5}
 \left\Vert \nabla\fgam{\gf}{p}{\gamma}(x^k)\right\Vert^{1+\vartheta} \leq \widehat{\varrho}^{-1} 2^{1+\vartheta}\left(\Delta_k-\Delta_{k+1} \right)\left(\left\Vert \nabla\fgam{\gf}{p}{\gamma}(x^k)\right\Vert+[\phi'(v_k)]^{-1}\right).
\end{equation}
Moreover, from \eqref{eq:ep-approx:dist}, it holds that
\begin{align*}
\Vert x^k - \prox{\gf}{\gamma}{p,  \varepsilon_k}({x}^k)\Vert &\leq \Vert x^k -\prox{\gf}{\gamma}{p}({x}^k)\Vert+\Vert\prox{\gf}{\gamma}{p, \varepsilon_k}({x}^k) - \prox{\gf}{\gamma}{p}({x}^k)\Vert
\\
&\leq\Vert x^k - \prox{\gf}{\gamma}{p}({x}^k)\Vert+\mu \Vert x^k - \prox{\gf}{\gamma}{p,  \varepsilon_k}({x}^k)\Vert.
\end{align*}
Thus, $(1-\mu)\Vert x^k - \prox{\gf}{\gamma}{p,  \varepsilon_k}({x}^k)\Vert\leq \Vert x^k - \prox{\gf}{\gamma}{p}({x}^k)\Vert$. From \eqref{eq:th:linconv:k1}, we obtain
\begin{align}\label{eq:th:conKL:z6}
[\phi'(v_k)]^{-1}=\frac{1}{c(1-\theta)}v_k^{\theta}=\frac{1}{c(1-\theta)}v_k^{\frac{1}{1+\vartheta}}
&\leq \frac{\left(\left(\widehat{c}\widehat{\varrho}\right)^{\frac{\vartheta}{1+\vartheta}}+2^{\frac{\vartheta}{1+\vartheta}}\right)^{\frac{1}{\vartheta}}\ov{\beta}^{\frac{1}{\vartheta}}}{c(1-\theta)} \Vert x^k -\prox{\gf}{\gamma}{p, \varepsilon_k}({x}^k)\Vert^{p-1}
\leq b_2 \left\Vert \nabla\fgam{\gf}{p}{\gamma}(x^k)\right\Vert,
\end{align}
with $b_2:=\frac{\gamma\left(\left(\widehat{c}\widehat{\varrho}\right)^{\frac{\vartheta}{1+\vartheta}}+2^{\frac{\vartheta}{1+\vartheta}}\right)^{\frac{1}{\vartheta}}\ov{\beta}^{\frac{1}{\vartheta}}}{c(1-\theta)(1-\mu)^{p-1}}$.
Combining \eqref{eq:th:conKL:z5} and \eqref{eq:th:conKL:z6}, lead to
\begin{align*}
  \left\Vert \nabla\fgam{\gf}{p}{\gamma}(x^k)\right\Vert^{1+\vartheta} \leq \widehat{\varrho}^{-1} 2^{1+\vartheta}\left(\Delta_k-\Delta_{k+1} \right)\left(\left(1+b_2\right)\left\Vert \nabla\fgam{\gf}{p}{\gamma}(x^k)\right\Vert\right),
\end{align*}
and thus,
\begin{align}\label{eq:th:conKL:z7}
 b_3   \left\Vert \nabla\fgam{\gf}{p}{\gamma}(x^k)\right\Vert^{\vartheta} \leq  \Delta_k-\Delta_{k+1},
\end{align}
with $b_3=\widehat{\varrho}2^{-(1+\vartheta)}\left(1+b_2\right)^{-1}$. 
Since  $(1-\mu)\Vert x^k - \prox{\gf}{\gamma}{p,  \varepsilon_k}({x}^k)\Vert\leq \Vert x^k - \prox{\gf}{\gamma}{p}({x}^k)\Vert$, we have
\[
(1-\mu) \Vert \nabla\fgam{\gf}{p,\varepsilon_i}{\gamma}(x^i)\Vert^{\frac{1}{p-1}}\leq \Vert \nabla\fgam{\gf}{p}{\gamma}(x^i)\Vert^{\frac{1}{p-1}}.
\]
Together with \eqref{eq:th:linconv:d3}, this ensures
\begin{align*}
  B_k=\sum_{i\geq k}\Vert \nabla\fgam{\gf}{p,\varepsilon_i}{\gamma}(x^i)\Vert^{\vartheta} &\leq \frac{1}{(1-\mu)^{\vartheta(p-1)}}\sum_{i\geq k}\Vert \nabla\fgam{\gf}{p}{\gamma}(x^i)\Vert^{\vartheta}
\leq \frac{1}{b_3(1-\mu)^{\vartheta(p-1)}}\sum_{i\geq k} \left(\Delta_i-\Delta_{i+1} \right)
\\
&\overset{(i)}{\leq} \frac{1}{b_3(1-\mu)^{\vartheta(p-1)}}\Delta_k\leq b_4 \Vert \nabla\fgam{\gf}{p,\varepsilon_k}{\gamma}(x^k)\Vert^{\vartheta}\leq  b_4 \left(B_k-B_{k+1}\right),
\end{align*}
where $b_4=\frac{b_1 \gamma^{\vartheta}}{b_3(1-\mu)^{\vartheta(p-1)}}$ and $(i)$  follows from the fact that $\phi$ is continuous and $\fgam{\gf}{p}{\gamma}(x^k) - \fv + v_k\to 0$, which implies $\Delta_k \to 0$.
Thus, we have $B_{k+1}\leq \left(1-\frac{1}{b_4}\right)B_k$, demonstrating the desired asymptotic $Q$-linear convergence of the sequence $\{B_k\}_{k\in \Nz}$. Hence, $\{x^k\}_{k\in \Nz}$ converges R-linearly to a proximal fixed point.
\end{proof}

\bibliographystyle{spbasic}
\bibliography{references}

\end{document}